\documentclass[10pt]{amsart}
\usepackage{geometry}
\usepackage[dvips]{graphicx}
\usepackage{hyperref}

\usepackage{amscd,amsfonts,amssymb,amsmath,amsthm,latexsym}
\usepackage{mathrsfs}
\usepackage{xcolor}
\usepackage{color}
\usepackage{tikz}
\usepackage{tensor}

\usepackage{pdflscape}
\usepackage{afterpage}
\usepackage{capt-of}

\usepackage{geometry}
\usepackage{mathtools}
\usepackage{enumerate}
\usepackage{float}
\usepackage{color}
\usepackage{url}
\usepackage{array}
\usepackage{multirow}
\usepackage{bigstrut}
\usepackage{fullpage}
\usepackage{graphics}
\usepackage[all]{xy}
\xyoption{curve}
\xyoption{import}
\xyoption{arc}
\xyoption{ps}

\numberwithin{equation}{section}

\theoremstyle{plain}
	\newtheorem{thmintro}{Theorem}
	
    \newtheorem{thm}{Theorem}[section]
    \newtheorem{lemma}[thm]{Lemma}
    \newtheorem{coro}[thm]{Corollary}
    \newtheorem{prop}[thm]{Proposition}

\theoremstyle{definition}
    \newtheorem{defi}[thm]{Definition}
    \newtheorem{ex}[thm]{Example}

    \newtheorem{remark}[thm]{Remark}
\theoremstyle{remark}

\newcommand{\Z}{\mathbb{Z}}

\newcommand{\suchthat}{\ | \ }
\newcommand{\lcm}{\operatorname{lcm}}
\newcommand*\circled[1]{\tikz[baseline=(char.base)]{
            \node[shape=circle,draw,inner sep=2pt] (char) {#1};}}
\newcommand{\Gal}{\operatorname{Gal}}
\newcommand{\Fix}{\operatorname{Fix}}

\newcommand{\surfnoM}{\Sigma}
\newcommand{\marked}{\mathbb{M}}
\newcommand{\orb}{\mathbb{O}}

\newcommand{\surf}{(\surfnoM,\marked,\orb)}
\newcommand{\SSigma}{{\mathbf{\Sigma}}}
\newcommand{\SSigmaw}{{\mathbf{\Sigma}_\omega}}
\newcommand{\tauw}[1][\tau]{{#1, \omega}}
\newcommand{\dtauwi}[1][i]{{d(\tauw)_{#1}}}
\newcommand{\dtuple}{\bldD}

\newcommand{\Ctau}{C_\bullet(\tau)}
\newcommand{\CZonetauwFtwo}{Z^1(\tau)}
\newcommand{\CCtauwFtwo}{C^\bullet(\tau)}
\newcommand{\CtauwFtwo}{C_\bullet(\tau)}

\newcommand{\bbC}{\mathbb{C}}
\newcommand{\R}{\mathbb{R}}
\newcommand{\F}{\mathbb{F}}
\newcommand{\usualRA}[1]{R\langle #1\rangle}
\newcommand{\compRA}[1]{R\langle\hspace{-0.075cm}\langle #1\rangle\hspace{-0.075cm}\rangle}

\newcommand{\Hom}{\operatorname{Hom}}

\newcommand{\Aut}{\operatorname{Aut}}
\newcommand{\Image}{\operatorname{Im}}

\newcommand{\ZZ}{\mathbb{Z}}

\newcommand{\maxid}{\mathfrak{m}}

\newcommand{\jacobalg}[1]{\mathcal{P}(#1)}

\newcommand{\Rep}{\operatorname{Rep}}
\newcommand{\rep}{\operatorname{rep}}

\newcommand{\calB}{\mathcal{B}}

\newcommand{\myid}{1\hspace{-0.125cm}1}

\newcommand{\bbS}{\mathbb{S}}
\newcommand{\curlS}{\mathscr{S}}

\newcommand{\bldsigma}{\boldsymbol{\sigma}}
\newcommand{\bldD}{\boldsymbol{d}}
\newcommand{\blddelta}{\boldsymbol{\delta}}
\newcommand{\maxcompRA}[1]{\maxid\langle\hspace{-0.075cm}\langle #1\rangle\hspace{-0.075cm}\rangle}

\newcommand{\compactbigotimesE}[1][0pt]{%
  \mathrel{\raisebox{#1}{$\underset{\scalebox{0.9}{
$\boldsymbol{E}\,$}}{\boldsymbol{\otimes}}$}}%
}
\newcommand{\smallotimesE}{\scalebox{0.6}{$\compactbigotimesE[0.45em]$}}

\newcommand{\compactbigotimesL}[1][0pt]{%
  \mathrel{\raisebox{#1}{$\underset{\scalebox{0.9}{
$\boldsymbol{L}\,$}}{\boldsymbol{\otimes}}$}}%
}
\newcommand{\smallotimesL}{\scalebox{0.6}{$\compactbigotimesL[0.45em]$}}

\newcommand{\compactbigotimesF}[1][0pt]{%
  \mathrel{\raisebox{#1}{$\underset{\scalebox{0.9}{
$\boldsymbol{F}\,$}}{\boldsymbol{\otimes}}$}}%
}
\newcommand{\smallotimesF}{\scalebox{0.6}{$\compactbigotimesF[0.45em]$}}

\newcommand{\compactbigotimesempty}[1][0pt]{%
  \mathrel{\raisebox{#1}{$\underset{}{\,\boldsymbol{\otimes}}\,$}}%
}
\newcommand{\smallotimesempty}{\scalebox{0.6}{$\compactbigotimesempty[0.2em]$}}

\newcommand{\compactminus}[1][[0pt]{%
  \mathrel{\raisebox{#1}{$\boldsymbol{-}$}}%
}
\newcommand{\smallminus}{\scalebox{0.65}{$\compactminus[0.2em]$}}

\newcommand{\compactplus}[1][0pt]{%
  \mathrel{\raisebox{#1}{$\boldsymbol{+}$}}%
}
\newcommand{\smallplus}{\scalebox{0.65}{$\compactplus[0.15em]$}}

\newcommand{\compactQhat}{\skew{3}\widehat{\rule{0ex}{1.35ex}\smash{Q}}}
\newcommand{\compactblddhat}{\skew{3}\widehat{\rule{0ex}{1.35ex}\smash{\mathbf{d}}}}
\newcommand{\compactdhat}{\skew{3}\widehat{\rule{0ex}{1.35ex}\smash{d}}}

\newcommand{\mathinsection}[1]{\texorpdfstring{$#1$}}

\newcommand{\inputtripleforblocks}{Q,\bldD,g}

\newcommand{\inputpairfromtriangulation}{\tau,\xi}

\newcommand{\indexset}{\mathtt{I}}

\begin{document}

\title{Semilinear clannish algebras arising\\ from surfaces with orbifold points}
\author{Raphael Bennett-Tennenhaus}
\address{Raphael Bennett-Tennenhaus\newline
Fakult\"at f\"ur Mathematik, Universit\"at Bielefeld, Germany
		}
\email{raphaelbennetttennenhaus@gmail.com}
\author{Daniel Labardini-Fragoso}
\address{Daniel Labardini-Fragoso\newline
Dipartimento di Matematica ``Tullio Levi-Civita'', Universit\`a degli Studi di Padova, Italy, and\newline
Instituto de Matem\'aticas, UNAM, Mexico.\newline
}
\email{daniel.labardinifragoso@unipd.it}

\begin{abstract} Semilinear clannish algebras have been recently introduced by the first author and Crawley-Boevey as a generalization of Crawley-Boevey's clannish algebras. In the present paper, we associate semilinear clannish algebras to the (colored) triangulations of a surface with marked points and orbifold points, and exhibit a Morita equivalence between these algebras and the Jacobian algebras constructed a few years ago by Geuenich and the second author. 
\end{abstract}

\maketitle

\tableofcontents

\section{Introduction}

\emph{Clannish algebras}, defined by Crawley-Boevey over thirty years ago \cite{CB-clans}, form a class of tame algebras whose indecomposable modules enjoy explicit parameterizations in terms of strings and bands that generalize the familiar parameterizations of indecomposables for gentle algebras. Very recently, Bennett-Tennenhaus--Crawley-Boevey \cite{BTCB}, have introduced \emph{semilinear clannish algebras}, a more general class of algebras where the action of an arrow of the quiver on a representation allows the scalars to ``come out'' up to the application of a field automorphism \emph{a priori} attached to the arrow, instead of requiring them to always ``come out'' linearly. 
The indecomposable modules over a semilinear clannish algebra still enjoy very handy parameterizations in terms of strings and bands. 

On the other hand, with the aim of categorifying the skew-symmetrizable cluster algebras associated by Felikson-Shapiro-Turmakin \cite{FeShTu-orbifolds} to surfaces 
$\Sigma$ equipped 
with 
a set $\mathbb{M}$ of 
marked points and 
a set $\mathbb{O}$ of 
 orbifold points $\boldsymbol{\Sigma}=(\Sigma,\mathbb{M},\mathbb{O})$, Geuenich--Labardini-Fragoso associated in \cite{GLF1,GLF2} a species with potential to each colored triangulation of such a surface $\boldsymbol{\Sigma}$, and showed that if $\boldsymbol{\Sigma}$ is either once-punctured closed or unpunctured, then whenever one is given two colored triangulations related by the flip of an arc, the associated species with potential are related by a mutation of species with potential (also defined in \cite{GLF1}, following the guidelines of Derksen-Weyman-Zelevinsky's mutations of quivers with potential \cite{DWZ1}), thus extending one of the main results from the second author's Ph.D. thesis from simply-laced to non-simply laced situations.

An example was given in \cite[\S5.4]{BTCB} of a 3-vertex semilinear clannish algebra that can be easily seen to be isomorphic to one of the Jacobian algebras discovered in \cite{GLF1} (namely, the Jacobian algebra of the species with potential associated to certain triangulation of a digon with two orbifold points, see Remark \ref{rem:ex-from-BTCB} below).
The main aim of this paper is to prove that, a lot more generally, the Jacobian algebras of all the species with potential from the previous paragraph are Morita-equivalent to semilinear clannish algebras. We do this by explicitly constructing semilinear clannish algebras for the colored triangulations of a surface with marked points and orbifold points, and by exhibiting explicit Morita-equivalences with the Jacobian algebras defined in \cite{GLF1,GLF2}.

Let us describe the contents of the paper in some detail. Section \ref{sec-background} is devoted to recalling some previously existing notions and to setting notation for them. In Subsection \ref{subsec-species} we recall some generalities on tensor rings, species and modulations of weighted quivers, as well as the notion of modulating function from \cite{GLF1}, the notion of semilinear path algebra from \cite{BTCB}, and some explicit computations that can be performed on bimodules over cyclic Galois field extensions whose base field contains certain roots of unity. In Subsection \ref{subsec-representations} we recall the notion of representation of a species, as well as the equivalence between the category of representations of a species and the category of left modules over the tensor ring of the species. Afterwards, we note that for any species arising from a modulating function of a weighted quiver $(Q,\mathbf{d})$ over the field extensions just mentioned, the category of representations of the species is equivalent to the category whose objects are quiver representations of $Q$ that to each vertex attach a vector space over the field corresponding to the vertex, and to each arrow attach a map that is semilinear over the intersection of the fields corresponding to the head and the tail of the arrow, see Lemma \ref{lemma:simplyfing-reps-of-these-species} and  Corollary \ref{coro:rep-of-species-equiv-to-cat-of-semilinear-maps}.

In Subsections \ref{subsec:Jac-algs} and \ref{subsec:semilinear-clannish-algs-def-of-concept} we recall from \cite{GLF1} and \cite{BTCB} the general notions of Jacobian algebra of a species with potential, and of semilinear clannish algebra.
In Section \ref{sec:specific-field-extensions} we describe the specific types of field extensions that will be used to construct the species associated to triangulations: only degree-$1$, degree-$2$ or degree-$4$ cyclic Galois extensions $E/F$ with $F$ having certain fourth roots of unity will be used. 

In Section \ref{sec:building-blocks} we introduce two sets of 3-vertex algebras. The first set is formed by ten 3-vertex Jacobian algebras defined by species with potential, and that we thus call \emph{Jacobian blocks}. The second set is formed by ten 3-vertex semilinear clannish algebras, that we call \emph{semilinear clannish blocks}. In Proposition \ref{prop-morita-equivalences-for-blocks} we prove that for $k=1,\ldots,10$, the $k^{\operatorname{th}}$ Jacobian block and the $k^{\operatorname{th}}$ semilinear clannish block are Morita-equivalent.

In Section \ref{sec:colored-triangs} we recall the combinatorial framework of surfaces with marked points and orbifold points and their triangulations, as well as the notion of colored triangulation defined in \cite{GLF2}. For the latter, one needs to first to introduce certain two-dimensional CW-complex $X(\tau)=(X_0(\tau),X_1(\tau),X_2(\tau))$ for each triangulation $\tau$, and consider the cochain complex $C^{\bullet}(\tau)$ dual to the cellular chain complex of $X(\tau)$ with coefficients in $\mathbb{F}_2\coloneqq \mathbb{Z}/2\mathbb{Z}$. A \emph{colored triangulation} is defined to be a pair $(\tau,\xi)$ consisting of a triangulation $\tau$ and a $1$-cocycle 
$\xi$ 
of $C^{\bullet}(\tau)$.

To associate a Jacobian algebra and a semilinear clannish algebra to a colored triangulation of $\boldsymbol{\Sigma}=(\Sigma,\mathbb{M},\mathbb{O})$, we need to fix one more piece of input, namely, a function $\omega:\orb\rightarrow\{1,4\}$. We refer to any such $\omega$ as a \emph{choice of weights}; there are, thus, $2^{|\orb|}$ distinct choices of weights. Roughly speaking, fixing a choice of weights corresponds to fixing one amongst all the skew-symmetrizable matrices that give rise to a given diagram in \cite[Definition 7.3]{FZ2}. E.g., the diagram $1\overset{2}{\longleftarrow}2\longleftarrow 3\longleftarrow\ldots\longleftarrow n$ arises both from a skew-symmetrizable matrix of type $B$, and from a skew-symmetrizable matrix of type $C$; fixing a choice of weights corresponds to picking one of these two matrices for the depicted diagram.

Given $\boldsymbol{\Sigma}_{\omega}\coloneqq (\Sigma,\mathbb{M},\mathbb{O},\omega)$ and a triangulation $\tau$ of $\boldsymbol{\Sigma}=(\Sigma,\mathbb{M},\mathbb{O})$, in Subsection \ref{subsec:weighted-quiver} we associate to $(\tau,\omega)$ a loop-free \emph{weighted quiver} $(Q(\tau,\omega),\mathbf{d}(\tau,\omega))$, that is, a pair consisting of a loop-free quiver $Q(\tau,\omega)$ and a tuple $\mathbf{d}(\tau,\omega))=(d(\tau,\omega)_{k})_{k\in\tau}$ of positive integers. The quiver $Q(\tau,\omega)$ is a modification of the quiver $\overline{Q}(\tau)$ that defines the 1-skeleton of $X(\tau)$, a quiver that we denote $\overline{Q}(\tau)$. 
As recalled in \S\ref{subsec:surfaces-and-triangs}, \emph{pending} arcs connect marked points with orbifold points. 
The integer $d(\tau,\omega)_{k}$ is set to be $\omega(q_k)\in\{1,4\}$ if $k$ is a pending arc incident to an orbifold point~$q_k$, and $d(\tau,\omega)_{k}\coloneqq 2$ if $k$ is a non-pending arc. Thus, $\lcm\{d(\tau,\omega)_{k}\suchthat k\in\tau\}\in\{1,2,4\}$, and this is why we only need degree-$1$, degree-$2$ and degree-$4$ field extensions $E/F$ in Section \ref{sec:specific-field-extensions}. 

With $(Q(\tau,\omega),\mathbf{d}(\tau,\omega))$ at hand, in Subsection \ref{subsec:algebras-for-arb-weights} we associate a Jacobian algebra (\S\ref{subsubsec:arb-weights-Jac-alg}) and a semilinear clannish algebra (\S\ref{subsubsec:arb-weights-semilin-clan-alg}) to each colored triangulation. 
Both constructions are defined in terms of the degree-$d$ field extension $E/F$ from Section \ref{sec:specific-field-extensions}, where $d\coloneqq \lcm\{d(\tau,\omega)_{k}\suchthat k\in\tau\}$, and the non-trivial element $\theta$ of the Galois group $\Gal(L/F)=\{\myid_{L},\theta\}$ of the unique subfield $L$ of $E$ such that $[L:F]=2$. 

In \S\ref{subsubsec:arb-weights-Jac-alg} we recall from \cite{GLF2} how for each $1$-cocycle $\xi$ of $C^\bullet(\tau)$ one can associate a species with potential $(A(\tau,\xi),W(\tau,\xi))$ to the colored triangulation $(\tau,\xi)$. 
For this we attach to each $k\in\tau$ the unique subfield $F_k$ of $E$ such that $[F_k:F]=d(\tau,\omega)_{k}$, and to each arrow $a:k\rightarrow j$ of $Q(\tau)$ the $F_j$-$F_k$-bimodule $A(\tau,\xi)_a\coloneqq F_j^{g(\tau,\xi)_a}\otimes_{F_j\cap F_k}F_k$, where $g(\tau,\xi)_a\in \Gal(F_j\cap F_k/F)$ is either an extension of $\theta^{\xi_a}:L\rightarrow L$ to $F_j\cap F_k$ (if $F\varsubsetneq F_j\cap F_k$, i.e., $L\subseteq F_j\cap F_k$) or the restriction $\theta^{\xi_a}|_{F} =\myid_F:F\rightarrow F$ (if $F= F_j\cap F_k$). The potential $W(\tau,\xi)$ is defined as a sum of ``obvious'' degree-$3$ cycles in the tensor ring $T_R(A(\tau,\xi))$, where $R\coloneqq \times_{k\in\tau}F_k$ and $A(\tau,\xi)\coloneqq \bigoplus_{a\in Q(\tau,\omega)_1}A(\tau,\xi)_a$. With the notion of \emph{cyclic derivative} from \cite[Definition 3.11]{GLF1}, we form the \emph{Jacobian algebra} $\mathcal{P}(A(\tau,\xi),W(\tau,\xi))$ as the quotient of the complete tensor ring of $A(\tau,\xi)$ over $R$, which we denote $\compRA{A(\tau,\xi)}$ and call \emph{complete path algebra}, modulo the $\mathfrak{m}$-adic topological closure of the two-sided ideal generated by the cyclic derivatives of $W(\tau,\xi)$ with respect to the arrows of $Q(\tau,\omega)$. 
Roughly speaking, for each arrow $a$ and each cycle $c$, the cyclic derivative $\partial_a(c)$ is defined to be the $g(\tau,\xi)_a^{-1}$-linear part of the usual sum of paths obtained by deleting each occurrence of $a$ in $c$ (with the accustomed reordering $yx$ whenever $c=xay$).

In \S\ref{subsubsec:arb-weights-semilin-clan-alg} we associate a semilinear clannish algebra to $(\tau,\xi)$ as follows. 
For an automorphism $\alpha$ of a field $K$ we write $K[x;\alpha]$ for the corresponding \emph{skew polynomial ring}, for which the non-negative powers of $x$ form a $K$-basis, and whose multiplication extends the equation $x\lambda=\alpha(\lambda) x$ linearly over $K$.  
Set $\widehat{Q}(\tau)$ to be the quiver obtained from $\overline{Q}(\tau)$ by adding a loop $s_j$ at each pending arc $j$ of $\tau$. Then we attach  $L$ to every vertex of $\widehat{Q}(\tau)$. To each arrow $a$ of $\widehat{Q}(\tau)$ we attach a field automorphism $\sigma_a\in\Gal(L/F)=\{\myid_L,\theta\}$ by
$$
\sigma_a\coloneqq \begin{cases}
\theta^{\xi_a} & \text{if $a\in\overline{Q}(\tau)_1$};\\
\theta & \text{if $a=s_j$ and $d(\tau,\omega)_j=1$};\\
\myid_L & \text{if $a=s_j$ and $d(\tau,\omega)_j=4$}.
\end{cases}
$$
Letting $L_{\boldsymbol{\sigma}}\widehat{Q}(\tau)$ be the \emph{semilinear path algebra} of $\widehat{Q}(\tau)$ with respect to $\boldsymbol{\sigma}\coloneqq (\sigma_a)_{a\in \widehat{Q}(\tau)_1}$, 
for each pending arc $j$ of $\tau$ we define a degree-$2$ polynomial $q_{s_j}\in L[s_j;\sigma_{s_j}]\subseteq L_{\boldsymbol{\sigma}}\widehat{Q}(\tau)$ by
$$
q_{s_j}\coloneqq \begin{cases}
s_j^2-e_j & \text{if $d(\tau,\omega)_j=1$};\\
s_j^2-ue_j & \text{if $d(\tau,\omega)_j=4$};
\end{cases}
$$
where $u$ is an \emph{a priori} given element of $L\setminus F$ such that $\theta(u)=-u$, and $e_j$ is the $j^{\operatorname{th}}$ primitive idempotent of the semisimple ring $S\coloneqq \times_{k\in \tau}L$. Furthermore, we set $Z(\tau,\xi)\subseteq L_{\boldsymbol{\sigma}}\widehat{Q}(\tau)$ to be the set of all paths of length two on $\overline{Q}(\tau)$ both of whose constituent arrows are contained in the same triangle of $\tau$. The quotient $L_{\boldsymbol{\sigma}}\widehat{Q}(\tau)/\langle Z(\tau,\xi)\cup \{q_{s_j}\suchthat j$ is a pending arc of $\tau\}\rangle$ turns out to be a semilinear clannish algebra, see Proposition~\ref{prop-algebras-are-semilinear-clannish}. It is an $F$-algebra, but not necessarily an $L$-algebra, as the action of $L$ on it is typically not central.

\textbf{Remark.}
If $\mathbb{O}\neq\varnothing$, then the semilinear clannish algebras we obtain are not
semilinear gentle, because of the presence of special loops attached to the orbifold points.
If $\mathbb{O}=\varnothing$, which we allow in Section 6.2, then the algebras constructed in
\S6.2.1 and \S6.2.2 are semilinear gentle (infinite-dimensional if $\surf$ is moreover
once-punctured closed). Finally, if $\mathbb{O}=\varnothing$, $\xi$ is the zero cocycle and,
moreover, $\marked\subseteq\partial\Sigma$, then the algebras obtained are precisely the
gentle algebras studied by Assem-Br\"{u}stle-Charbonneau-Plamondon in~\cite{ABCP}, introduced earlier in \cite{LF-QPsurfs1}.

Special attention deserve the constant choices of weights $\omega:\orb\rightarrow\{1,4\}$, i.e., the constant functions $\omega\equiv 1$ and $\omega\equiv 4$. For $\omega\equiv 1$ we have $d\coloneqq \lcm\{d(\tau,\omega)_{k}\suchthat k\in\tau\}=2$, and the settings from Subsection \ref{subsec:algebras-for-arb-weights} allow us to take $E/F$ to be $\mathbb{C}/\mathbb{R}$ (the field thus attached in \S\ref{subsubsec:arb-weights-Jac-alg} to the non-pending arcs is $\mathbb{C}$, the one attached to the pending arcs is $\mathbb{R}$, whereas the field attached to all arcs in \S\ref{subsubsec:arb-weights-semilin-clan-alg} is then $\mathbb{C}$). 

In contrast, for $\omega\equiv 4$ we have $d\coloneqq \lcm\{d(\tau,\omega)_{k}\suchthat k\in\tau\}=4$, and strictly speaking, the settings from Subsection \ref{subsec:algebras-for-arb-weights} do not allow us to take $E/F$ or $E/L$ to be $\mathbb{C}/\mathbb{R}$ (since $[E:F]=d=4$ and $L$ has degree $2$ over its subfield $F$); however, a careful reading of \S\ref{subsubsec:arb-weights-Jac-alg} and \S\ref{subsubsec:arb-weights-semilin-clan-alg} suggests that, by taking $\xi$ to be the zero $1$-cocycle, one should be able to work over $\mathbb{C}/\mathbb{R}$ and still obtain all the definitions and results from Subsection~\ref{subsec:algebras-for-arb-weights}. We do this in detail in \S\ref{subsubsec:constant-weights-Jac-alg} and \S\ref{subsubsec:constant-weights-semilin-clan-alg}. For $\orb\neq \varnothing$ and $\omega\equiv 4$, one is thus allowed in \S\ref{subsubsec:constant-weights-Jac-alg} to attach the field $\mathbb{R}$ to the non-pending arcs, and the field $\mathbb{C}$ to the pending arcs, whereas the field attached to all arcs in \S\ref{subsubsec:constant-weights-semilin-clan-alg} is, concordantly, $\mathbb{R}$.

In Section \ref{sec-morita-equivalence-between-algebras-from-triangulations} we state and prove our main result, namely:
\begin{thmintro}
\label{main-theorem}
Let $\SSigma\coloneqq (\Sigma,\marked,\orb)$ be a surface with marked points and orbifold points, $\omega:\orb\rightarrow\{1,4\}$ a function, and $(\tau,\xi)$ a colored triangulation of $\SSigma$. If
\begin{itemize}
\item $\partial\Sigma=\varnothing$ and $|\marked|=1$, or
\item $\partial\Sigma\neq\varnothing$ and $\marked\subseteq\partial\Sigma$,
\end{itemize}
then the Jacobian algebra of the species with potential associated to $(\tau,\xi)$ in \S\ref{subsubsec:arb-weights-Jac-alg} (resp. \S\ref{subsubsec:constant-weights-Jac-alg}) is Morita-equivalent to the semilinear clannish algebra associated to $(\tau,\xi)$ in \S\ref{subsubsec:arb-weights-semilin-clan-alg} (resp. \S\ref{subsubsec:constant-weights-semilin-clan-alg}). The Morita-equivalence is $F$-linear and restricts to an equivalence between the categories of finite-dimensional left modules. 
\end{thmintro}

The species with potential associated to $(\tau,\xi)$ in \S\ref{subsubsec:arb-weights-Jac-alg} (resp.~\S\ref{subsubsec:constant-weights-Jac-alg}) was first constructed in \cite{GLF2} (resp.~\cite{GLF1}). Theorem \ref{main-theorem} is stated more precisely as Theorem \ref{thm:main-result-of-paper} below. To prove it we recall from \cite{FeShTu-orbifolds,Fomin-Shapiro-Thurston} that $\tau$ can be obtained by gluing finitely many \emph{puzzle pieces}, and notice a few facts, namely,
\begin{itemize}
\item that the Jacobian blocks and the semilinear clannish blocks defined respectively in Subsections \ref{subsec:Jacobian-blocks} and \ref{subsec:semilinear-clannish-blocks} are precisely the Jacobian algebras and the semilinear clannish algebras that would arise from the puzzle pieces according to the constructions of Section \ref{sec:Jac-algs-and-semilinear-clan-algs-of-colored-triangs} (allowing boundary segments to be vertices of the quivers or, alternatively, adding to them some artificial boundary triangles);
\item that one can glue not only the puzzle pieces, but also their associated Jacobian blocks (resp. semilinear clannish blocks), following a procedure first defined by Br\"ustle in \cite{Brustle-kit}, and that the result of this gluing of blocks is precisely the Jacobian algebra (resp. the semilinear clannish algebra) associated to $(\tau,\xi)$; alternatively, applying the notion of $\rho$-block decomposition from \cite{GLFS-schemes}, the $\rho$-blocks of the Jacobian algebra (resp. the semilinear clannish algebra) associated to $(\tau,\xi)$ are precisely the Jacobian blocks (resp the semilinear clannish blocks) given by the puzzle-piece decomposition of $(\tau,\xi)$; 
\item that the explicit Morita equivalences given in the proof of Proposition \ref{prop-isomorphisms-between-some-of-the-blocks} between Jacobian blocks and semilinear clannish blocks, can be glued as well to produce an explicit Morita equivalence between the Jacobian algebra and the semilinear clannish algebra associated to $(\tau,\xi)$.
\end{itemize}

Section \ref{sec:indecs-for-blocks} is devoted to recalling the parameterizations of the indecomposable modules over a semilinear clannish algebra in terms of symmetric and asymmetric strings and bands given in \cite{BTCB}, to providing a full list of strings and bands for the semilinear clannish blocks from Subsection \ref{subsec:semilinear-clannish-blocks}, and to illustrating the construction of symmetric-string modules by means of an explicit example, and the representation of the Jacobian algebra corresponding to it under Morita equivalence.

Finally, in Section \ref{sec:on-need-of-cocycles} we explain why we need to work with $1$-cocycles $\xi$ as part of the input $(\tau,\xi)$ to which we associate algebras. Roughly and informally speaking, only the species arising from modulating functions satisfying a cocycle conditions have the chance of admitting a non-degenerate potential for the notion of mutation of species with potential from \cite{GLF1}. In contrast, one does not need any cocycle condition to be satisfied in order to define a semilinear clannish algebra. This means that the constructions from Subsections \ref{subsec:semilinear-clannish-blocks}, \ref{subsubsec:arb-weights-semilin-clan-alg} and \ref{subsubsec:constant-weights-semilin-clan-alg} can be carried out, without requiring $\xi$ to be a cocycle, to still obtain semilinear clannish algebras in the end.

\section{Algebraic background}
\label{sec-background}

\subsection{Tensor rings, species and modulations}
\label{subsec-species}

\,

In this subsection we recall some generalities on tensor rings, species and modulations of weighted quivers. We will start with very general concepts and settings, and gradually decrease the level of generality. The intention of this approach is to make as transparent as possible how the concrete Jacobian and semilinear clannish algebras we will later introduce fit into classical well-known general constructions of rings that are not necessarily algebras over an algebraically closed field. Our main references for \S2.1 are \cite[\S 7.1]{Gabriel1973}, \cite[\S 10]{Dlab-Ringel-on-algebras-finite},\cite[\S 1B]{Green-representation-tensor-algebra}, \cite[\S 2]{Ringel-species} and \cite[\S 2]{Simson}, \cite[\S2]{BTCB}, \cite[\S2 and \S3]{GLF1}, \cite[\S2]{Berg} and \cite[\S2]{GeuenichPHD}.

\vspace{1mm}

\subsubsection{\textbf{Tensor rings}} The \emph{tensor ring} $\usualRA{A}$ of an $R$-$R$-bimodule $A$ over a ring $R$ is
\[
\begin{array}{ccc}
\usualRA{A}\coloneqq \bigoplus_{n\geq0} A^{\otimes_R n}, & A^{\otimes_R 0}\coloneqq R, & A^{\otimes_R n}\coloneqq \underbrace{A\otimes_{R}\dots\otimes_{R}A}_{n}\quad(n>0),
\end{array}
\]
with  multiplication given by the natural $R$-balanced maps $A^{\otimes_R n}\times A^{\otimes_R m}\to A^{\otimes_R (n+m)}$. 
Define the \emph{complete tensor ring} $\compRA{A}$, 
 and the (two-sided) \emph{arrow ideal} $\maxcompRA{A}\triangleleft \compRA{A}$, by setting
\[
\begin{array}{cc}
\compRA{A}\coloneqq \prod_{n\geq 0} A^{\otimes_R n}=\varprojlim_{\,l>0}
\left(\usualRA{A}/\bigoplus_{n\geq l}A^{\otimes_R n}\right), 
&
\maxcompRA{A}\coloneqq \prod_{n\geq 1}A^{\otimes_R n}.
\end{array}
\]
Both $\usualRA{A}$ and $\compRA{A}$ are $R$-\emph{rings}, meaning each occurs as the codomain of a ring homomorphism from $R$. The image of $R$ under such ring homomorphism is often not contained in the center of $\usualRA{A}$ (or $\compRA{A}$), even when $R$ is commutative.
A two-sided ideal $I\triangleleft\usualRA{A}$ is said to be  \emph{bounded} if  $\bigoplus_{n\geq l}A^{\otimes n}\subseteq I$ for some $l\gg 0$, in which case 
\[
\begin{array}{cc}
\usualRA{A}/I\cong \compRA{A}/ J, & 
J=\bigcap_{n>0}(I+\maxcompRA{A}^{n}),
\end{array}
\]
that is, $J$ is the closure of $I$ in the $\maxcompRA{A}$-adic topology. 
See \cite[Definition 3.6]{GLF1}.

\vspace{1mm}

\subsubsection{\textbf{Species}}\label{subsubsec:species-def-of-concept}
A \emph{species} is a pair $\curlS=((D_{i})_{i\in \indexset},(A_{ij})_{(i,j)\in \indexset\times\indexset})$ subject to  the following conditions:
\begin{itemize}
 \item $D_{i}$ is a division ring attached to each index $i$ in the given finite set $\indexset$.
    \item $A_{ij}$ is a  $D_{i}$-$D_{j}$-bimodule for each $(i,j)\in \indexset\times\indexset$;
    \item $\Hom_{D_{i}\text{-}\mathbf{Mod}}(A_{ij},D_{i})\cong \Hom_{\mathbf{Mod}\text{-}D_{j}}(A_{ij},D_{j})$ as $D_{j}$-$D_{i}$-bimodules for each $(i,j)\in \indexset\times\indexset$.
\end{itemize}

Every species $((D_{i})_{i\in \indexset},(A_{ij})_{(i,j)\in \indexset\times\indexset})$ gives rise to a semisimple ring $R\coloneqq \times_{i\in\indexset}D_i$ and an $R$-$R$-bimodule $A\coloneqq \bigoplus_{(i,j)\in\indexset\times\indexset}A_{ij}$. For $j\in \indexset$ we denote $e_j\coloneqq (\delta_{i,j})_{i\in\indexset}\in R$, where $\delta_{i,j}\in D_i$ is the Kronecker delta between $i$ and $j$, and call $e_j$ the \emph{trivial path at $j$}. Thus, $e_j^2=e_j$ and $1_R=\sum_{j\in \indexset}e_j$.

For a field $F$ we say that $\curlS$ as above is an $F$-\emph{species} provided that for every $(i,j)\in \indexset\times\indexset$, $F$ acts centrally on $D_{i}$ and $A_{ij}$, turning them unambiguously into $F$-vector spaces, and provided these vector spaces are finite-dimensional. 

\vspace{1mm}

\subsubsection{\textbf{Weighted quivers and modulations.}}\label{subsubsec:weighted-quivers-and-modulations}
We write quivers as quadruples $Q=(Q_{0},Q_{1},h,t)$ where $Q_{0}$ is the set of vertices, $Q_{1}$ is the set of arrows and $h$ (respectively, $t$) denotes the function  $Q_{1}\to Q_{0}$ assigning to each arrow $a$ its head $h(a)$ (respectively, tail $t(a)$). 
A \emph{weighted quiver} is a pair $(Q,\mathbf{d})$ consisting of a finite quiver $Q=(Q_0,Q_1,t,h)$ and a $Q_0$-tuple $\mathbf{d}=(d_i)_{i\in Q_0}$ of positive integers. Each  $d_{i}$ is referred to as the \emph{weight attached to} $i$. See \cite[Definition 2.2]{LZ}. 

\vspace{1mm}

For a field $F$  an $F$-\emph{modulation} of $(Q,\bldD)$ is a pair $((D_{i})_{i\in Q_0},(A_{a})_{a\in Q_1})$ such that: 
\begin{itemize}
    \item $D_{i}$ is a finite-dimensional division algebra over $F$ (in particular, $F$ is contained in the center of $D_i$), with $\dim_{F}(D_{i})=d_{i}$, for each $i\in Q_{0}$;
    \item $A_{a}$ is a $D_{h(a)}$-$D_{t(a)}$-bimodule for each $a\in Q_{1}$, finitely generated both on the left over $D_{h(a)}$ and on the right over $D_{t(a)}$;
    \item for each $a\in Q_1$, the action of $F$ on $A_a$, arising from the $D_{h(a)}$-$D_{t(a)}$-bimodule structure of $A_a$, is central.
\end{itemize}
Any $F$-modulation gives rise to an $F$-species $\curlS=((D_{i})_{i\in Q_0},(A_{ij})_{(i,j)\in Q_0\times Q_0})$ by setting $A_{ij}\coloneqq \bigoplus_{a:j\to i} A_{a}$ where the sum runs over all arrows $a\in Q_1$ that go from $j$ to $i$.

By the Krull--Remak--Schmidt property of finite-length modules, any $F$-species $\curlS=((D_{i})_{i\in \indexset},(A_{ij})_{(i,j)\in \indexset\times\indexset})$ arises from an $F$-modulation of a weighted quiver $(Q,\bldD)$ by setting $Q_{0}=\indexset$, $d_{i}\coloneqq \dim_{F}(R_{i})$, by stipulating that the number of arrows $j$ to $i$ is  the number of indecomposable $D_{j}$-$D_{i}$-bimodule summands of $A_{ij}$. The pair $((D_{i})_{i\in Q_0},(A_{a})_{a\in Q_1})$ is then an $F$-modulation of $(Q,\bldD)$ giving rise to $\curlS$.
See \cite[Remarks 2.4.5 and 2.4.6, Corollary 2.3.11]{GeuenichPHD}.

\begin{remark}
We will thus use the term \emph{$F$-species} to refer indistinctly to either an $F$-species or the $F$-modulation it arises from. We shall use the notations $\curlS=((D_{i})_{i\in Q_0},(A_{ij})_{(i,j)\in Q_0\times Q_0})$ and $\curlS=((D_{i})_{i\in Q_0},(A_{a})_{a\in Q_1})$ indistinctly as well to denote an $F$-species.
\end{remark}

\vspace{1mm}

\subsubsection{\textbf{Modulating functions}}\label{subsubsec:modulating-functions-def-of-concept} Let $(Q,\bldD)$ be a weighted quiver, $d\coloneqq \lcm\{d_{i}\suchthat i\in Q_0\}$, and $E/F$ a degree-$d$ cyclic Galois field extension. 
For $i\in Q_0$ define $F_i$ to be the unique degree-$d_{i}$ extension of $F$ contained in~$E$.

Following \cite[Definition 3.2]{GLF1}, we define a \emph{modulating function} to be a collection $$g=(g_a)_{a\in Q_1}\in\times_{a\in Q_1}\Gal(F_{h(a)}\cap F_{t(a)}/F)$$ of field automorphisms $g_a\in \Gal(F_{h(a)}\cap F_{t(a)}/F)$. 

For each $a\in Q_1$ we denote by $F_{h(a)}^{g_a}$ the $F_{h(a)}$-$(F_{h(a)}\cap F_{t(a)})$-bimodule defined as follows (cf. \cite[Section 2]{GLF1}):
\begin{enumerate}
    \item as an additive group, $F_{h(a)}^{g_a}\coloneqq F_{h(a)}$;
    \item for $w\in F_{h(a)}$, $z\in F_{h(a)}\cap F_{t(a)}$ and $m\in F_{h(a)}^{g_a}$, the left action of $w$ on $m$ and the right action of $z$ on $m$ are defined by the rules
    \begin{equation}\label{eq:def-of-twisted-bimodule-actions}
    w\star m \coloneqq  wm \qquad m\star z\coloneqq mg_a(z),
    \end{equation}
    where the products on the right hand sides of the equalities in \eqref{eq:def-of-twisted-bimodule-actions} are taken according to the multiplication that $F_{h(a)}$ has as a field.
\end{enumerate}

Each modulating function $g$ gives rise to an $F$-modulation $((F_i)_{i\in Q_0},(A_{a}(\inputtripleforblocks))_{a\in Q_1})$, where 
 \[
A_{a}(\inputtripleforblocks)\coloneqq F_{h(a)}^{g_a}\otimes_{F_{h(a)}\cap F_{t(a)}} F_{t(a)}\quad \text{for} \ a\in Q_{1}.
\]

Set $R\coloneqq \times_{i\in Q_0}F_i$ and $A(\inputtripleforblocks)\coloneqq \bigoplus_{a\in Q_1}A_{a}(\inputtripleforblocks)$. In \cite[Definition 3.5]{GLF1} and \cite{GLF2}, the tensor ring $\usualRA{A(\inputtripleforblocks)}$ (resp. the complete tensor ring $\compRA{A(\inputtripleforblocks)}$) is called the \emph{path algebra} (resp. \emph{complete path algebra}) of $(Q,\mathbf{d},g)$.

\vspace{1mm}

\subsubsection{\textbf{Semilinear path algebras}}\label{subsubsec:semilinear-path-algs-def-of-concept} 

Let $(\compactQhat,\compactblddhat)$ be a weighted quiver all of whose weights are the same positive integer $\compactdhat$, i.e., such that $\compactblddhat=(\compactdhat)_{i\in \compactQhat_0}$. The reason for this notation is that, later, $(\compactQhat,\compactblddhat)$ will be defined in terms of a given loop-free weighted quiver $(Q,\bldD)$ by adding some loops to $Q$ and setting $\compactdhat$ to be one of the integers $d_i$ appearing in $\bldD$.

Let $K/F$ be a degree-$\compactdhat$ cyclic Galois field extension, 
and let $\bldsigma\colon \compactQhat_{1}\to \Gal(K/F)$ be a function assigning a ring automorphism $\sigma_{b}\in\Gal(K/F)$ to each $b\in \compactQhat_{1}$, i.e., a modulating function.

Following \cite[\S2.1]{BTCB} and the notation therein, we define the \emph{semilinear path algebra} $K_{\bldsigma}\compactQhat$ to be the tensor ring $S\langle A(\widehat{Q},\compactblddhat,\bldsigma)\rangle$, where 
the semisimple ring $S$ and the $S$-$S$-bimodule $A(\compactQhat,\compactblddhat,\bldsigma)$ are defined as (cf. \cite[\S2.1]{BTCB})
\[
\begin{array}{cc}
  S\coloneqq \times_{i\in \widehat{Q}_{0}}K,   
  &  
  A(\compactQhat,\compactblddhat,\bldsigma)\coloneqq \bigoplus_{b\in\widehat{Q}_1}{}_{\pi_{h(b)}} K_{\sigma_b \pi_{t(b)}}.
\end{array}
\]
where $\pi_{j}\colon S\to K$ is the restriction sending $(\lambda_{i})$ to $\lambda_{j}$, and where ${}_{\pi_{h(b)}} K_{\sigma_b \pi_{t(b)}}$ is the set $K$ whose $S$-$S$-bimodule action is indicated by the subscripts, namely, by means of the equations
\[
\begin{array}{ccc}
(\lambda_{i})\cdot \mu=\lambda_{h(b)}  \mu,     & 
\mu\cdot ( \lambda_{i})=\mu \sigma_{b}(\lambda_{t(b)}), & (( \lambda_{i}\colon i\in \compactQhat_{0})\in S,\quad \mu\in {}_{\pi_{h(b)}} K_{\sigma_b \pi_{t(b)}},\quad b\in\compactQhat_{1}).
\end{array}
\]

\begin{remark}
The concept of semilinear path algebra just defined forms a particular instance of the concept of path algebra defined in \S\ref{subsubsec:modulating-functions-def-of-concept}. This is not the case for the more general concept of semilinear path algebra that Bennett-Tennenhaus--Crawley-Boevey work with in \cite{BTCB}: they allow $K$ to be a division ring, and do not require it to be finite-dimensional over some particular field.
\end{remark}

\vspace{1mm}

\subsubsection{\textbf{Eigenbases of cyclic Galois extensions and bimodules}}\label{subsubsec:eigenbases}

Let $d$ be a positive integer, $F$ a field containing a primitive $d^{\operatorname{th}}$ root of unity $\zeta\in F$, and $E/F$ a degree-$d$ cyclic Galois extension with Galois group $\Gal(E/F)=\langle\rho\rangle$. Then there exists $v\in E$ which is an eigenvector of $\rho$ with eigenvalue $\zeta$, i.e.,
$$
\rho(v)=\zeta v.
$$
For $n,m\in\mathbb{Z}$ we then have
$$
\rho^n(v^m)=(\rho^n(v))^m=(\zeta^nv)^m=\zeta^{nm}v^m.
$$
It follows that the set
$$
\calB_{E/F}\coloneqq \{1,v,v^2,\ldots,v^{d-1}\}
$$
is an \emph{eigenbasis of $E/F$}, that is, an $F$-vector space basis of $E$ consisting of eigenvectors of all the elements of $\Gal(E/F)$.

For each positive divisor $d_i$ of $d$, let $F_i$ be the unique subfield of $E$ containing $F$ and such that $[F_i:F]=d_i$. If $d_i$ and $d_j$ are positive divisors of $d$, then $F_j/F_i\cap F_j$ is a degree-$\frac{d_j}{\gcd(d_i,d_j)}$ cyclic Galois extension with Galois group $\Gal(F_j/F_i\cap F_j)=\{\myid_{F_j},\rho|_{F_j}, \rho^2|_{F_j},\ldots,\rho^{\frac{d_j}{\gcd(d_i,d_j)}-1}|_{F_j}\}$, and
\begin{equation}\label{eq:eigenbasis-intermediate-subextension}
\calB^i_{i,j}\coloneqq \{1,v^{\frac{d}{d_j}},(v^{\frac{d}{d_j}})^{2},\ldots,(v^{\frac{d}{d_j}})^{\frac{d_j}{\gcd(d_i,d_j)}-1}\}
\end{equation}
is an eigenbasis of $F_j/F_i\cap F_j$. Notice that $\calB^i_{i,j}$, and hence also any eigenbasis of $F_j/F_i\cap F_j$ has the property that for any two elements $\omega_1,\omega_2\in\calB^i_{i,j}$, the product $\omega_1\omega_2$ is an $F$-multiple of some element of $\calB^i_{i,j}$.

\begin{remark}
    Thanks to the facts that all the intermediate field subextensions of $E/F$ admit eigenbases, that such eigenbases can be computed explicitly, and that each of them is closed under multiplication up to $F$-multiples, many definitions and computations (e.g., cyclic derivatives of potentials) can be done very explicitly when working with the path algebras and the semilinear path algebras from \S\ref{subsubsec:modulating-functions-def-of-concept} and \S\ref{subsubsec:semilinear-path-algs-def-of-concept}.
\end{remark}

The following result from \cite{GLF1} lies behind the definition of \emph{cyclic derivative} for the species we will~work with, see Definition \ref{def:paths-potentials-cyclic-derivatives-jacobian-algebras}-\eqref{item:def-cyclic-derivative} below.

\begin{prop}\cite[Proposition 2.15]{GLF1}\label{prop:properties-of-Fi-Fj-bimodules} For $F_i$ and $F_j$ as above, let $\mathcal{C}_{F_i,F_j}$ be the category of those $F_i$-$F_j$-bimodules on which $F$ acts centrally and whose dimension over $F$ is finite.
\begin{enumerate}
\item the bimodules $F_i^{\rho}\otimes_{F_i\cap F_j}F_j$, with $\rho$ running in $\Gal(F_i\cap F_j/F)$, form a complete set of pairwise non-isomorphic simple objects in $\mathcal{C}_{F_i,F_j}$;
\item for every $M\in\mathcal{C}_{F_i,F_j}$ and every $\rho\in\Gal(F_i\cap F_j/F)$, the function $\pi_\rho=\pi^M_\rho:M\rightarrow M$ defined by
    $$
    m\mapsto\frac{1}{[F_i\cap F_j:F]}\sum_{\omega\in\calB_{F_i\cap F_j/F}}\rho(\omega^{-1})m\omega
    $$
    is an idempotent $F_i$-$F_j$-bimodule homomorphism;
\item for every $M\in\mathcal{C}_{F_i,F_j}$ and every pair $\rho_1,\rho_2\in\Gal(F_i\cap F_j/F)$, if $\rho_1\neq\rho_2$, then $\pi_{\rho_1}\pi_{\rho_2}=0$;
\item for every $M\in\mathcal{C}_{F_i,F_j}$ we have an internal direct sum decomposition
$$
M=\bigoplus_{\rho\in\Gal(F_i\cap F_j/F)}\Image\pi_\rho;
$$
\item for every $M\in\mathcal{C}_{F_i,F_j}$, every $m\in\Image\pi_\rho$ and every $x\in F_i\cap F_j$ we have $mx=\rho(x)m$;
\item for every $M\in\mathcal{C}_{F_i,F_j}$ and every $m\in\Image\pi_\rho$ there exists a unique $F_i$-$F_j$-bimodule homomorphism $\varphi:F_i^{\rho}\otimes_{F_i\cap F_j}F_j\rightarrow M$ such that $\varphi(1\otimes 1)=m$.
\end{enumerate}
\end{prop}

\vspace{1mm}

\subsection{Representations of species}
\label{subsec-representations}

\,

\subsubsection{\textbf{Representations and modules}}\label{subsubsec-representations-and-modules}
Let $F$ be a field and $\curlS=((D_{i})_{i\in Q_0},(A_{a})_{a\in Q_1})$ an $F$-species of a weighted quiver  $(Q,\bldD)$.  
Write $\Rep (\curlS)$ for the category of $\curlS$-\emph{representations}, with objects and morphisms defined as follows.

An $\curlS$-\emph{representation} refers to a collection $M=((M_{i})_{i\in Q_0},(M_{a})_{a\in Q_1})$ subject to the items below.
\begin{itemize}
    \item $M_{i}$ is a left $D_{i}$-module for each $i\in Q_{0}$. 
    \item $M_{a}\colon A_{a}\otimes_{D_{t(a)}} M_{t(a)}\to M_{h(a)}$ is a homomorphism of left  $D_{h(a)}$-modules for each $a\in Q_{1}$.
\end{itemize}
If each $M_{i}$ is of finite rank over $D_{i}$, $M$ is said to be \emph{finite}-\emph{dimensional}. 
A \emph{morphism} of $\curlS$-representations $f:M=((M_{i})_{i\in Q_0},(M_{a})_{a\in Q_1})\to N=((N_{i})_{i\in Q_0},(N_{a})_{a\in Q_1})$ is a $Q_0$-tuple $f=(f_{i})_{i\in Q_0}$ satisfying: 
\begin{itemize}
\item $f_{i}\colon M_{i}\to N_{i}$ is a homomorphism of left $D_{i}$-modules for each $i\in Q_0$; 
\item for each $a\in Q_1$ the diagram of $D_{h(a)}$-module homomorphisms 
$$
\xymatrix{
A_{a}\otimes_{D_{t(a)}} M_{t(a)} \ar[r]^{\quad M_a} \ar[d]_{\myid_{A_{a}}\otimes  f_{t(a)}} & M_{h(a)} \ar[d]^{f_{h(a)}}\\
A_{a}\otimes_{D_{t(a)}} N_{t(a)} \ar[r]^{\quad N_a}  & N_{h(a)}
}
$$
commutes.
\end{itemize}

Let $F$ be a field and $\curlS=((D_{i})_{i\in Q_0},(A_{a})_{a\in Q_1})$ an $F$-modulation of a weighted quiver  $(Q,\bldD)$.
For $n\geq 0$ and a path $p$ in $Q$  of length $n$, we 
define an $R_{h(p)}$-$R_{t(p)}$-bimodule  $A_{p}$ as follows. 
For $n=0$, and hence $p$ a trivial path, let $A_{p}=D_{i}$ where $h(p)=i=t(p)$. For $n>0$, say where $p=a_{n}\dots a_{1}$ with $a_{i}\in Q_{1}$, we let $A_{p}=A_{a}$ if $n=1$ and $a_{1}=a$, and for $n>1$ we let
\[
A_{p}=
A_{a_{n}}
\otimes_{D_{t(a_n)}}
A_{a_{n-1}}
\otimes_{D_{t(a_{n-1})}}
\,
\dots \,
\otimes_{D_{t(a_{3})}}
A_{a_{2}}
\otimes_{D_{t(a_2)}}
A_{a_{1}}.
\]
Using this notation, for $i,j\in Q_{0}$ we have the $D_{i}$-$D_{j}$-bimodule
\[
\begin{array}{c}
e_{i}\usualRA{A}e_{j}=\bigoplus_{\text{paths } p\text{ in }Q\,:\,h(p)=i,\,t(p)=j}A_{p}.
\end{array}
\]
 
 For a representation $M=((M_{i})_{i\in Q_0},(M_{a})_{a\in Q_1})$ of $\curlS$ and a path $p$ in $Q$ as above, the morphisms $M_{a}$ can be combined to give a $D_h(p)$-module homomorphism $M_{p}\colon A_{p}\otimes_{D_{t(p)}}M_{t(p)}\to M_{h(p)}$ defined as follows. 
If $p$ is the trivial path (with $n=0$) at $i\in Q_{0}$ then take  $M_{p}$ as the isomorphism $D_{i}\otimes_{D_i} M_{i}\to M_{i}$. 
If instead $n>0$ and $p=a_{n}\dots a_{1}$ we take
\[
\begin{array}{c}
M_{p}=
M_{a_{n}}
\circ
(\myid_{A_{a_{n}}}\otimes M_{a_{n-1}})
\circ (\myid_{A_{a_{n}}}\otimes \myid_{A_{a_{n-1}}}\otimes M_{a_{n-2}})\circ\dots \circ
(\myid_{A_{a_{n}}}\otimes\dots \otimes \myid_{A_{a_{2}}}\otimes M_{a_{1}}).
\end{array}
\]
For $i,j\in Q_{0}$ we define the map 
$M_{ij}\colon e_{i}\usualRA{A}e_{j}\otimes_{D_{j}}M_{j}\to M_{i}$ by assembling the maps $M_p$
through the universal property of the coproduct of $D_i$-modules.

By a \emph{relation with head $i\in Q_0$ and tail $j\in Q_0$ in the $F$-species $\curlS$} we mean an element $\sigma$ of $e_{i}\usualRA{A}e_{j}$. A representation $M$ is said to be \emph{annihilated by $\sigma$} provided $M_{ij}\circ(\iota \otimes_{D_{j}}\myid_{M_{j}} )=0$ where $\iota$ is the inclusion $R_{i}\sigma R_{j}\subseteq e_{i}\usualRA{A}e_{j}$ of $D_{i}$-$D_{j}$-bimodules.

\vspace{2mm}

It was shown by Dlab and Ringel \cite[Proposition 10.1]{Dlab-Ringel-on-algebras-finite} that $\Rep (\curlS)$ is equivalent to the category $\usualRA{\curlS}$-$\mathbf{Mod}$ of left modules over the tensor ring $\usualRA{A}$; see also \cite[Theorem A]{Green-representation-tensor-algebra} and \cite[Proposition~2.1]{Berg}. 
Namely, there are functors
\[
\begin{array}{ccc}
   \Omega \colon \usualRA{A}\text{-}\mathbf{Mod}\to  \Rep (\curlS), &&   \Gamma\colon \Rep (\curlS)\to \usualRA{A}\text{-}\mathbf{Mod}, 
\end{array}
\]
which are mutually quasi-inverse. 
\vspace{1mm}

Let $I=\langle \rho \rangle$ be a two-sided ideal in the tensor ring $\usualRA{\curlS}$ generated by a set $\rho$ of relations in $\curlS$. 
A representation $M=((M_{i})_{i\in Q_0},(M_{a})_{a\in Q_1})$ of $\curlS$ is called \emph{finite}-\emph{dimensional over} $F$ provided each $M_{i}$ is finite-dimensional over $F$. 
Consider the following subcategories:
\begin{itemize}
    \item $\usualRA{A}/I\text{-}\mathbf{Mod}$, the full subcategory of $\usualRA{A}\text{-}\mathbf{Mod}$ whose modules are annihilated by  $I=\langle\rho\rangle$.  
    \item $\Rep (\curlS, \rho)$, the full subcategory of $\Rep (\curlS)$ whose $\curlS$-representations are annihilated by every $\sigma\in\rho$. 
    \item $\usualRA{A}/I\text{-}\mathbf{mod}$, the full subcategory of $\usualRA{A}/I\text{-}\mathbf{Mod}$ whose modules are finite-dimensional over $F$. 
    \item $\rep (\curlS, \rho)$, the full subcategory of $\Rep (\curlS, \rho)$ whose $\curlS$-representations 
    are finite-dimensional over $F$.
\end{itemize}
By \cite[Corollary 2.2, Proposition 2.3, Corollary 2.4]{Berg} the equivalences $\Omega$ and $\Gamma$ restrict to equivalences 
\[
\begin{array}{ccc}
\usualRA{A}/I\text{-}\mathbf{Mod}\to  \Rep (\curlS, \rho), &&    \Rep (\curlS, \rho)\to \usualRA{A}/I\text{-}\mathbf{Mod}, \\    \usualRA{A}/I\text{-}\mathbf{mod}\to  \rep (\curlS, \rho), &&   \rep (\curlS, \rho)\to \usualRA{A}/I\text{-}\mathbf{mod}.
\end{array}
\]

\subsubsection{\textbf{Representations as tuples of semilinear maps}}
\label{subsubsec-semilinear-clannish}

For a field $K$, a field automorphism $\rho:K\rightarrow K$ and $K$-vector spaces $M$ and $N$, we~write
$$
\Hom_{K}^\rho(M,N)\coloneqq \{\varphi:M\rightarrow N\suchthat \text{for all} \ \alpha\in K \ \text{and for all} \ m,n\in M: \ \varphi(\alpha m+n)=\rho(\alpha)\varphi(m)+\varphi(n)\}.
$$

The following lemma is well known, see e.g. \cite[Lemma 12.5]{LZ}, where the result is proved under the assumption that $\gcd(d_i,d_j)=1$.

\begin{lemma}\label{lemma:simplyfing-reps-of-these-species} Let $d$ be a positive integer, $d_i$ and $d_j$ be positive divisors of $d$, $F$ a field containing a primitive $d^{\operatorname{th}}$ root of unity, $E/F$ a degree-$d$ cyclic Galois extension, and $F_i,F_j\subseteq E$ the subfields of $E$ containing $F$ such that $[F_i:F]=d_i$ and $[F_j:F]=d_j$. For any given $\rho\in \Gal(F_i\cap F_j/F)$, there exist $F$-vector space isomorphisms
\begin{equation}\label{eq:reps-of-species-simplified-via-semilinear-maps}
\Hom_{F_i}((F_i^\rho\otimes_{F_i\cap F_j}F_j)\otimes_{F_j} M, N)\cong\Hom_{F_i\cap F_j}^{\rho}(M,N)\cong\Hom_{F_j}(M,(F_j^{\rho^{-1}}\otimes_{F_i\cap F_j}F_i)\otimes_{F_i}N),
\end{equation}
natural in the $F_j$-vector space $M$ and the $F_i$-vector space $N$.
\end{lemma}

\begin{proof}  The proof is rather elementary and follows e.g. from \cite[p. 14]{Dlab-Ringel-reps-of-graphs} by first noting that, given any field automophism $\widetilde{\rho}:F_iF_j\rightarrow F_iF_j$ with $\widetilde{\rho}|_{F_i\cap F_j}=\rho$, the rules 
\begin{align*}
a\otimes b & \mapsto [x\otimes y\mapsto a \operatorname{Tr}_{F_i/F_i\cap F_j}(\widetilde{\rho}^{-1}(bx))y]
&
a\otimes b & \mapsto [x\otimes y\mapsto x\operatorname{Tr}_{F_j/F_i\cap F_j}(\widetilde{\rho}(ya))b]
\end{align*}
induce well-defined $F_j$-$F_i$ bimodule isomorphisms
\begin{align*}
F_j^{\rho^{-1}}\otimes_{F_i\cap F_j}F_i & \rightarrow \Hom_{\mathbf{vec}\text{-}F_j}(F_i^{\rho}\otimes_{F_i\cap F_j} F_j,F_j)
&
F_j^{\rho^{-1}}\otimes_{F_i\cap F_j}F_i & \rightarrow \Hom_{F_i\text{-}\mathbf{vec}}(F_i^{\rho}\otimes_{F_i\cap F_j} F_j,F_i),
\end{align*}
where $\operatorname{Tr}_{F_j/F_i\cap F_j}:F_j\rightarrow F_i\cap F_j$ is the trace function of the field extension $F_j/F_i\cap F_j$.

Now, under our hypotheses, we have the eigenbases $\mathcal{B}^j_{i,j}$ and 
$\{\omega^{-1}\suchthat \omega\in \mathcal{B}^j_{i,j}\}$ of $F_j/F_i\cap F_j$ at hand. One easily checks that they are 
mutually dual with respect to the $F_i\cap F_j$-$F_i\cap F_j$-bilinear form 
$$\operatorname{Tr}_{F_j/F_i\cap F_j}(\bullet \cdot ?):F_j\times F_j\rightarrow F_i\cap F_j,$$ and this enables us to give 
very concrete formulae for the desired isomorphisms \eqref{eq:reps-of-species-simplified-via-semilinear-maps}. For this reason, and 
in order to establish some notation that will be used later, we explicitly exhibit these formulae.

We shall use the natural abbreviations $$F_i^\rho\otimes_{F_i\cap F_j} M=(F_i^\rho\otimes_{F_i\cap F_j}F_j)\otimes_{F_j} M \quad \text{and} \quad F_j^{\rho^{-1}}\otimes_{F_i\cap F_j}N=(F_j^{\rho^{-1}}\otimes_{F_i\cap F_j}F_i)\otimes_{F_i}N.$$

\begin{itemize}
\item[(1)] $\Hom_{F_i}(F_i^\rho\otimes_{F_i\cap F_j} M, N)\cong\Hom_{F_i\cap F_j}^{\rho}(M,N)$.
\end{itemize}
A straightforward computation shows that
\begin{align*}
\overrightarrow{\bullet}:\Hom_{F_i}(F_i^\rho\otimes_{F_i\cap F_j} M, N)\rightarrow\Hom_{F_i\cap F_j}^{\rho}(M,N) & & \overrightarrow{f}(m)&\coloneqq f(1\otimes m)\\
\overleftarrow{\bullet}:\Hom_{F_i\cap F_j}^{\rho}(M,N)\rightarrow\Hom_{F_i}(F_i^\rho\otimes_{F_i\cap F_j} M, N) & & \overleftarrow{a}(e\otimes m)&\coloneqq ea(m)
\end{align*}
are well-defined, mutually inverse $F$-vector space isomorphisms.

\begin{itemize}
\item[(2)] $\Hom_{F_i\cap F_j}^{\rho}(M,N)\cong\Hom_{F_j}(M,F_j^{\rho^{-1}}\otimes_{F_i\cap F_j}N)$.
\end{itemize}
Let $\mathcal{B}^j_{i,j}$ be an eigenbasis of the field extension $F_j/F_i\cap F_j$. Such a basis exists because $K/F$ is a finite-degree cyclic Galois field extension and $F$ contains a primitive $[K:F]$-th root of unity.

Given $f\in \Hom_{F_j}(M,F_j^{\rho^{-1}}\otimes_{F_i\cap F_j}N)$ and $m\in M$, one can uniquely write
$$
f(m)=\sum_{\omega\in\mathcal{B}^j_{i,j}}\omega^{-1}\otimes n_{f,m,\omega^{-1}}
$$
for some elements $n_{f,m,\omega^{-1}}\in N$ uniquely determined by $f$ and $m$.
With this in mind, one can verify that
\begin{align*}
\overrightarrow{\bullet}:\Hom_{F_i\cap F_j}^{\rho}(M,N)\rightarrow\Hom_{F_j}(M,F_j^{\rho^{-1}}\otimes_{F_i\cap F_j}N) & &  \overrightarrow{b}(m)&\coloneqq \sum_{\omega\in\mathcal{B}^j_{i,j}}\omega^{-1}\otimes b(\omega m) \\
\overleftarrow{\bullet}:\Hom_{F_j}(M,F_j^{\rho^{-1}}\otimes_{F_i\cap F_j}N)\rightarrow\Hom_{F_i\cap F_j}^{\rho}(M,N) & & \overleftarrow{f}(m)&\coloneqq n_{f,m,1} 
\end{align*}
 are well-defined, mutually inverse $F$-vector space isomorphisms, independent of the eigenbasis $\mathcal{B}^j_{i,j}$ of $F_j/F_i\cap F_j$ chosen.

\begin{itemize}
\item[(3)] $\Hom_{F_j}(M,F_j^{\rho^{-1}}\otimes_{F_i\cap F_j}N)\cong \Hom_{F_i}(F_i^\rho\otimes_{F_i\cap F_j} M, N)$.
\end{itemize}
The maps $\overleftarrow{\overleftarrow{\bullet}}:\Hom_{F_j}(M,F_j^{\rho^{-1}}\otimes_{F_i\cap F_j}N)\rightarrow\Hom_{F_i}(F_i^\rho\otimes_{F_i\cap F_j} M, N)$ and $\overrightarrow{\overrightarrow{\bullet}}:\Hom_{F_i}(F_i^\rho\otimes_{F_i\cap F_j} M, N)\rightarrow\Hom_{F_j}(M,F_j^{\rho^{-1}}\otimes_{F_i\cap F_j}N)$, obtained by composing the isomorphisms from (1) and (2) above, are obviously inverse to each other. For the convenience of the reader we display the rules these maps obey:
$$
\overleftarrow{\overleftarrow{f}}(x\otimes m)=xn_{f,m,1}, \ \ \text{and} \ \ \overrightarrow{\overrightarrow{g}}(m)=\sum_{\omega\in\mathcal{B}_j}\omega^{-1}\otimes g(1\otimes \omega m).
$$
\end{proof}

\begin{coro}\label{coro:rep-of-species-equiv-to-cat-of-semilinear-maps} Let $(Q,\bldD)$ be a weighted quiver, $d\coloneqq \lcm\{d_i\suchthat i\in Q_0\}$, $E/F$ and $(F_i)_{i\in Q_0}$ be as in \S\ref{subsubsec:modulating-functions-def-of-concept}, with the field $F$ containing a primitive $d^{\operatorname{th}}$ root of unity, and let $$g=(g_a)_{a\in Q_1}\in\times_{a\in Q_1}\Gal(F_{h(a)}\cap F_{t(a)}/F)$$ be a modulating function. The category of representations of the $F$-species $$((F_i)_{i\in Q_0},(F_{h(a)}^{g_a}\otimes_{F_{h(a)}\cap F_{t(a)}}F_{t(a)})_{a\in Q_1})$$ is isomorphic to the category
\begin{enumerate}
\item whose objects are the pairs $((M_i)_{i\in Q_0},(\varphi_{a})_{a\in Q_1})$ that attach an $F_i$-vector space $M_i$ to each $i\in Q_0$, and a map $\varphi_a\in \Hom^{g_a}_{F_{h(a)}\cap F_{t(a)}}(M_{t(a)},M_{h(a)})$ to each $a\in Q_1$ (i.e., an $F$-linear map $\varphi_a:M_{t(a)}\rightarrow M_{h(a)}$ such that $\varphi_a(\ell m)=g_a(\ell)\varphi_a(m)$ for every $\ell\in F_{h(a)}\cap F_{t(a)}$ and every $m\in M_{t(a)}$);
\item whose morphisms $((M_i)_{i\in Q_0},(\varphi_a)_{a\in Q_1}\rightarrow ((N_i)_{i\in Q_0},(\psi_a)_{a\in Q_1}$ are $Q_0$-tuples $(f_i)_{i\in Q_0}$ consisting of an $F_i$-linear map $f_i:M_i\rightarrow N_i$ for each $i\in Q_0$, such that the diagram of $F$-linear maps
$$
\xymatrix{
M_{t(a)} \ar[r]^{\varphi_{a}} \ar[d]_{f_{t(a)}} & M_{h(a)} \ar[d]^{f_{h(a)}} \\
N_{t(a)} \ar[r]_{\psi_{a}} & N_{h(a)}
} 
$$
commutes for every $a\in Q_1$.
\end{enumerate}
\end{coro}

\begin{remark} Although Lemma \ref{lemma:simplyfing-reps-of-these-species} and Corollary \ref{coro:rep-of-species-equiv-to-cat-of-semilinear-maps} are of course valid in significantly broader generality, the very explicit formulae in the proof of Lemma \ref{lemma:simplyfing-reps-of-these-species}, that our hypotheses on the field extension $E/F$ allow us to obtain, will later enable us to perform many computations very explicitly.
\end{remark}

\vspace{1mm}

\subsection{Jacobian algebras}\label{subsec:Jac-algs}

\,

Let $(Q,\bldD)$ be a weighted quiver, $d\coloneqq \lcm\{d_i\suchthat i\in Q_0\}$, $E/F$ and $(F_i)_{i\in Q_0}$ be as in \S\ref{subsubsec:modulating-functions-def-of-concept}, with the field $F$ containing a primitive $d^{\operatorname{th}}$ root of unity, and let $$g=(g_a)_{a\in Q_1}\in\times_{a\in Q_1}\Gal(F_{h(a),t(a)}/F)$$ be a modulating function. Let $((F_i)_{i\in Q_0},(A_{a}(\inputtripleforblocks))_{a\in Q_1})$ be the corresponding $F$-modulation, and let 
$\usualRA{A(\inputtripleforblocks)}$ and $\compRA{A(\inputtripleforblocks)}$ be the path algebra and complete path algebra of $(Q,\mathbf{d},g)$, defined in \S\ref{subsubsec:modulating-functions-def-of-concept}. Recall $\maxcompRA{A(\inputtripleforblocks)}$ denotes the arrow ideal in $\compRA{A(\inputtripleforblocks)}$ defined by the product over $n>0$ of $n$-fold tensor products $A(\inputtripleforblocks)\otimes_{R}\cdots \otimes_{R}A(\inputtripleforblocks)$.  Denote by $\calB^i$ the eigenbasis of $F_i/F$ given by \eqref{eq:eigenbasis-intermediate-subextension}. 

\begin{defi}\label{def:paths-potentials-cyclic-derivatives-jacobian-algebras}\, 
\begin{enumerate}
\item Following \cite[Definition 4.4]{LZ} and \cite[Definition 3.6]{GLF1}, we define a \emph{path of length $n$ on $A(\inputtripleforblocks)$} to be any element $\omega_{0}a_{1}\omega_{1}\dots \omega_{n-1}a_{n}\omega_{n}\in  \compRA{A(\inputtripleforblocks)}$ such that 
\begin{itemize}
\item $a_1,\ldots,a_n$ are arrows of $Q$ such that $h(a_{r+1})=t(a_r)$ for $r=1,...,n-1$;
\item $\omega_0\in \calB^{h(a_1)}$ and $\omega_r\in\calB^{t(a_r)}$ for $r = 1,\ldots,n$.
\end{itemize}
A path $\omega_{0}a_{1}\omega_{1}\dots \omega_{n-1}a_{n}\omega_{n}$ is \emph{cyclic} if $h(a_1)=t(a_n)$.

\item A \emph{potential} on $A(\inputtripleforblocks)$ is any element $W\in (\maxcompRA{A(\inputtripleforblocks)})^2$ satisfying $W=\sum_{i\in Q_0}e_iWe_i$, i.e., any element of $\compRA{A(\inputtripleforblocks)}$ that can be written as a possibly infinite $F$-linear combination of cyclic paths of length $\geq 2$ on $A(\inputtripleforblocks)$, cf. \cite[Definition 5.1]{LZ} and \cite[Definition~3.11]{GLF1}.

\item A potential $W\in \compRA{A(\inputtripleforblocks)}$ will be called \emph{polynomial potential} if it actually belongs to $\usualRA{A(\inputtripleforblocks)}$.

\item\label{item:def-cyclic-derivative} Following \cite[Definition 3.11]{GLF1} and \cite[Equation~(10.1)]{GLF2}, for each arrow $a\in Q_{1}$ and each cyclic path $\omega_{0}a_{1}\omega_{1}\dots \omega_{n-1}a_{n}\omega_{n}$ on $A(\inputtripleforblocks)$  we define 
 the \emph{cyclic derivative} 
\begin{equation}\label{eq:def-cyclic-derivative-of-cyclic-path}
\partial_{a}(\omega_{0}a_{1}\omega_{1}\dots \omega_{n-1}a_{n}\omega_{n})\coloneqq \frac{1}{d_{a}}\sum_{m=0}^{d_{a}-1}g_{a}^{-1}(v^{\frac{-md}{d_{a}}})\left(\sum_{k=1}^{n}\delta_{a,a_{k}}\omega_{k}a_{k+1}\dots a_{n}\omega_{n}\omega_{0}a_{1}\dots a_{k-1}\omega_{k-1}\right)v^{\frac{md}{d_{a}}}
\end{equation}
where $d_{a}\coloneqq \gcd(d_{h(a)},d_{t(a)})$ and $\delta_{a,a_k}$ is the Kronecker delta between $a$ and $a_k$. The cyclic derivative $\partial_a(W)$ for an arbitrary potential $W$ on $A(\inputtripleforblocks)$ is defined by extending \eqref{eq:def-cyclic-derivative-of-cyclic-path} by $F$-linearity and continuity. 

\item For a potential $W\in \compRA{A(\inputtripleforblocks)}$, the \emph{Jacobian ideal} $J(W)$ is defined to be the topological closure of the two-sided ideal of $\compRA{A(\inputtripleforblocks)}$ generated by $\{\partial_a(W)\suchthat a\in Q_1\}$, and the \emph{Jacobian algebra} is the quotient
$$
\jacobalg{A(\inputtripleforblocks),W}\coloneqq \compRA{A(\inputtripleforblocks)}/J(W).
$$

\item
Given a polynomial potential $W$ on $A(\inputtripleforblocks)$ we will say that the two-sided ideal 
\[
J_{0}(W)
\coloneqq \langle\partial_{a}(W)\mid a\in Q_{1}\rangle
\] 
of $\usualRA{A(\inputtripleforblocks)}$
is the \emph{polynomial Jacobian ideal} of $W$.
\end{enumerate}
\end{defi}

\begin{remark}
\label{remark-properties-of-jacobian-blocks}
\begin{enumerate}\item Let $p\coloneqq \sum_{k=1}^{n}\delta_{a,a_{k}}\omega_{k}a_{k+1}\dots a_{n}\omega_{n}\omega_{0}a_{1}\dots a_{k-1}\omega_{k-1}$ in \eqref{eq:def-cyclic-derivative-of-cyclic-path}. By Proposition \ref{prop:properties-of-Fi-Fj-bimodules}, there exist elements $p_\rho$, uniquely determined by $p$, such that $\rho$ runs in $\Gal(F_{h(a)}\cap F_{t(a)}/F)$, $p=\sum_{\rho}p_\rho$, and $p_\rho z=\rho(z)p_{\rho}$ for every $z\in F_{h(a)}\cap F_{t(a)}$. That is, $p_\rho$ is the $\rho$-linear part of $p$. Thus, the cyclic derivative \eqref{eq:def-cyclic-derivative-of-cyclic-path} is the $g_a^{-1}$-linear part of $p$.
\item For cyclic derivatives defined in more general species and contexts, see \cite{Bautista-Lopez,Lopez-Aguayo,Rota-Sagan-Stein}.
\item If $W$ is a polynomial potential on $A(\inputtripleforblocks)$ such that $(\maxcompRA{A(\inputtripleforblocks)})^n\subseteq J(W)$ for some $n>0$, then the ring homomorphism $\usualRA{A(\inputtripleforblocks)}/J_0(W)\rightarrow \jacobalg{A(\inputtripleforblocks),W}$ induced by the canonical inclusion $\usualRA{A(\inputtripleforblocks)}\hookrightarrow\compRA{A(\inputtripleforblocks)}$ is an $F$-algebra isomorphism. 
See \cite[Lemma 2.6.2]{GeuenichPHD}. 
\end{enumerate}
\end{remark}

\begin{ex}
\label{ex-main-example-for-jacobian-blocks}
Let $(Q,\bldD)$ be the  weighted quiver
\[
\xymatrix{ &
& 1 \ar[dr]^{\gamma} & \\
&2 \ar[ur]^{\alpha} & & 3\ar[ll]^{\beta} 
}
\xymatrix{ &
& d_1=2  & \\
& d_2=1  & & d_3= 4
}
\]
Thus $d\coloneqq \lcm(d_{1},d_{2},d_{3})=4=[E:F]$. Take $E/F$, $(F_i)_{i\in Q_0}$ and $g=(g_a)_{a\in Q_1}\in\times_{a\in Q_1}\Gal(F_{h(a)}\cap F_{t(a)}/F)$ as in the opening paragraph of the ongoing \S\ref{subsec:Jac-algs}. 
Following \S\ref{subsubsec:modulating-functions-def-of-concept}, to begin this means that $F_{1}=L$, $F_{2}=F$ and $F_{3}=E$  where $L$ is the unique subfield of $E$ that contains $F$ and satisfies $[L:F]=2$ (hence also $[E:L]=2$). 
Furthermore, since $g_{a}$ is an $F$-linear automorphism of $F_{h(a),t(a)}=F_{h(a)}\cap F_{t(a)}$, we have that 
\[
\begin{array}{ccccc}
F_{1,2}=F,
&
F_{3,1}=L,
&
F_{2,3}=F,
&
g_{\alpha}=\myid_{F},
&
g_{\beta}=\myid_{F}.
\end{array}
\]
Then the species $((F_i)_{i\in Q_0},(A_{a}(\inputtripleforblocks))_{a\in Q_1})$ can be mnemotechnically  visualized as follows:
    $$
    \xymatrix{ &
& L \ar[dr]_{\gamma}^{E^{g_\gamma}\otimes_FL} & \\
&F \ar[ur]_{\alpha}^{L\otimes_FF} & & E\ar[ll]_{\beta}^{F\otimes_FE}. 
}
    $$
    where $L$ is the unique subfield of $E$ that contains $F$ and satisfies $[L:F]=2$ (hence also $[E:L]=2$). By 
    Lemma \ref{lemma:simplyfing-reps-of-these-species} and Corollary \ref{coro:rep-of-species-equiv-to-cat-of-semilinear-maps}, the category of representations of this species is equivalent to the category whose objects have the form
    \begin{equation}\label{eq:example-ultimately-explicit-representations-of-a-species}
    \xymatrix{ &
& M_1 \ar[dr]^{\varphi_\gamma} & \\
& M_2 \ar[ur]^{\varphi_\alpha} & & M_3\ar[ll]^{\varphi_\beta}. 
}
    \end{equation}
    with $M_1$ an $L$-vector space, $M_2$ an $F$-vector space, $M_3$ an $E$-vector space, $\varphi_\alpha$ an $F$-linear map, $\varphi_\beta$ an $F$-linear map, and $\varphi_\gamma$ an $F$-linear map satisfying $\varphi_\gamma(\ell  m)=g_\gamma(\ell)\varphi_\gamma(m)$ for $\ell\in L$ and $m\in M_1$. (We skip the description of the morphisms here.)
    
    Letting $\calB_{E/F}\coloneqq \{1,v,v^2,v^3\}$ be an eigenbasis of $E/F$ as in \S\ref{subsubsec:eigenbases}, and setting $u\coloneqq v^2$, we see that $\calB_{L/F}\coloneqq \{1,u\}$ is an eigenbasis of $L/F$. Furthermore, writing $\Gal(L/F)=\{\myid_L,\theta\}$, we have $g_\gamma=g_\gamma^{-1}=\theta^{\xi_\gamma}$ with $\xi_\gamma\in\ZZ/2\ZZ$, hence $g_\gamma(u)=(-1)^{\xi_\gamma}u$ and $g_\gamma^{-1}(u^{-1})=(-1)^{\xi_\gamma}u^{-1}$. Thus, for the potential $W=\alpha\beta\gamma$ we have
    \[
    \begin{array}{ccc}
    \partial_{\alpha}(\alpha\beta\gamma)=\beta\gamma, &
    \partial_{\beta}(\alpha\beta\gamma)=\gamma\alpha, & 
    \partial_{\gamma}(\alpha\beta\gamma)=\frac{1}{2}(\alpha\beta+(-1)^{\xi_\gamma}u^{-1}\alpha\beta u).
    \end{array}
    \] 
We now consider the consequence of Proposition \ref{prop:properties-of-Fi-Fj-bimodules}(4) in this context. 
    This result describes a direct sum decomposition of an arbitrary $L$-$E$-bimodule. 
    This specifies here to the  $L$-$E$-bimodule decomposition $\Hom_{F}(M_{3},M_{1})=\Hom_{L}(M_{3},M_{1})\oplus \Hom_{L}^{\theta}(M_{3},M_{1})$. 
    To see this, note any $\varphi\in \Hom_{F}(M_{3},M_{1})$ satisfies
    \[
    \begin{array}{ccc}
    \varphi=\varphi'+\varphi'',
    &
    \varphi'\coloneqq\frac{1}{2}(\varphi+u^{-1}\varphi u)\in \Hom_{L}(M_{3},M_{1}),
    &
    \varphi''\coloneqq\frac{1}{2}(\varphi-u^{-1}\varphi u)\in  \Hom_{L}^{\theta}(M_{3},M_{1}).
    \end{array}
    \]
    Indeed, it is clear that $\varphi',\varphi''\in \Hom_{F}(M_{3},M_{1})$, and since $u^{2}\in F$, for any $m\in M_{3}$ we have 
    \[
    2\varphi''(um)=\varphi(um)- u^{-1}\varphi(u^{2}m)=\varphi(um)-u\varphi(m)=-u(\varphi(m)-u^{-1}\varphi(um))=-2u\varphi''(m).
    \]
    Note also that $\varphi=u^{-2}\varphi u^{2}$ and so $\varphi'=u^{-1}\varphi'u$ and $\varphi''=-u^{-1}\varphi''u$. Therefore, the category of left modules over the Jacobian algebra $\jacobalg{A(Q,\mathbf{d},g)}$ is equivalent to the category whose objects are the objects \eqref{eq:example-ultimately-explicit-representations-of-a-species} above satisfying that $\varphi_\beta\circ\varphi_\gamma=0$, $\varphi_\gamma\circ\varphi_\alpha=0$, and that writing $\varphi_{\alpha}\circ\varphi_{\beta}=\varphi'+\varphi''$ as above gives
    \[
    0=\frac{1}{2}(\varphi+(-1)^{\xi_{\gamma}}u^{-1}\varphi u)=\frac{1}{2}(\varphi'+\varphi''+(-1)^{\xi_{\gamma}}u^{-1}(\varphi'+\varphi'')u)=\begin{cases}
        \varphi'
        &
        (\text{if }\xi_{\gamma}=0)
        \\
        \varphi''
        &
        (\text{if }\xi_{\gamma}=1)
    \end{cases}
    \]
    (We skip the description of the morphisms here.)
\end{ex}

\vspace{2mm}

\subsection{Semilinear clannish algebras}\label{subsec:semilinear-clannish-algs-def-of-concept} 

\,

As in \S\ref{subsubsec:semilinear-path-algs-def-of-concept} (and less generally than \cite{BTCB}), let $(\compactQhat,\compactblddhat)$ be a weighted quiver all of whose weights are the same positive integer $\compactdhat$, $K/F$ be a degree-$\compactdhat$ cyclic Galois field extension, and $\bldsigma\colon \compactQhat_{1}\to \Gal(K/F)$ be a function assigning a field automorphism $\sigma_{b}\in\Gal(K/F)$ to each $b\in \compactQhat_{1}$. Let $K_{\bldsigma}\compactQhat$ be the semilinear path algebra constructed with this data.

\begin{defi}
\label{defi-semilinear-clannish-algebra}
\cite[\S1,~\S2.1,~\S2.3]{BTCB} A \emph{semilinear clannish algebra} is a $K$-ring of the form $K_{\bldsigma}\compactQhat/I$ where $I=\langle Z\cup S\rangle$ is generated by a set $Z$, of \emph{zero}-\emph{relations}, and a set $S$, of \emph{special}-\emph{relations}, such that:
\begin{enumerate}
    \item[(Q)] $\compactQhat$ contains a specified set $\bbS$ of \emph{special loops}, the arrows not belonging to $\bbS$ thus being called \emph{ordinary}, and such that for any $i\in\compactQhat_{0}$ there are at most $2$ arrows $b\in\compactQhat_1$ with $h(b)=i$, and at most two arrows $c\in\compactQhat_1$ with $t(c)=i$; 
    \item[(Z)] $Z$ consists of paths in $\compactQhat$ of length at least two, such that:
    \begin{itemize}
        \item for any ordinary arrow $a$ there is at most one arrow $b\in \compactQhat_1$  with $ba$ a path outside $Z$;
        \item for any ordinary arrow $a$ there is at most one arrow  $b\in \compactQhat_1$ with $ab$ a path outside $Z$; and 
        \item if $p\in Z$ and  $s\in \bbS$, then $p$ does not start nor end with $s$, and does not contain $s^{2}$ as a subpath.  
    \end{itemize}
    \item[(S)] $S=\{s^{2}-\beta_{s} s+\gamma_{s} e_{i}\mid s\in \bbS,h(s)=i=t(s)\}$ for some elements $\beta_{s},\gamma_{s}\in K$ ($s\in \bbS$). That is, $S$ is given by specifying a quadratic polynomial $q_{s}(x)=x^{2}-\beta_{s}x+\gamma_{s}\in K[x;\sigma_{s}]$ for each $s\in\bbS$. 
    \end{enumerate}
\end{defi}

\begin{remark} The semilinear clannish algebras introduced in \cite{BTCB} are far more general: $K$ is allowed to be a division ring, not necessarily finite-dimensional over any field, and for each the arrow $b\in\compactQhat_1$, $\sigma_b$ is allowed to be any ring automorphism of $K$.
\end{remark}

\begin{defi}\, 
\label{def:types-of-clannish}
\begin{enumerate}
\item
Given $\sigma\in\Aut(K)$ and $q(x)=x^{2}-\beta x +\gamma\in K[x;\sigma]$ we say that $q(x)$ is:
\begin{enumerate}[(i)]
    \item \emph{normal} if the left and right ideals generated by $q(x)$  coincide, that is, $K[x;\sigma]q(x)=q(x)K[x;\sigma]$; 
    \item \emph{non}-\emph{singular} if the constant term of $q(x)$ is non-zero, that is, $\gamma\neq 0$; and
    \item \emph{of semisimple type} if the quotient $K[x;\sigma]/\langle q(x)\rangle$ is a semisimple ring. 
\end{enumerate}
\item
A semilinear clannish algebra $K_{\bldsigma}\compactQhat/I$ is  \emph{normally}-\emph{bound}  \emph{non}-\emph{singular} or \emph{of semisimple type} if each of the quadratics $q_{s}(x)\in K[x;\sigma_{s}]$ ($s\in\bbS$) is normal, non-singular or of semisimple type, respectively. 
\end{enumerate}
\end{defi}

\begin{remark}\label{rem:normality-and-semisimplicity-of-q}
For $\sigma\in \Aut(K)$ and $\mu\in K$, set $q(x)\coloneqq x^{2}-\mu\in K[x;\sigma]$.  
By \cite[Lemma~2.1(i)]{BTCB}, to say that $q(x)$ is normal is equivalent to the conditions that $\sigma(\mu)=\mu$ and $\sigma^{2}(\lambda)\mu=\mu\lambda$ for all $\lambda\in K$.
\end{remark}

\begin{ex}\label{ex:specific-pols-of-semisimple-type}
Let $K$ be a field.
\begin{enumerate}
    \item If $\sigma\coloneqq \myid_{K}$ and $\mu\in\{u\in K\suchthat\nexists x\in K \ \text{with} \ x^2=u\}$, then $L[x;\sigma]/\langle q(x)\rangle=L[x]/\langle q(x)\rangle$ is a field,  
    so $q(x)\coloneqq x^{2}-\mu\in K[x;\sigma]=K[x]$ is of semisimple type.
    \item If $\sigma$ is an order-$2$ field automorphism of $K$ and $\mu\coloneqq 1$, then
    $K[x;\sigma]/\langle q(x)\rangle \cong F^{2\times 2}$, where $F\coloneqq \{x\in K\suchthat \sigma(x)=x\}$, so $q(x)\coloneqq x^{2}-\mu\in K[x;\sigma]$ is of semisimple type.
\end{enumerate}
\end{ex}

\begin{ex}
\label{ex-main-example-for-semilinear-clannish-blocks}
Let $\compactQhat $ be one of the following three connected quivers
\[
\begin{array}{ccc}
\xymatrix{
& 1 \ar[dr]^{\gamma} & \\
2 \ar[ur]^{\alpha} & & 3 \ar[ll]^{\beta} 
}
&
\xymatrix{
 & 1 \ar@(lu,ru)^{s_1} \ar[dr]^{\gamma} &\\
 2 \ar[ur]^{\alpha} & & 3\ar[ll]^{\beta}
}
&
\xymatrix{
 & 1  \ar[dr]^{\gamma} &  \\
 2 \ar@(u,l)_{s_2} \ar[ur]^{\alpha} & & 3 \ar@(u,r)^{s_3} \ar[ll]^{\beta} 
}
\end{array}
\]
Set $\bbS\coloneqq \compactQhat_{1}\setminus \{\alpha,\beta,\gamma\}$, meaning that every loop in $\compactQhat$ is special. 
Take the weight $\compactdhat$ to be either $1$ or $2$, and let $K/F$ be a degree-$\compactdhat$ field extension. Fix arbitrary elements 
\[
\begin{array}{cccc}
\sigma_{\alpha},\sigma_{\beta},\sigma_{\gamma},\sigma_{s_{i}}\in \Gal(K/F), & \mu_{i}\in K, & q_{s_{i}}(x)=x^{2}-\mu_{i}\in K[x;\sigma_{s_{i}}],     & (s_{i}\in \bbS).
\end{array}
\]
Let $Z=\{\alpha\beta,\beta\gamma,\gamma\alpha\}$. 
It is straightforward to observe that conditions (Q), (Z) and (S) from Definition \ref{defi-semilinear-clannish-algebra} hold, so  $K_{\bldsigma}\compactQhat /\langle Z\cup S\rangle$ is a semilinear clannish algebra. 

We now specify $K$, $\sigma$ and $\mu$ to particular examples that arise from Section \ref{sec:building-blocks} on. By Remark \ref{rem:normality-and-semisimplicity-of-q} and Example \ref{ex:specific-pols-of-semisimple-type}, $K_{\bldsigma}\compactQhat /\langle Z\cup S\rangle$ is normally-bound, non-singular and of semisimple type in each of the following two situations: 
\begin{itemize}
\item $[K:F]=2$, $\operatorname{char}(K)\neq 2$, and $(\sigma_{s_{i}},\mu_{i})\in\{(\theta,1)\}\cup\{(\myid_{K},u)\suchthat \nexists x\in K \ \text{with} \ x^2=u\}$ for each special loop $s_i$, where $\Gal(K/F)=\{\myid_K,\theta\}$.
\item 
$K=F$ and $(\sigma_{s_{i}},\mu_{i})\in \{(\myid_{F},u)\suchthat \nexists x\in K \ \text{with} \ x^2=u\}$ for each special loop $s_i$.
\end{itemize}
\end{ex}

\vspace{1mm}

\section{Specific field extensions}
\label{sec:specific-field-extensions}

In this short section we give a brief description of the specific field extensions over which we will define Jacobian algebras and semilinear clannish algebras for triangulations.

\begin{defi}[Degree-4 datum] We will say that a $E/F$ is a \emph{degree-$4$ datum} if:
\begin{enumerate}
\item $E/F$ is a degree-$4$ cyclic Galois field extension; and
\item $F$ contains a primitive $4^{\operatorname{th}}$ root of unity.
\end{enumerate}
\end{defi}

For a degree-$4$ datum $E/F$, the following notation will always be adopted:
\begin{itemize}
\item $\rho:E\rightarrow E$ will be a generator of the Galois group $\Gal(E/F)$, so $\Gal(E/F)=\{\myid_E,\rho,\rho^2,\rho^3\}$;
\item $\zeta\in F$ will be a primitive $4^{\operatorname{th}}$ root of unity;
\item $v\in E\setminus\{0\}$ will be an eigenvector for $\rho$ with eigenvalue $\zeta$, i.e., $\rho(v)=\zeta v$, and $u\coloneqq v^2$;
\item $L/F$ will be the unique degree-$2$ field extension with $L\subseteq E$, i.e., $L=F(u)=\Fix_E(\rho^2)$;
\item $\Gal(L/F)=\{\myid_L,\theta\}$, i.e., $\myid_E|_L=\myid_L=\rho^2|_L$ and $\rho|_L=\theta=\rho^3|_L$.
\end{itemize}

\begin{ex}
Let $F$ be a finite field whose characteristic $p$ is a prime number congruent to $1$ modulo $4$, and let $E/F$ be the unique degree-$4$ field extension of $F$ inside an \emph{a priori} given algebraic closure of $F$. Then $E/F$ is a degree-$4$ datum.
\end{ex}

\begin{ex} Let  $p$ be a positive prime number congruent to $1$ modulo $4$, $F$ be any finite extension of the field of $p$-adic numbers $\mathbb{Q}_p$, and $E/F$ be the unique degree-$4$ unramified extension of $F$ inside an \emph{a priori} given algebraic closure of $F$. Then $E/F$ is a degree-$4$ datum. See, e.g., \cite[\S5.3 and \S5.4]{Gouvea-p-adic-numbers}
or \cite[\S III.3]{Koblitz-p-adic-numbers}.
\end{ex}

\begin{defi}[Degree-2 datum]
We will say that a $L/F$ is a \emph{degree-$2$ datum} if:
\begin{enumerate}
\item $L/F$ is a degree-$2$ field extension; and
\item $F$ contains a primitive $2^{\operatorname{nd}}$ root of unity, i.e., $\operatorname{char}(F)\neq 2$.
\end{enumerate}
\end{defi}

Notice that for a field $F$ containing a primitive $2^{\operatorname{nd}}$ root of unity, every degree-$2$ field extension is Galois with cyclic Galois group.

For a degree-$2$ datum, the following notation will always be adopted:
\begin{itemize}
\item $\theta:L\rightarrow L$ will be a generator of the Galois group $\Gal(L/F)$, so $\Gal(L/F)=\{\myid_L,\theta\}$;
\item $u\in L\setminus\{0\}$ will be an eigenvector for $\theta$ with eigenvalue $-1$, i.e., $\theta(u)=-u$.
\end{itemize}

Notice that $c\coloneqq u^2\in F$ and $\theta^{-1}(u^{-1})=\theta(u^{-1})=\theta(u)^{-1}=-u^{-1}$.

\begin{ex}
The well-known field extension $\mathbb{C}/\mathbb{R}$ is a degree-$2$ datum for which $\theta$ is the usual conjugation of complex numbers, and one can take $u$ to be $i$ or $-i$.
\end{ex}

\begin{defi}[Degree-1 datum]
By \emph{a degree-$1$ datum} we simply mean a field $F$ with no further conditions imposed.
\end{defi}

\begin{remark}
Every degree-$4$ datum $E/F$ contains two degree-$2$ data, namely $E/L$ and $L/F$. However, it is not true that every degree-$2$ datum is part of a degree-$4$ datum. Consider, for instance, the field extension $\mathbb{C}/\mathbb{R}$: degree-$2$ extensions of $\mathbb{C}$ do not exist because $\mathbb{C}$ is algebraically closed, and no subfield $F$ of $\mathbb{R}$ satisfies $[\mathbb{R}:F]=2$ because $\Aut(\mathbb{R})=\{\myid_\mathbb{R}\}$. This is the technical reason why in Section \ref{sec:Jac-algs-and-semilinear-clan-algs-of-colored-triangs} the constructions from \S\ref{subsec:constant-weights-algs-def-over-C/R} cannot be simply said to be a particular case of the constructions from \S\ref{subsec:algebras-for-arb-weights}.
\end{remark}

\section{Three-vertex blocks}\label{sec:building-blocks}

In this section we introduce two lists of $3$-vertex algebras. The first list will consist of $10$ Jacobian algebras (see Tables \ref{table-Jacobian-blocks-1-to-5} and \ref{table-Jacobian-blocks-6-to-10}), whereas the second one will consist of $10$ semilinear clannish algebras (see Tables \ref{table-semilinear-clannish-blocks-1-to-5} and \ref{table-semilinear-clannish-blocks-6-to-10}). The main aim of the section is to show that for each $k=1,\ldots,10$, the $k^{\operatorname{th}}$ Jacobian block is Morita-equivalent to the $k^{\operatorname{th}}$ semilinear clannish block.

The ten Jacobian blocks are instances of the algebras constructed in \cite{GLF1,GLF2}, whereas, except for the blocks 8 and 10, the construction of the semilinear clannish blocks is brand new. Later on, in Section \ref{sec:Jac-algs-and-semilinear-clan-algs-of-colored-triangs}, we will separately associate a Jacobian algebra and a semilinear-clannish algebra to each colored triangulation of a surface with orbifold points. It will turn out that these can alternatively be obtained by gluing copies of the blocks we are about to introduce.

The reader will notice that in Tables \ref{table-Jacobian-blocks-1-to-5}, \ref{table-Jacobian-blocks-6-to-10}, \ref{table-semilinear-clannish-blocks-1-to-5} and \ref{table-semilinear-clannish-blocks-6-to-10} some entries of the weight triple $\bldD=(d_1,d_2,d_3)$ appear enclosed in a small circle. This means that the corresponding vertex is an \emph{outlet} that in Section \ref{sec-morita-equivalence-between-algebras-from-triangulations} below will be allowed to be matched and glued to another outlet, in a fashion similar to \cite[Definition 13.1]{Fomin-Shapiro-Thurston} and \cite{Brustle-kit}.

\afterpage{%
    \clearpage
    \thispagestyle{empty}
    \begin{landscape}
        \centering 
\renewcommand{\arraystretch}{1.25}
{\small
\begin{tabular}{|c|c|c|c|c|c|}
\hline
 & Block 1 & Block 2 & Block 3 & Block 4 & Block 5 \\
\hline
\begin{tabular}{c}
Weight triple\\ 
$\left(\begin{array}{ccc} & d_1 &\\ d_2 & & d_3\end{array}\right)$\end{tabular} & 
$\left(\begin{array}{ccc} & \circled{2} &\\ \circled{2} & & \circled{2} \end{array}\right)$&
$\left(\begin{array}{ccc} & 1 &\\ \circled{2} & & \circled{2} \end{array}\right)$
&
$\left(\begin{array}{ccc} & 4 &\\ \circled{2} & & \circled{2}\end{array}\right)$ & $\left(\begin{array}{ccc} & \circled{2} &\\ 1 & & 1\end{array}\right)$ & $\left(\begin{array}{ccc} & \circled{2} &\\ 4 & & 4\end{array}\right)$ \\
\hline
\begin{tabular}{c}
Vertex fields\\
$\begin{array}{ccc} & F_1 &\\ F_2 & & F_3\end{array}$
\end{tabular}&
\begin{tabular}{c}\\ $\begin{array}{ccc} & L &\\ L & & L\end{array}$\end{tabular}
&
\begin{tabular}{c}\\ $\begin{array}{ccc} & F &\\ L & & L\end{array}$\end{tabular}
& 
\begin{tabular}{c}\\ $\begin{array}{ccc} & E &\\ L & & L\end{array}$\end{tabular} & \begin{tabular}{c}\\ $\begin{array}{ccc} & L &\\ F & & F\end{array}$ \end{tabular} & \begin{tabular}{c}\\ $\begin{array}{ccc} & L &\\ E & & E\end{array}$\end{tabular}\\
\hline
\begin{tabular}{c}
Arrow bimodules\\
\begin{tabular}{c}$A(Q,\bldD,\xi)_\alpha$\\
$A(Q,\bldD,\xi)_\beta$\\ 
$A(Q,\bldD,\xi)_{\beta_{0}}$\\
$A(Q,\bldD,\xi)_{\beta_{1}}$\\
$A(Q,\bldD,\xi)_\gamma$\end{tabular}
\end{tabular}
&
\begin{tabular}{c}
\\
\begin{tabular}{c}$L^{\theta^{\xi_\alpha}}\smallotimesL L$\\
$L^{\theta^{\xi_\beta}}\smallotimesL L$\\ 
\\
\\
$L^{\theta^{\xi_\gamma}}\smallotimesL L$\end{tabular}
\end{tabular}
&
\begin{tabular}{c}
\\
\begin{tabular}{c}$F\smallotimesF L$\\
$L^{\theta^{\xi_\beta}}\smallotimesL L$\\ 
\\
\\
$L\smallotimesF F$\end{tabular}
\end{tabular}
& 
\begin{tabular}{c}
\\
\begin{tabular}{c}$E^{\theta^{\xi_\alpha}}\smallotimesL L$\\
$L^{\theta^{\xi_\beta}}\smallotimesL L$\\ 
\\
\\
$L^{\theta^{\xi_\gamma}}\smallotimesL E$\end{tabular}
\end{tabular}
& 
\begin{tabular}{c}
\\
\begin{tabular}{l}$L\smallotimesF F$\\
\\
$F\smallotimesF F$\\
$F\smallotimesF F$\\ $F\smallotimesF L$\end{tabular}
\end{tabular}
&
\begin{tabular}{c}
\\
\begin{tabular}{l}$L^{\theta^{\xi_\alpha}}\smallotimesL L$\\
\\
$E^{\rho^l}\smallotimesE E$\\
$E^{\rho^{l+2}}\smallotimesE E$\\ $L^{\theta^{\xi_\gamma}}\smallotimesL L$\end{tabular}
\end{tabular}
\\
\hline
Mnemotechnics & 
$\xymatrix@C=0.7em@R=6em{
& L \ar[dr]_(0.6){\gamma}^(.4){L^{\theta^{\xi_\gamma}}\smallotimesL L} & \\
L \ar[ur]_(0.4){\alpha}^(.6){L^{\theta^{\xi_\alpha}}\smallotimesL L} & & L \ar[ll]_{\beta}^{L^{\theta^{\xi_\beta}}\smallotimesL L} 
}$
&
$\xymatrix@C=0.7em@R=6em{
& F \ar[dr]_(.6){\gamma}^(.4){L\smallotimesF F} & \\
L \ar[ur]_(.4){\alpha}^(.6){F\smallotimesF L} & & L \ar[ll]_{\beta}^{L^{\theta^{\xi_\beta}}\smallotimesL L} 
}$
&
$\xymatrix@C=0.7em@R=6em{
 & E  \ar[dr]_(.6){\gamma}^(.4){L^{\theta^{\xi_\gamma}}\smallotimesL E} &\\
 L \ar[ur]_(.4){\alpha}^(.6){E^{\theta^{\xi_\alpha}}\smallotimesL L} & & L\ar[ll]_{\beta}^{L^{\theta^{\xi_\beta}}\smallotimesL L}
}$
 &{
 $\xymatrix@R=2.5em{
  & L  \ar[ddr]_(.4){\gamma}^(.4){F\smallotimesF L} &  \\
  & &  \\
 F \ar[uur]_(.6){\alpha}^(.6){L\smallotimesF F}  & &  F \ar@/_0.7pc/[ll]^{\beta_0}_{F\smallotimesF F} \ar@/^0.9pc/[ll]_{\beta_1}^{F\smallotimesF F}
}$}
 &{
 $\xymatrix@R=2.5em{
  & L  \ar[ddr]_(.4){\gamma}^(.4){E^{\theta^{\xi_\gamma}}\smallotimesL L} &  \\
  & &  \\
 E \ar[uur]_(.6){\alpha}^(.6){L^{\theta^{\xi_\alpha}}\smallotimesL E}  & &  E \ar@/_0.7pc/[ll]^{\beta_0}_{E^{\rho^l}\smallotimesE E} \ar@/^0.9pc/[ll]_{\beta_1}^{E^{\rho^{l+2}}\smallotimesE E}
}$}
\\
\hline
\begin{tabular}{c}
Potential\\
$W(Q,\bldD,\xi)$\end{tabular}
&
$\alpha\beta\gamma$
&
$\alpha\beta\gamma$
& 
$\alpha\beta\gamma$
& 
$\beta_0\gamma\alpha+\beta_1\gamma u\alpha$
&
$\alpha(\beta_0+\beta_1)\gamma$
\\
\hline
\begin{tabular}{c}
Cyclic derivatives \\
$\partial_\alpha(W(Q,\bldD,\xi))$\\
$\partial_\beta(W(Q,\bldD,\xi))$\\
$\partial_{\beta_0}(W(Q,\bldD,\xi))$\\
$\partial_{\beta_1}(W(Q,\bldD,\xi))$\\
$\partial_\gamma(W(Q,\bldD,\xi))$
\end{tabular}
&
\begin{tabular}{c}
 \\
$\beta\gamma$\\
$\gamma\alpha$\\
\\
\\
$\alpha\beta$
\end{tabular}
&
\begin{tabular}{l}
 \\
$\beta\gamma$\\
$\frac{1}{2}(\gamma\alpha+(-1)^{\xi_\beta}u^{-1}\gamma\alpha u)$\\
\\
\\
$\alpha\beta$
\end{tabular}
& 
\begin{tabular}{c}
 \\
$\beta\gamma$\\
$\gamma\alpha$\\
\\
\\
$\alpha\beta$
\end{tabular}
& 
\begin{tabular}{l}
 \\
$\beta_0\gamma+\beta_1\gamma u$\\
\\
$\gamma\alpha$\\
$\gamma u\alpha$\\
$\alpha\beta_0+u\alpha\beta_1$
\end{tabular}
& 
\begin{tabular}{l}
 \\
$(\beta_0+\beta_1)\gamma$\\
\\
$\frac{1}{2}(\gamma\alpha+\rho^{-l}(v^{-1})\gamma\alpha v)$\\
$\frac{1}{2}(\gamma\alpha+\rho^{-l-2}(v^{-1})\gamma\alpha v)$\\
$\alpha(\beta_0+\beta_1)$
\end{tabular}
 \\
\hline
\end{tabular}}
       \captionof{table}{Jacobian blocks 1 to 5 \label{table-Jacobian-blocks-1-to-5}}
    \end{landscape}
    \clearpage
}

\afterpage{%
    \clearpage
    \thispagestyle{empty}
    \begin{landscape}
        \centering 
\renewcommand{\arraystretch}{1.25}
{\small
\begin{tabular}{|c|c|c|c|c|c|}
\hline
& Block 6 & Block 7 & Block 8 & Block 9 & Block 10 \\
\hline
\begin{tabular}{c}
Weight triple\\ 
$\left(\begin{array}{ccc} & d_1 &\\ d_2 & & d_3\end{array}\right)$\end{tabular} & 
$\left(\begin{array}{ccc} & \circled{2} &\\ 4 & & 1\end{array}\right)$&
$\left(\begin{array}{ccc} & \circled{2} &\\ 1 & & 4\end{array}\right)$
&
$\left(\begin{array}{ccc} & \circled{1} &\\ \circled{1} & & \circled{1}\end{array}\right)$ & $\left(\begin{array}{ccc} & 2 &\\ \circled{1} & & \circled{1}\end{array}\right)$ & $\left(\begin{array}{ccc} & \circled{1} &\\ 2 & & 2\end{array}\right)$ \\
\hline
\begin{tabular}{c}
Vertex fields\\
$\begin{array}{ccc} & F_1 &\\ F_2 & & F_3\end{array}$
\end{tabular}&
\begin{tabular}{c}\\ $\begin{array}{ccc} & L &\\ E & & F\end{array}$\end{tabular}
&
\begin{tabular}{c}\\ $\begin{array}{ccc} & L &\\ F & & E\end{array}$\end{tabular}
& 
\begin{tabular}{c}\\ $\begin{array}{ccc} & F &\\ F & & F\end{array}$\end{tabular} & \begin{tabular}{c}\\ $\begin{array}{ccc} & L &\\ F & & F\end{array}$ \end{tabular} & \begin{tabular}{c}\\ $\begin{array}{ccc} & F &\\ L & & L\end{array}$\end{tabular}\\
\hline
\begin{tabular}{c}
Arrow bimodules\\
\begin{tabular}{c}$A(Q,\bldD,\xi)_\alpha$\\
$A(Q,\bldD,\xi)_\beta$\\ 
$A(Q,\bldD,\xi)_{\beta_{0}}$\\
$A(Q,\bldD,\xi)_{\beta_{1}}$\\
$A(Q,\bldD,\xi)_\gamma$\end{tabular}
\end{tabular}
&
\begin{tabular}{c}
\\
\begin{tabular}{c}$L^{\theta^{\xi_\alpha}}\smallotimesL  E$\\
$E\smallotimesF F$\\ 
\\
\\
$F\smallotimesF L$\end{tabular}\end{tabular}
&
\begin{tabular}{c}
\\
\begin{tabular}{c}$L\smallotimesF F$\\
$F\smallotimesF E$\\ 
\\
\\
$E^{\theta^{\xi_\gamma}}\smallotimesL L$\end{tabular}\end{tabular}
& \begin{tabular}{c}
\\
\begin{tabular}{c}$F\smallotimesF F$\\
$F\smallotimesF F$\\ 
\\
\\
$F\smallotimesF F$\end{tabular}\end{tabular}
& \begin{tabular}{c}\\ \begin{tabular}{c}$L\smallotimesF F$ \\ 
$F\smallotimesF F$ \\
\\
\\
$F\smallotimesF L$\end{tabular}\end{tabular}
&\begin{tabular}{c}\\ \begin{tabular}{l}$F\smallotimesF L$ \\ 
\\
$L\smallotimesL L$
\\
$L^\theta\smallotimesL L$ \\ $L\smallotimesF F$\end{tabular}\end{tabular}\\
\hline
Mnemotechnics & 
$\xymatrix@C=0.7em@R=6em{
 & L  \ar[dr]_(.6){\gamma}^(.4){F\smallotimesF L} & \\
 E \ar[ur]_(.4){\alpha}^(.6){L^{\theta^{\xi_\alpha}}\smallotimesL E} & & F \ar[ll]_{\beta}^{E\smallotimesF F}
}$
&
$\xymatrix@C=0.7em@R=6em{
 & L  \ar[dr]_(.6){\gamma}^(.4){E^{\theta^{\xi_\gamma}}\smallotimesL L} & \\
 F \ar[ur]_(.4){\alpha}^(.6){L\smallotimesF F} & & E \ar[ll]_{\beta}^{F\smallotimesF E}
}$
&
$\xymatrix@C=0.7em@R=6em{
& F \ar[dr]_(.6){\gamma}^(.4){F\smallotimesF F} & \\
F \ar[ur]_(.4){\alpha}^(.6){F\smallotimesF F} & & F, \ar[ll]_{\beta}^{F\smallotimesF F} 
}$ & $\xymatrix@C=0.7em@R=6em{
& L \ar[dr]_(.6){\gamma}^{F\smallotimesF L} & \\
F \ar[ur]_(.4){\alpha}^{L\smallotimesF F} & & F, \ar[ll]_{\beta}^{F\smallotimesF F} 
}$ &  $\xymatrix@C=1.2em@R=2.5em{
 & & F  \ar[ddrr]_(.4){\gamma}^(.4){L\smallotimesF F} & & \\
 & & & & \\
 L \ar[uurr]_(.6){\alpha}^(.6){F\smallotimesF L} & & & & L \ar@/_1pc/[llll]^{\beta_0}_{L\smallotimesL L} \ar@/^1pc/[llll]_{\beta_1}^{L^\theta\smallotimesL L}
}$
\\
\hline
\begin{tabular}{c}
Potential\\
$W(Q,\bldD,\xi)$\end{tabular}
&
$\alpha\beta\gamma$
&
$\alpha\beta\gamma$
&  $\alpha\beta\gamma$ & $\alpha\beta\gamma$ &$\alpha(\beta_0+\beta_1)\gamma$ \\
\hline
\begin{tabular}{c}
Cyclic derivatives \\
$\partial_\alpha(W(Q,\bldD,\xi))$\\
$\partial_\beta(W(Q,\bldD,\xi))$\\
$\partial_{\beta_0}(W(Q,\bldD,\xi))$\\
$\partial_{\beta_1}(W(Q,\bldD,\xi))$\\
$\partial_\gamma(W(Q,\bldD,\xi))$
\end{tabular}
&
\begin{tabular}{l}
\\
$\frac{1}{2} (\beta\gamma + 
\theta^{-\xi_\alpha}(u^{-1})\beta\gamma u)$\\
$\gamma\alpha$ \\ 
\\
\\
$\alpha\beta$\end{tabular}
&
\begin{tabular}{l}
\\
$\beta\gamma$ \\ $\gamma\alpha$ \\ \\ \\  $\frac{1}{2} (\alpha\beta + 
\theta^{-\xi_\gamma}(u^{-1})\alpha\beta u)$\end{tabular}
& 
\begin{tabular}{c}\\
$\beta\gamma$ \\ $\gamma\alpha$ \\ \\ \\ $\alpha\beta$,\end{tabular}
& \begin{tabular}{c}\\
$\beta\gamma$, \\ $\gamma\alpha$ \\ \\ \\ $\alpha\beta$,\end{tabular} & 
\begin{tabular}{l}\\
$(\beta_0+\beta_1)\gamma$ \\ \\ $\frac{1}{2}(\gamma\alpha+u^{-1}\gamma\alpha u)$ \\ 
$\frac{1}{2}(\gamma\alpha-u^{-1}\gamma\alpha u)$, \\ 
$\alpha(\beta_0+\beta_1)$\end{tabular} \\
\hline
\end{tabular}}
       \captionof{table}{Jacobian blocks 6 to 10 \label{table-Jacobian-blocks-6-to-10}}
    \end{landscape}
    \clearpage
}

\vspace{1mm}

\subsection{Jacobian blocks}
\label{subsec:Jacobian-blocks}
\,

In Tables \ref{table-Jacobian-blocks-1-to-5} and \ref{table-Jacobian-blocks-6-to-10} the reader can see ten Jacobian algebras of the form $\jacobalg{A(\inputtripleforblocks(\xi)),W(\inputtripleforblocks(\xi))}$, where: 
\begin{enumerate}\item $Q$ is one of the following $3$-vertex quivers
\begin{equation}\label{eq:3-vertex-quivers-for-blocks}
\xymatrix{
& 1 \ar[dr]^{\gamma} & \\
2 \ar[ur]^{\alpha} & & 3 \ar[ll]^{\beta}
}\qquad\qquad
\xymatrix{
& 1 \ar[dr]^{\gamma} & \\
2 \ar[ur]^{\alpha} & & 3 \ar@/_0.25pc/[ll]_{\beta_0} \ar@/^0.25pc/[ll]^{\beta_1}
}
\end{equation}
\item $\bldD=(d_{1},d_{2},d_{3})$ is a triple of integers with $d\coloneqq \lcm(d_{1},d_{2},d_{3})\in\{1,2,4\}$; 
\item $\xi:\{\alpha,\beta,\gamma\}\rightarrow\mathbb{Z}/2\mathbb{Z}$ is a function satisfying $\xi_\alpha+\xi_\beta+\xi_\gamma=0$, where we write $\beta=\{\beta_0,\beta_1\}$ if $Q$ is the quiver on the right in \eqref{eq:3-vertex-quivers-for-blocks};
\item to each vertex $i\in Q_0$ there is attached a field $F_i$ extracted from a degree-$d$ datum $E/F$ and satisfying $[F_i:F]=d_i$;
\item the bimodules $A(Q,\mathbf{d},g(\xi))_a\coloneqq F_{h(a)}^{g_a(\xi)}\otimes_{F_{h(a)}\cap F_{t(a)}}F_{t(a)}$ arise from a modulating function $g(\xi)=(g_a(\xi))_{a\in Q_1}\in\times_{a\in Q_1}\Gal(F_{h(a)}\cap F_{t(a)}/F)$ defined in terms of the cocycle $\xi$ above;
\item $W(\inputtripleforblocks(\xi))$ is a potential defined using the guidelines from \cite{GLF2} for Blocks 1--7 (resp. from \cite{GLF1} for Blocks 8, 9 and 10), and the cyclic derivatives are computed via Definition \ref{def:paths-potentials-cyclic-derivatives-jacobian-algebras}-\eqref{item:def-cyclic-derivative}.
\end{enumerate}

For Block 5, the elements $\rho^l,\rho^{l+2}\in \Gal(E/F)$ are defined so that $\rho^l|_L=\theta^{\xi_\beta}=\rho^{l+2}|_L\in \Gal(L/F)$. 


\begin{lemma}
\label{lem-quotient-of-completed-path-algebra-is-f.d}
For each of the Jacobian blocks in Tables \ref{table-Jacobian-blocks-1-to-5} and \ref{table-Jacobian-blocks-6-to-10}, the canonical inclusion $\usualRA{A(Q,\bldD,\xi)}\hookrightarrow\compRA{A(Q,\bldD,\xi)}$ induces an $F$-algebra isomorphism
$
\usualRA{A(Q,\bldD,\xi)}/J_0(W(Q,\bldD,\xi))
\rightarrow 
\jacobalg{Q,\bldD,\xi}
$
that acts as the identity on $R$. In particular, $\dim_F(\jacobalg{Q,\bldD,\xi})<\infty$.
\end{lemma}

\begin{proof}

For each $i=1,\dots,10$ let $P$ denote the set of paths of length $4$ in the quiver $Q$ associated to Jacobian block $i$. 
So every path of length at least $4$ factors through an element in $P$. 
We claim that, when considering $P$ as a subset of the bimodule $A(Q,\bldD,\xi)$, it must be contained in $J(W(Q,\bldD,\xi))$. 
    Note that the asserted isomorphism then immediately follows from Remark \ref{remark-properties-of-jacobian-blocks}(3). 
Note that for any $i$ each element of $P$ factors through the path $\gamma\alpha$. 
So the claim holds for $i=1,3,4,6,7,8,9,10$. 
Likewise, when $i=2$, each element of $P$ factors through $\alpha\beta$, and so again the claim holds. 
Thus we now just consider the case where $i=5$.

In $A(Q,\bldD,\xi)$ we have
\begin{equation}
\label{eqn-useful-for-2-proofs}
    \beta_0=\frac{1}{2}\left((\beta_0+\beta_1) + \zeta^{-l} v^{-1}e_2(\beta_0+\beta_1) ve_3 \right),\quad \text{and} \quad \beta_1=\frac{1}{2}\left((\beta_0+\beta_1) + \zeta^{-l-2} v^{-1}e_2(\beta_0+\beta_1) ve_3 \right),
\end{equation}
 After right multiplication by $\gamma$, it follows that $\beta_{0}\gamma,\beta_{1}\gamma\in J(W(Q,\bldD,\xi))$, and since every element of $P$ factors through either $\beta_{0}\gamma$ or $\beta_{1}\gamma$, the claim follows. 
\end{proof}

\vspace{1mm}

\subsection{Semilinear clannish blocks}\label{subsec:semilinear-clannish-blocks}

\,

Next, we present ten semilinear clannish blocks in Tables \ref{table-semilinear-clannish-blocks-1-to-5} and \ref{table-semilinear-clannish-blocks-6-to-10}. Just as in Subsection \ref{subsec:semilinear-clannish-algs-def-of-concept}, we have tried to adapt as much as possible to the notation in \cite{BTCB}.  
Specifically, in each column of Tables \ref{table-semilinear-clannish-blocks-1-to-5} and \ref{table-semilinear-clannish-blocks-6-to-10} we construct an  $F$-algebra $K_{\bldsigma}\compactQhat /I$, where the field $K$ may be $F$ or $L$ depending on the block, the quiver $\compactQhat$ has three vertices, and when $K=L$ (resp. when $K=F$), the function $\bldsigma:\compactQhat_{1}\rightarrow\Gal(L/F)\subseteq\Aut(L)$ satisfies $\sigma_\alpha\circ\sigma_\beta\circ\sigma_\gamma=\myid_{L}$ (resp. $\sigma_\alpha=\sigma_\beta=\sigma_\gamma=\myid_F$).
According to Example \ref{ex-main-example-for-semilinear-clannish-blocks}, $K_{\bldsigma}\compactQhat /I$ is a semilinear clannish algebra that turns out to be normally-bound, non-singular, and of semisimple type.

\afterpage{%
    \clearpage
    \thispagestyle{empty}
    \begin{landscape}
        \centering 
        \hspace{-1cm}
\renewcommand{\arraystretch}{1.25}
{\small
\begin{tabular}{|c|c|c|c|c|c|}
\hline
 & Block 1 & Block 2 & Block 3 & Block 4 & Block 5 \\
\hline
\begin{tabular}{c}
Weight triple\\ 
$\left(\begin{array}{ccc} & \widehat{d} &\\ \widehat{d} & & \widehat{d}\end{array}\right)$\end{tabular} & 
$\left(\begin{array}{ccc} & \circled{2} &\\ \circled{2} & & \circled{2} \end{array}\right)$&
$\left(\begin{array}{ccc} & 2 &\\ \circled{2} & & \circled{2} \end{array}\right)$
&
$\left(\begin{array}{ccc} & 2 &\\ \circled{2} & & \circled{2}\end{array}\right)$ & $\left(\begin{array}{ccc} & \circled{2} &\\ 2 & & 2\end{array}\right)$ & $\left(\begin{array}{ccc} & \circled{2} &\\ 2 & & 2\end{array}\right)$ \\
\hline
\begin{tabular}{c}
Ordinary quiver\\
$\compactQhat$
\end{tabular}
&
$\xymatrix@C=1.2em@R=4.5em{
& 1 \ar[dr]^{\gamma} & \\
2 \ar[ur]^{\alpha} & & 3, \ar[ll]^{\beta} 
}$
&
$\xymatrix@C=1.2em@R=4.5em{
 & 1 \ar@(lu,ru)^{s_1} \ar[dr]^{\gamma} &\\
 2 \ar[ur]^{\alpha} & & 3\ar[ll]^{\beta}
}$
&
$\xymatrix@C=1.2em@R=4.5em{
 & 1 \ar@(lu,ru)^{s_1} \ar[dr]^{\gamma} &\\
 2 \ar[ur]^{\alpha} & & 3\ar[ll]^{\beta}
}$
&
$\xymatrix@C=1.2em@R=4.5em{
 & 1  \ar[dr]^{\gamma} &  \\
 2 \ar@(dr,dl)^{s_2} \ar[ur]^{\alpha} & & 3 \ar@(dr,dl)^{s_3} \ar[ll]^{\beta} 
}$
&
$\xymatrix@C=1.2em@R=4.5em{
 & 1  \ar[dr]^{\gamma} &\\
 2 \ar@(dr,dl)^{s_2} \ar[ur]^{\alpha} & & 3 \ar@(dr,dl)^{s_3} \ar[ll]^{\beta}
}$
\\
\hline
Special loop set $\bbS$
&
$\varnothing$
&
$\{s_1\}$
&
$\{s_1\}$
&
$\{s_2,s_3\}$
&
$\{s_2,s_3\}$
\\
\hline
Field $K$
&
$L$
&
$L$
& 
$L$
& 
$L$
& 
$L$
\\
\hline
\begin{tabular}{c}
Field automorphisms\\
$\sigma_a\in\Aut(K)$\\ for $a\in\compactQhat_1$
\end{tabular} 
&
\begin{tabular}{l}
$\sigma_\alpha=\theta^{\xi_\alpha}$\\
$\sigma_\beta=\theta^{\xi_\beta}$\\
$\sigma_\gamma=\theta^{\xi_\gamma}$
\end{tabular}
&
\begin{tabular}{ll}
$\sigma_\alpha=\theta^{-\xi_\beta}$ & $\sigma_{s_1}=\theta$ \\ $\sigma_\beta=\theta^{\xi_\beta}$
\\ $\sigma_\gamma=\myid_L$ 
\end{tabular} 
&
\begin{tabular}{ll}
$\sigma_\alpha=\theta^{\xi_\alpha}$ & $\sigma_{s_1}=\myid_L$\\  
$\sigma_\beta=\theta^{\xi_\beta}$\\ $\sigma_\gamma=\theta^{\xi_\gamma}$
\end{tabular}
&
\begin{tabular}{ll}
$\sigma_\alpha=\myid_L$\\ $\sigma_\beta=\myid_L$& $\sigma_{s_2} =\theta$\\ $\sigma_\gamma=\myid_L$ & $\sigma_{s_3}=\theta$
\end{tabular}
& 
\begin{tabular}{ll}
$\sigma_\alpha=\theta^{\xi_\alpha}$\\
$\sigma_\beta=\theta^{\xi_\beta}$ & $\sigma_{s_2} =\myid_L$ \\ $\sigma_\gamma=\theta^{\xi_\gamma}$ & $\sigma_{s_3}=\myid_L$
\end{tabular}
\\
\hline
\begin{tabular}{c}
Arrow bimodules\\
$ K^{\sigma_a}\otimes_{K}K$\\ for $a\in\compactQhat_1$
\end{tabular}
&
\begin{tabular}{c}
$L^{\sigma_\alpha}\smallotimesL L$\\ $L^{\sigma_\beta}\smallotimesL L$
\\ $L^{\sigma_\gamma}\smallotimesL L$
\end{tabular}
&
\begin{tabular}{cc}
$L^{\sigma_\alpha}\smallotimesL L$ & $L^{\sigma_{s_1}}\smallotimesL L$\\ $L^{\sigma_\beta}\smallotimesL L$
\\ $L^{\sigma_\gamma}\smallotimesL L$ 
\end{tabular}
& 
\begin{tabular}{cc}
$L^{\sigma_\alpha}\smallotimesL L$ & $L^{\sigma_{s_1}}\smallotimesL L$ \\ $L^{\sigma_\beta}\smallotimesL L$ \\
$ L^{\sigma_\gamma}\smallotimesL L$
\end{tabular}
& 
\begin{tabular}{cc} $L^{\sigma_\alpha}\smallotimesL L$ \\
$L^{\sigma_\beta}\smallotimesL L$ & $L^{\sigma_{s_2}}\smallotimesL L$\\
$L^{\sigma_\gamma}\smallotimesL L$ & 
$L^{\sigma_{s_3}}\smallotimesL L$
\end{tabular}
&
\begin{tabular}{cc} $ L^{\sigma_\alpha}\smallotimesL L$\\
$ L^{\sigma_\beta}\smallotimesL L$ & $L^{\sigma_{s_2}}\smallotimesL L$\\
$L^{\sigma_\gamma}\smallotimesL L$ &
$L^{\sigma_{s_3}}\smallotimesL L$
\end{tabular}
\\
\hline
Mnemotechnics & 
$\xymatrix@C=1.2em@R=4.5em{& L \ar[dr]_(.6){\gamma}^(.4){L^{\theta^{\xi_\gamma}}\smallotimesL L} & \\
L \ar[ur]_(.4){\alpha}^(.6){L^{\theta^{\xi_\alpha}}\smallotimesL L} & & L \ar[ll]_{\beta}^{L^{\theta^{\xi_\beta}}\smallotimesL L}}$
&
$\xymatrix@C=1.2em@R=4.5em{
& L \ar@(lu,ru)_{s_1}^{L^\theta\smallotimesL L} \ar[dr]_(.6){\gamma}^(.4){L\smallotimesL L} & \\
L \ar[ur]_(.4){\alpha}^(.6){L^{\theta^{-\xi_\beta}}\smallotimesL L} & & L \ar[ll]_{\beta}^{L^{\theta^{\xi_\beta}}\smallotimesL L} 
}$
&
$\xymatrix@C=1.2em@R=4.5em{
& L \ar@(lu,ru)_{s_1}^{L\smallotimesL L} \ar[dr]_(.6){\gamma}^(.4){L^{\theta^{\xi_\gamma}}\smallotimesL L} & \\
L \ar[ur]_(.4){\alpha}^(.6){L^{\theta^{\xi_\alpha}}\smallotimesL L} & & L \ar[ll]_{\beta}^{L^{\theta^{\xi_\beta}}\smallotimesL L} 
}$
 &
 $\xymatrix@C=1.7em@R=4.5em{
 & L  \ar[dr]_(.6){\gamma}^(.4){L\smallotimesL  L} &  \\
 L \ar@(dr,dl)_{s_2}^{L^{\theta}\smallotimesL  L} \ar[ur]_(.4){\alpha}^(.6){L\smallotimesL L} & & L \ar@(dr,dl)_{s_3}^{L^\theta\smallotimesL L} \ar[ll]_{\beta}^{L\smallotimesL L} 
}$
 &
$\xymatrix@C=1.7em@R=4.5em{
 & L  \ar[dr]_(.6){\gamma}^(.4){L^{\theta^{\xi_\gamma}}\smallotimesL  L} &  \\
 L \ar@(dr,dl)_{s_2}^{L\smallotimesL  L} \ar[ur]_(.4){\alpha}^(.6){L^{\theta^{\xi_\alpha}}\smallotimesL L} & & L \ar@(dr,dl)_{s_3}^{L\smallotimesL L} \ar[ll]_{\beta}^{L^{\theta^{\xi_\beta}}\smallotimesL L} 
}$ 
\\
\hline
Ideal $I=\langle Z\cup S\rangle\begin{array}{c}Z\\S
\end{array}$
& 
\begin{tabular}{r}
$\{\alpha\beta,\beta\gamma,\gamma\alpha\}$\\
$\emptyset\quad\quad$
\end{tabular}
&
\begin{tabular}{r}
$\{\alpha\beta,\beta\gamma,\gamma\alpha\}$\\ $\{s_1^2-e_1\}$
\end{tabular}
& 
\begin{tabular}{r}
$\{\alpha\beta,\beta\gamma,\gamma\alpha\}$\\ $\{s_1^2-ue_1\}$
\end{tabular}
& 
\begin{tabular}{r}
$\{\alpha\beta,\beta\gamma,\gamma\alpha\}\quad$\\ $\{s_2^2-e_2,s_3^2-e_3\}$
\end{tabular}
& 
\begin{tabular}{r}
$\{\alpha\beta,\beta\gamma,\gamma\alpha\}\quad$\\ $\{s_2^2-ue_2,s_3^2-ue_3\}$
\end{tabular}\\
\hline
\end{tabular}}
       \captionof{table}{Semilinear clannish blocks 1 to 5 \label{table-semilinear-clannish-blocks-1-to-5}}
    \end{landscape}
    \clearpage
}

\afterpage{%
    \clearpage
    \thispagestyle{empty}
    \begin{landscape}
        \centering 
                \hspace{-0.75cm}
\renewcommand{\arraystretch}{1.25}
{\small
\begin{tabular}{|c|c|c|c|c|c|}
\hline
& Block 6 & Block 7 & Block 8 & Block 9 & Block 10 \\
\hline
\begin{tabular}{c}
Weight triple\\ 
$\left(\begin{array}{ccc} & \widehat{d} &\\ \widehat{d} & & \widehat{d}\end{array}\right)$\end{tabular} 
& 
$\left(\begin{array}{ccc} & \circled{2} &\\ 2 & & 2\end{array}\right)$
&
$\left(\begin{array}{ccc} & \circled{2} &\\ 2 & & 2\end{array}\right)$
&
$\left(\begin{array}{ccc} & \circled{1} &\\ \circled{1} & & \circled{1}\end{array}\right)$ 
& 
$\left(\begin{array}{ccc} & 1 &\\ \circled{1} & & \circled{1}\end{array}\right)$ 
& 
$\left(\begin{array}{ccc} & \circled{1} &\\ 1 & & 1\end{array}\right)$ \\
\hline
\begin{tabular}{c}
Ordinary quiver\\
$\compactQhat $
\end{tabular}
&
$\xymatrix@C=1.2em@R=4.5em{
 & 1  \ar[dr]^{\gamma} &\\
 2 \ar@(dr,dl)^{s_2} \ar[ur]^{\alpha} & & 3 \ar@(dr,dl)^{s_3} \ar[ll]^{\beta}
}$
&
$\xymatrix@C=1.2em@R=4.5em{
 & 1  \ar[dr]^{\gamma} &\\
 2 \ar@(dr,dl)^{s_2} \ar[ur]^{\alpha} & & 3 \ar@(dr,dl)^{s_3} \ar[ll]^{\beta}
}$
&
$\xymatrix@C=1.2em@R=4.5em{
& 1 \ar[dr]^{\gamma} & \\
2 \ar[ur]^{\alpha} & & 3, \ar[ll]^{\beta} 
}$
&
$\xymatrix@C=1.2em@R=4.5em{
 & 1 \ar@(lu,ru)^{s_1} \ar[dr]^{\gamma} &\\
 2 \ar[ur]^{\alpha} & & 3\ar[ll]^{\beta}
}$
&
$\xymatrix@C=1.2em@R=4.5em{
 & 1  \ar[dr]^{\gamma} &\\
 2 \ar@(dr,dl)^{s_2} \ar[ur]^{\alpha} & & 3 \ar@(dr,dl)^{s_3} \ar[ll]^{\beta}
}$
\\
\hline
Special loop set $\bbS$
&
$\{s_2,s_3\}$
&
$\{s_2,s_3\}$
&
$\varnothing$
&
$\{s_1\}$
&
$\{s_2,s_3\}$
\\
\hline
Field $K$
&
$L$
&
$L$
& 
$F$
& 
$F$
& 
$F$
\\
\hline
\begin{tabular}{c}
Field automorphisms\\
$\sigma_a\in\Aut(K)$\\ for $a\in\compactQhat_1$
\end{tabular} 
&
\begin{tabular}{ll}
$\sigma_\alpha=\theta^{\xi_\alpha}$ &  \\ $\sigma_\beta=\myid_L$ & $\sigma_{s_2}=\myid_L$
\\ $\sigma_\gamma=\theta^{-\xi_\alpha}$ & $\sigma_{s_3}=\theta$
\end{tabular} 
&
\begin{tabular}{ll}
$\sigma_\alpha=\theta^{-\xi_\gamma}$ &  \\ $\sigma_\beta=\myid_L$ & $\sigma_{s_2}=\theta$
\\ $\sigma_\gamma=\theta^{\xi_\gamma}$ & $\sigma_{s_3}=\myid_L$
\end{tabular} 
&
\begin{tabular}{l}
$\sigma_\alpha=\myid_F$ \\  
$\sigma_\beta=\myid_F$\\ $\sigma_\gamma=\myid_F$
\end{tabular}
&
\begin{tabular}{ll}
$\sigma_\alpha=\myid_F$ & $\sigma_{s_1}=\myid_F$ \\  
$\sigma_\beta=\myid_F$ & \\ $\sigma_\gamma=\myid_F$ &
\end{tabular}
& 
\begin{tabular}{ll}
$\sigma_\alpha=\myid_F$ &  \\ $\sigma_\beta=\myid_F$ & $\sigma_{s_2}=\myid_F$
\\ $\sigma_\gamma=\myid_F$ & $\sigma_{s_3}=\myid_F$
\end{tabular} 
\\
\hline
\begin{tabular}{c}
Arrow bimodules\\
$ K^{\sigma_a}\otimes_KK$\\ for $a\in\compactQhat_1$
\end{tabular} 
&
\begin{tabular}{ll}
$L^{\sigma_\alpha}\smallotimesL L$ &  \\ $L^{\sigma_\beta}\smallotimesL L$ & $L^{\sigma_{s_2}}\smallotimesL L$
\\ $L^{\sigma_\gamma}\smallotimesL L$ & $L^{\sigma_{s_3}}\smallotimesL L$
\end{tabular} 
&
\begin{tabular}{ll}
$L^{\sigma_\alpha}\smallotimesL L$ &  \\ $L^{\sigma_\beta}\smallotimesL L$ & $L^{\sigma_{s_2}}\smallotimesL L$
\\ $L^{\sigma_\gamma}\smallotimesL L$ & $L^{\sigma_{s_3}}\smallotimesL L$
\end{tabular} 
&
\begin{tabular}{l}
$F^{\sigma_\alpha}\smallotimesF F$ \\  
$F^{\sigma_\beta}\smallotimesF F$\\ $F^{\sigma_\gamma}\smallotimesF F$
\end{tabular}
&
\begin{tabular}{ll}
$F^{\sigma_\alpha}\smallotimesF F$ & $F^{\sigma_{s_1}}\smallotimesF F$ \\  
$F^{\sigma_\beta}\smallotimesF F$ & \\ $F^{\sigma_\gamma}\smallotimesF F$ &
\end{tabular}
& 
\begin{tabular}{ll}
$F^{\sigma_\alpha}\smallotimesF F$ &  \\ $F^{\sigma_\beta}\smallotimesF F$ & $F^{\sigma_{s_2}}\smallotimesF F$
\\ $F^{\sigma_\gamma}\smallotimesF F$ & $F^{\sigma_{s_3}}\smallotimesF F$
\end{tabular} 
\\
\hline
Mnemotechnics & 
$\xymatrix@C=1.2em@R=4.5em{
 & L  \ar[dr]_(.6){\gamma}^(.4){L^{\theta^{-\xi_\alpha}}\smallotimesL  L} &  \\
 L \ar@(dr,dl)_{s_2}^{L\smallotimesL  L} \ar[ur]_(.4){\alpha}^(.6){L^{\theta^{\xi_\alpha}}\smallotimesL L} & & L \ar@(dr,dl)_{s_3}^{L^{\theta}\smallotimesL L} \ar[ll]_{\beta}^{L\smallotimesL L} 
}$
&
$\xymatrix@C=1.2em@R=4.5em{
 & L  \ar[dr]_(.6){\gamma}^(.4){L^{\theta^{\xi_\gamma}}\smallotimesL  L} &  \\
 L \ar@(dr,dl)_{s_2}^{L^\theta\smallotimesL  L} \ar[ur]_(.4){\alpha}^(.6){L^{\theta^{-\xi_\gamma}}\smallotimesL L} & & L \ar@(dr,dl)_{s_3}^{L\smallotimesL L} \ar[ll]_{\beta}^{L\smallotimesL L} 
}$
&
$\xymatrix@C=1.2em@R=4.5em{
& F \ar[dr]_(.6){\gamma}^(.4){F\smallotimesF F} & \\
F \ar[ur]_(.4){\alpha}^(.6){F\smallotimesF F} & & F \ar[ll]_{\beta}^{F\smallotimesF F} 
}$
 &
$\xymatrix@C=1.2em@R=4.5em{
& F \ar@(lu,ru)_{s_1}^{F\smallotimesF F} \ar[dr]_(.6){\gamma}^(.4){F\smallotimesF F} & \\
F \ar[ur]_(.4){\alpha}^(.6){F\smallotimesF F} & & F \ar[ll]_{\beta}^{F\smallotimesF F} 
}$
 &
$\xymatrix@C=1.2em@R=4.5em{
 & F  \ar[dr]_(.6){\gamma}^(.4){F\smallotimesF F} &  \\
 F \ar@(dr,dl)_{s_2}^{F\smallotimesF F} \ar[ur]_(.4){\alpha}^(.6){F\smallotimesF F} & & F \ar@(dr,dl)_{s_3}^{F\smallotimesF F} \ar[ll]_{\beta}^{F\smallotimesF F} 
}$ 
\\
\hline
Ideal $I=\langle Z\cup S\rangle\begin{array}{c}Z\\S
\end{array}$
&
\begin{tabular}{r}
$\{\alpha\beta,\beta\gamma,\gamma\alpha\}\quad$\\  $\{s_2^2-ue_2,s_3^2-e_3\}$
\end{tabular}
& 
\begin{tabular}{r}
$\{\alpha\beta,\beta\gamma,\gamma\alpha\}\quad$\\  $\{s_2^2-e_2,s_3^2-ue_3\}$
\end{tabular}
& 
\begin{tabular}{r}
$\{\alpha\beta,\beta\gamma,\gamma\alpha\}$\\
$\emptyset\quad\quad\quad$
\end{tabular}
& 
\begin{tabular}{r}
$\{\alpha\beta,\beta\gamma,\gamma\alpha\}$\\  $\{s_1^2-u^2e_1\}$
\end{tabular}
& 
\begin{tabular}{r}
$\{\alpha\beta,\beta\gamma,\gamma\alpha\}\quad$\\  $\{s_2^2-u^2e_2, s_3^2-u^2e_3\}$
\end{tabular}\\
\hline
\end{tabular}}
 \captionof{table}{Semilinear clannish blocks 6 to 10 \label{table-semilinear-clannish-blocks-6-to-10}}
    \end{landscape}
    \clearpage
}

\vspace{1mm}
\subsection{Morita equivalences for blocks}
\label{subsec:Morita-equiv-for-blocks}
\,

\begin{prop}\label{prop-isomorphisms-between-some-of-the-blocks}    
For $k\in\{1,3,5,8,9,10\}$ there is an $F$-algebra isomorphism between the $k^{\operatorname{th}}$ Jacobian block from Tables \ref{table-Jacobian-blocks-1-to-5} and \ref{table-Jacobian-blocks-6-to-10} and the $k^{\operatorname{th}}$ semilinear clannish block from Tables \ref{table-semilinear-clannish-blocks-1-to-5} and  \ref{table-semilinear-clannish-blocks-6-to-10}.
\end{prop}
\begin{proof}
The cases $k=1,3,8,9$ are straightforward. 
For $k=1,8$, one is considering tensor rings for the same species. 
For $k=3$ one can use the $F$-algebra isomorphism $E\cong L[x]/(x^{2}-u)$. 
For  $k=9$ one can use the $F$-algebra isomorphism $L=F(u)\cong F[x]/(x^{2}-u^{2})$.
The cases $k=5,10$ are more difficult. 
We exhibit an explicit isomorphism between the $5^{\operatorname{th}}$ Jacobian and semilinear clannish blocks. 
The treatment of the $10^{\operatorname{th}}$ blocks is similar and somewhat simpler, so we leave it in the reader's hands.

For the $5^{\operatorname{th}}$ semilinear clannish block we are taking $K=L$. (For the $10^{\operatorname{th}}$ block one takes $K=F$.) Notice that $K_{\bldsigma}\compactQhat /I$ has a natural $R$-$R$-bimodule structure with $R\coloneqq F_1\times F_2\times F_3=L\times E\times E$, extending its natural $S$-$S$-bimodule structure with $S\coloneqq L\times L\times L$. 
Here the left and right actions of the element $ve_2=(0,v,0)$ (resp. $ve_3=(0,0,v)$) of $R$ on $K_{\bldsigma}\compactQhat /I$ are respectively given by left and right multiplications by $s_2$ (resp. $s_3$). 
This uses the $F$-algebra isomorphism $E\cong L[x]/(x^{2}-u)$ together with the fact that $\sigma_{s_{2}}=\sigma_{s_{3}}=\myid_{L}$ for this block. 
Furthermore, the assignment
$$
ve_2\mapsto s_2,\quad ve_3\mapsto s_3,\quad
\alpha \mapsto \alpha,\quad \beta_0\mapsto \frac{1}{2}(\beta + (\zeta^{l} u)^{-1} s_2\beta s_3 ),\quad \beta_1\mapsto \frac{1}{2}(\beta + (\zeta^{l+2} u)^{-1} s_2\beta s_3 ),\quad \gamma\mapsto \gamma,
$$
extends uniquely to an $R$-$R$-bimodule homomorphism
$\varphi:A(Q,\bldD,\xi)\rightarrow K_{\bldsigma}\compactQhat /I$, where: on the one hand, we identify each $a\in\{\alpha,\beta_0,\beta_1,\gamma\}=Q_{1}$ with the element $1\otimes 1$ of the summand $A(Q,\bldD,\xi)_a$ of $A(Q,\bldD,\xi)$; and, on the other hand, we identify each $b\in\{\alpha,\beta,\gamma,s_{2},s_{3}\}=\compactQhat_{1}$ with the coset (modulo $I=\langle Z\cup S\rangle$) represented by the element $1\otimes 1$ of the summand $L^{\sigma_b}\otimes L$ of the arrow bimodule ${}_{\pi_{h(b)}} K_{\sigma_b \pi_{t(b)}}$ associated to the arrow $b$. For instance, since $v^{2}=u$, $\rho^l|_{L}=\theta^{\xi_\beta}$, $\rho^l(v)=\zeta^{ l} v$, $\sigma_\beta=\theta^{\xi_\beta}$ and $\sigma_{s_2}=\myid_L$, the fact that $\varphi$ can be defined as an $R$-$R$-bimodule homomorphism on the whole direct summand $A(Q,\bldD,\xi)_{\beta_0}$ of $A(Q,\bldD,\xi)$ follows from the following computation:
$$
\begin{array}{lclclcl}
\varphi(\beta_0 ve_3) &=& 
\frac{1}{2}(\beta + (\zeta^{l} u)^{-1} s_2\beta s_3 ) s_3 
&=& 
\frac{1}{2}(\beta s_3 + (\zeta^{l} u)^{-1} s_2\beta ue_3 )
&=&
\frac{1}{2}(\beta s_3 + (\zeta^{l} u)^{-1} s_2 \theta^{\xi_\beta}(u)\beta  )\\
&=&
\frac{1}{2}(\beta s_3 + (\zeta^{l} u)^{-1} \theta^{\xi_\beta}(u)s_2 \beta  ) 
&=&
\frac{1}{2}(\beta s_3 + (\zeta^{l} u)^{-1} \zeta^{2l}u s_2 \beta  ) 
&=& 
\frac{1}{2}(\beta s_3 + \zeta^{l}  s_2 \beta  )\\
&=&\frac{\zeta^{l}s_2}{2}(\zeta^{-l}u^{-1}s_2\beta s_3 +   \beta  )
&=&\varphi(\rho^l(v)e_2\beta_0).
\end{array}
$$
Thus, $\varphi$ induces a well-defined ring homomorphism 
\[
\begin{array}{cc}
\varphi:\usualRA{A(Q,\bldD,\xi)}\rightarrow K_{\bldsigma}\compactQhat/I,     & (I\coloneqq \langle \alpha\beta,\beta\gamma,\gamma\alpha,s_{2}^{2}-ue_{2},s_{3}^{2}-ue_{3}\rangle)
\end{array}
\] 
(we use the same letter $\varphi$ in order to avoid making the notation even heavier) which is $F$-linear and an $R$-$R$-bimodule homomorphism. In particular $\varphi$ sends the trivial path $e_{i}$ in the path algebra $\usualRA{A(Q,\bldD,\xi)}$ of $Q$ to the trivial path $e_i$ in the semilinear path algebra $K_{\bldsigma}\compactQhat$  (see \S\ref{subsubsec:species-def-of-concept} for the definition of trivial path). We claim that $\varphi$ is surjective. Indeed, since $\zeta^4=1$ and $\zeta^2\neq 1$, we have $\zeta^2+1=0$, hence $\beta=\varphi(\beta_0+\beta_1)$.  So, each arrow in $\compactQhat $ lies in the image of $\varphi$. Furthermore, $S\coloneqq 
L\times L\times L$ is an $L$-subalgebra of $R=F_1\times F_2\times F_3=L\times E\times E$, so the image of the coproduct 
\[
\begin{array}{c}
A(\compactQhat,\bldsigma)=\bigoplus_{b\in\widehat{Q}_1}{}_{\pi_{h(b)}} K_{\sigma_b \pi_{t(b)}}
\end{array}
\] 
under the projection $K_{\bldsigma}\compactQhat\rightarrow K_{\bldsigma}\compactQhat/I$ is contained in the image of $\varphi$. Hence the whole $K_{\bldsigma}\compactQhat/I$ is contained in the image of $\varphi$. 

Next, the cyclic derivatives of $W(Q,\bldD,\xi)=\alpha(\beta_{0}+\beta_{1})\gamma$ are contained in $\ker(\varphi)$. Indeed, with the aid of the last columns of Tables \ref{table-Jacobian-blocks-1-to-5} and \ref{table-semilinear-clannish-blocks-1-to-5}, we see that
\begin{align*}
    \varphi(\partial_\alpha(W(Q,\bldD,\xi))) &= \beta\gamma \in I, &
    \varphi(\partial_{\beta_0}(W(Q,\bldD,\xi))) &= 
    \frac{1}{2}(\gamma\alpha+\zeta^{l}u^{-1}s_2\gamma\alpha s_3)\in I,\\
    \varphi(\partial_\gamma(W(Q,\bldD,\xi)))&=\alpha\beta\in I, & 
    \varphi(\partial_{\beta_1}(W(Q,\bldD,\xi)))&= \frac{1}{2}(\gamma\alpha+\zeta^{l+2}u^{-1}s_2\gamma\alpha s_3)\in I.
\end{align*}
So, we have an induced surjective ring homomorphism 
$$\varphi:\usualRA{A(Q,\bldD,\xi)}/J_{0}(W(Q,\bldD,\xi))\rightarrow K_{\bldsigma}\compactQhat/I$$
(we again use the same letter $\varphi$ in order to avoid making the notation heavier) which is an $F$-linear $R$-$R$-bimodule homomorphism.

On other hand, the assignments
$$
s_2\mapsto ve_2,\quad s_3\mapsto ve_3,\quad
\alpha \mapsto \alpha,\quad \beta\mapsto \beta_0 + \beta_1,\quad \gamma\mapsto \gamma,
$$ 
extend 
uniquely to an $S$-$S$-bimodule homomorphism 
$\psi:\bigoplus_{b\in\widehat{Q}_1}{}_{\pi_{h(b)}} K_{\sigma_b \pi_{t(b)}}\rightarrow \usualRA{A(Q,\bldD,\xi)}$,
where we are using the same identifications of arrows with elements of the direct summands ${}_{\pi_{h(b)}} K_{\sigma_b \pi_{t(b)}}$ and $A(Q,\bldD,\xi)_a$ as above. For instance, the fact that $\psi$ can be defined on the whole direct summand ${}_{\pi_{h(\beta)}} K_{\sigma_\beta \pi_{t(\beta)}}$ as an $S$-$S$-bimodule homomorphism 
follows from the fact that $\rho^l|_{L}=\theta^{\xi_\beta}=\rho^{l+2}|_L$ and 
$$
\begin{array}{lclcl}
\psi(\beta \ell e_3) &=& (\beta_0 + \beta_1)\ell e_3
& = & \rho^l(\ell)e_2\beta_0+\rho^{l+2}(\ell)e_2\beta_1
\\
&=&
\theta^{\xi_\beta}(\ell)e_2(\beta_0+\beta_1)
& = &
\psi(\theta^{\xi_\beta}(\ell)e_2\beta).
\end{array}
$$
Thus, $\psi$ induces a well-defined ring homomorphism 
$$
\psi: K_{\bldsigma}\compactQhat\rightarrow \usualRA{A(Q,\bldD,\xi)}
$$
(we use the same letter $\psi$ in order to avoid making the notation even heavier) which is $F$-linear and an $S$-$S$-bimodule homomorphism. 

Recalling \eqref{eqn-useful-for-2-proofs} from the proof of Lemma \ref{lem-quotient-of-completed-path-algebra-is-f.d}, we have that  
$R$  and all the arrows of $Q$ are in the image of $\psi$. Thus, $A(Q,\bldD,\xi)$, and hence the whole $\usualRA{A(Q,\bldD,\xi)}$, are contained in the image of $\psi$. 
The image of $I\subseteq K_{\bldsigma}\compactQhat$ under $\psi$ is contained in the (incomplete) Jacobian ideal $J_0(W(Q,\bldD,\xi))$. Indeed, since $v^2=u$ and $\rho^{-l}(v^{-1})+\rho^{-l-2}(v^{-1})=\zeta^l(1+\zeta^{2})v^{-1}=0$ in $E$, with the aid of the last columns of Tables \ref{table-Jacobian-blocks-1-to-5} and \ref{table-semilinear-clannish-blocks-1-to-5}, we see that
\begin{center}
\begin{tabular}{ll}
     $\psi(\alpha\beta)=\partial_\gamma(W(Q,\bldD,\xi)),$ &  \\
     $\psi(\beta\gamma)=\partial_\alpha(W(Q,\bldD,\xi)),$ & $\psi(s_2^2-ue_2)=0,$\\
     $\psi(\gamma\alpha)=\partial_{\beta_0}(W(Q,\bldD,\xi))+\partial_{\beta_1}(W(Q,\bldD,\xi))$, &  $\psi(s_3^2-ue_3)=0$.
\end{tabular}
\end{center}
So, we have an induced surjective ring homomorphism 
$$\psi:K_{\bldsigma}\compactQhat/I\rightarrow  \usualRA{A(Q,\bldD,\xi)}/J_0(W(Q,\bldD,\xi))$$
(we again use the same letter $\psi$ in order to avoid making the notation heavier) which is $F$-linear and an $S$-$S$-bimodule homomorphism.

With the above considerations, one easily verifies that $\psi\circ\varphi$ and $\varphi\circ\psi$ act as the identity on specific sets that generate  $\usualRA{A(Q,\bldD,\xi)}/J_{0}(W(Q,\bldD,\xi))$ and $K_{\bldsigma}\compactQhat/I$ as $F$-algebras. This implies that $\varphi$ and $\psi$ are mutually inverse $F$-algebra  isomorphisms. This finishes the treatment of the $5^{\operatorname{th}}$ blocks.
\end{proof}

\begin{remark}\label{rem:ex-from-BTCB}
Notice that the isomorphism between the $10^{\operatorname{th}}$ Jacobian and semilinear clannish blocks is the one alluded to in \cite[\S5.4]{BTCB} (with $F=\mathbb{R}$ and $L=\mathbb{C}$ in \emph{loc. cit.}).
\end{remark}

\begin{remark}
    To motivate the statement of Proposition \ref{prop-morita-equivalences-for-blocks}, we explain why Proposition \ref{prop-isomorphisms-between-some-of-the-blocks} 
 does not extend to the cases of blocks $2$, $4$, $6$ and $7$. That is, we explain why, for each of these blocks, the Jacobian  algebra is not isomorphic to the semilinear clannish algebra. 
All of the rings we are considering are finite-dimensional over the field $F$, and hence they are artinian, and hence perfect, and hence semiperfect. 
Hence, for each said ring, the multiplicative identity can be written as a finite sum of primitive (and hence local) pairwise  orthogonal idempotents, and any such sum must have the same number of summands.

Note that each jacobian algebra modulo its radical is the product of exactly $3$ division rings. It follows that the multiplicative identity of each Jacobian algebra, in each case, is a sum of exactly $3$ primitive pairwise  orthogonal idempotents. 
Let $A=L_{\bldsigma}\compactQhat/I$ be any of the semilinear clannish blocks $2$, $4$, $6$ and $7$.
From what we have observed so far, it sufiices to find at least $4$ pairwise orthogonal idempotents in $A$.

Note that there exists some $i=1,2,3$ such that vertex $i$ in the quiver $\compactQhat$ has a special loop $s$ such that $q_{s}(x)=x^{2}-1$. 
Now let $e=e_{i}$. 
Recall that $L$ has characteristic different from $2$. 
Consider that, in $L_{\bldsigma}\compactQhat$,
\[
\begin{array}{cc}
(\frac{1}{2}(e\pm s))^{2}=\frac{1}{4}((e+s^{2})\pm 2s)=\frac{1}{2}(\frac{1}{2}(e+s^{2})\pm s),
&
\frac{1}{2}(e\pm s)\frac{1}{2}(e\mp s)=\frac{1}{4}(e-s^{2}).
\end{array}
\]

Since $q_{s}(x)=x^{2}-1$ we have $e-s^{2}\in I$, and so $e+s^{2}+I=2e+I$, and it follows that $e'+I$ and $e''+I$ are pairwise orthogonal idempotents in $A$ where $e'\coloneqq \frac{1}{2}(e+ s)$ and $e''\coloneqq \frac{1}{2}(e-s)$.  
Finally, writing $f$ and $f'$ for the trivial paths in $\compactQhat$ such that $\{e_{j}\mid j\neq i\}=\{f,f'\}$, observe that $f+I$, $f'+I$, $e'+I$ and $e''+I$ are pairwise orthogonal idempotents in $A$, since $ef=fe=ef'=f'e=sf=fs=s'f=f's=0$ in $L_{\bldsigma}\compactQhat$. 
\end{remark}

\begin{prop}
\label{prop-morita-equivalences-for-blocks}
For $k\in\{1,2,3,4,5,6,7,8,9,10\}$, the $k^{\operatorname{th}}$ Jacobian block from Tables \ref{table-Jacobian-blocks-1-to-5}, \ref{table-Jacobian-blocks-6-to-10}, and the $k^{\operatorname{th}}$ semilinear clannish block from Tables \ref{table-semilinear-clannish-blocks-1-to-5}, \ref{table-semilinear-clannish-blocks-6-to-10}, are Morita equivalent through an $F$-linear Morita equivalence.
\end{prop}

\begin{proof}  
By Proposition \ref{prop-isomorphisms-between-some-of-the-blocks}, it only remains to show that the $k^{\operatorname{th}}$ Jacobian and semilinear clannish blocks are Morita-equivalent for  $k=2, 4, 6, 7$. This can be done by a direct exhibition of functors 
\begin{align*}
\Psi&:\usualRA{A(Q,\bldD,\xi)}/J_0(W(Q,\bldD,\xi))\text{-}\mathbf{Mod} \rightarrow K_{\bldsigma}\compactQhat/I\text{-}\mathbf{Mod},\\
\Phi&:K_{\bldsigma}\compactQhat/I\text{-}\mathbf{Mod} \rightarrow \usualRA{A(Q,\bldD,\xi)}/J_0(W(Q,\bldD,\xi))\text{-}\mathbf{Mod},
\end{align*}
(recall that $K:=L$ for $k=2, 4, 6, 7$) and invertible natural transformations
\begin{align*}
\varepsilon&:\myid_{\usualRA{A(Q,\bldD,\xi)}/J_0(W(Q,\bldD,\xi))\text{-}\mathbf{Mod}}\rightarrow \Phi\circ\Psi,\\
\eta&:\myid_{K_{\bldsigma}\widehat{Q}/I\text{-}\mathbf{Mod}}\rightarrow \Psi\circ\Phi.
\end{align*}
By Corollary \ref{coro:rep-of-species-equiv-to-cat-of-semilinear-maps} and the discussion in \S\ref{subsubsec-representations-and-modules}, we can treat the two module categories involved as full subcategories of categories of semilinear representations of quivers. 
In Table \ref{table-Morita-equivalences-2467}, 
the reader can see the correspondence rules for $\Psi$ and $\Phi$ on objects. We shall write down in detail the correspondence rules for $\Psi$ and $\Phi$ on morphisms, as well as the correspondence rules of $\varepsilon$ and $\eta$, only in the case $k=6$. The cases $k=2,4,7,$ can be handled similarly, and are thus left in the reader's hands.

\afterpage{%
    \clearpage
    \thispagestyle{empty}
    \begin{landscape}
        \hspace{-3.5cm}
        \hspace{1.3cm}
\renewcommand{\arraystretch}{1.25}
{\tiny
\begin{tabular}{|c|c|c|c|c|}
\hline
& Blocks 2 & Blocks 4 & Blocks 6 & Blocks 7\\
\hline
$\xymatrix@C=0.01em@R=0.1em{& d_1 &\\ d_2 & & d_3}$
& 
$\xymatrix@C=0.01em@R=0.1em{& 1 &\\ \circled{2} & & \circled{2}}$
&
$\xymatrix@C=0.01em@R=0.1em{& \circled{2} &\\ 1 & & 1}$
 &
$\xymatrix@C=0.01em@R=0.1em{ & \circled{2} &\\ 4 & & 1}$&
$\xymatrix@C=0.01em@R=0.1em{& \circled{2} &\\ 1 & & 4}$ \\
\hline
\begin{tabular}{c}
Jacobian\\ block
\end{tabular} 
&
\begin{tabular}{c}
$\xymatrix@C=4em@R=3em
{
& F \ar@/^0.6pc/[dr]_(.4){\gamma}^(.4){L\smallotimesF F} & \\
L \ar@/^0.6pc/[ur]_(.6){\alpha}^(.6){F\smallotimesF L} & & L \ar@/_0.6pc/[ll]_{\beta}^{L^{\theta^{\xi_\beta}}\smallotimesL L} 
}$\\
$\beta\gamma$, $\alpha\beta$, $\frac{1}{2}(\gamma\alpha+(-1)^{\xi_\beta}u^{-1}\gamma\alpha u)$
\end{tabular}
&{
\begin{tabular}{c}
 $\xymatrix@C=8em@R=3em
 {
  & L  \ar@/^0.6pc/[dr]_(.4){\gamma}^(.4){F{\smallotimesF }L} &  \\
 F \ar@/^0.6pc/[ur]_(.6){\alpha}^(.6){L{\smallotimesF }F}  & &  F \ar@/_0.8pc/[ll]^{\beta_0}_{F{\smallotimesF }F} \ar@/^0.8pc/[ll]_{\beta_1}^{F{\smallotimesF }F}
}$\\
$\beta_0\gamma+\beta_1\gamma u$,
$\gamma\alpha$, $\gamma u\alpha$,
$\alpha\beta_0+u\alpha\beta_1$
\end{tabular}
}
&
\begin{tabular}{c}$\xymatrix@C=3.5em@R=3em{
 & L  \ar@/^0.6pc/[dr]_{\gamma}^{F\smallotimesF L} & \\
 E \ar@/^0.6pc/[ur]_{\alpha}^{L^{\theta^{\xi_\alpha}}\smallotimesL E} & & F \ar@/_0.6pc/[ll]_{\beta}^{E\smallotimesF F}
}$\\
$\frac{1}{2} (\beta\gamma + 
\theta^{-\xi_\alpha}(u^{-1})\beta\gamma u)$,
$\gamma\alpha$,
$\alpha\beta$
\end{tabular}
&
\begin{tabular}{c}
$\xymatrix@C=3.5em@R=3em{
 & L  \ar@/^0.6pc/[dr]_{\gamma}^{E^{\theta^{\xi_\gamma}}\smallotimesL L} & \\
 F \ar@/^0.6pc/[ur]_{\alpha}^{L\smallotimesF F} & & E \ar@/_0.6pc/[ll]_{\beta}^{F\smallotimesF E}
}$\\
$\beta\gamma$, $\gamma\alpha$,  $\frac{1}{2} (\alpha\beta + 
\theta^{-\xi_\gamma}(u^{-1})\alpha\beta u)$
\end{tabular}
\\
\hline
$\xymatrix@R=1.5em{
M  \ar@{|-}[d] 
\\ \,  
}$ 
& $
\xymatrix@C=4em@R=3em{
   & M_1  \ar@/^0.6pc/[dr]^{M_\gamma} & \\
 M_2 \ar@/^0.6pc/[ur]^{M_\alpha } & & M_3 \ar@/_0.6pc/[ll]^{M_{\beta}}}
$ & 
$\xymatrix@C=8em@R=3em{
  & M_1 \ar@/^0.6pc/[dr]^{M_\gamma} & \\
 M_2 \ar@/^0.6pc/[ur]^{M_\alpha } & & M_3 \ar@/_0.5pc/[ll]_{M_{\beta_0}} \ar@/^0.3pc/[ll]^{M_{\beta_1}}
}$
&
$
\xymatrix@C=3.5em@R=3em{
   & M_1  \ar@/^0.6pc/[dr]^{M_\gamma} & \\
 M_2 \ar@/^0.6pc/[ur]^{M_\alpha } & & M_3 \ar@/_0.6pc/[ll]^{M_{\beta}}
}
$
&
$
\xymatrix@C=3.5em@R=3em{
   & M_1  \ar@/^0.6pc/[dr]^{M_\gamma} & \\
 M_2 \ar@/^0.6pc/[ur]^{M_\alpha } & & M_3 \ar@/_0.6pc/[ll]^{M_{\beta}}
}
$
\\
\hline
$\xymatrix@R=2em{
\,  \ar@{->}[d]  
\\ \Psi(M) 
}$ 
& $
\xymatrix@C=4em@R=4em{
  & L\smallotimesF  M_1 \ar@(ld,rd)_{\theta\smallotimesempty\myid_{M_1}} \ar@/^0.6pc/[dr]^{\overleftarrow{M_\gamma}} &\\
 M_2 \ar@/^0.6pc/[ur]^{\overrightarrow{M_\alpha}} & & M_3\ar@/_0.6pc/[ll]^{M_\beta} 
}
$ & 
$\xymatrix@C=4.5em@R=4em{
  & M_1 \ar@/^0.6pc/[dr]^{\overrightarrow{M_\gamma}} & \\
 L\smallotimesF M_2 \ar@/^0.6pc/[ur]^{\overleftarrow{M_\alpha}} \ar@(u,l)_(.4){\theta\smallotimesempty\myid_{M_2}} & & L\smallotimesF M_3 \ar@/_0.6pc/[ll]^{\myid\smallotimesempty M_{\beta_0}+u\myid\smallotimesempty M_{\beta_1}} \ar@(u,r)^(.4){\theta\smallotimesempty  \myid_{M_{3}}}
}$
&
$
\xymatrix@C=3em@R=4em{
  & M_1  \ar@/^0.6pc/[dr]^(.45){(\theta^{-\xi_\alpha}\smallotimesempty\myid_{M_3})\overrightarrow{M_\gamma}} &\\
 M_2 \ar@/^0.6pc/[ur]^{M_\alpha} \ar@(u,l)_(.4){v\myid_{M_{2}}} & & L\smallotimesF M_3\ar@/_0.6pc/[ll]^{\overleftarrow{M_\beta}} \ar@(u,r)^(.4){\theta\smallotimesempty \myid_{M_{3}}}
}
$
&
$
\xymatrix@C=3em@R=4em{
   & M_1  \ar@/^0.6pc/[dr]^{M_\gamma} &\\
 L\smallotimesF M_2 \ar@/^0.6pc/[ur]^{\overleftarrow{M_\alpha}} \ar@(u,l)_(.4){\theta\smallotimesempty\myid_{M_2}} & & M_3\ar@/_0.6pc/[ll]^{\overrightarrow{M_\beta}} \ar@(u,r)^(.4){v\myid_{M_3}}
}
$
\\
  & \begin{tabular}{l}  
  $\overrightarrow{M_\alpha}(m)=\frac{1}{2}\sum_{j=0}^{1}u^{-j}\smallotimesempty M_\alpha(u^jm) $
  \\
$\overleftarrow{M_\gamma}(\ell\smallotimesempty m)=\ell M_\gamma(m)$\end{tabular} &  
\begin{tabular}{l}  
  $\overleftarrow{M_\alpha}(\ell\smallotimesempty m)=\theta^{-\xi_\gamma}(\ell) M_\alpha(m)$
  \\
$\overrightarrow{M_\gamma}(m)=\frac{1}{2}\sum_{j=0}^1u^{-j}\smallotimesempty M_\gamma(u^jm) $\end{tabular}
&
\begin{tabular}{l} $\overleftarrow{M_\beta}(\ell\smallotimesempty m)=\ell M_\beta(m)$ \\
$\overrightarrow{M_\gamma}(m)=\frac{1}{2}\sum_{j=0}^1u^{-j}\smallotimesempty M_\gamma(u^jm) $\end{tabular}
&
\begin{tabular}{l} $\overleftarrow{M_\alpha}(\ell\smallotimesempty m)=\theta^{-\xi_\gamma}(\ell) M_\alpha(m)$ \\
$\overrightarrow{M_\beta}(m)=\frac{1}{2}\sum_{j=0}^1u^{-j}\smallotimesempty M_\beta(u^jm) )$
\end{tabular}
\\
\hline
\begin{tabular}{c}
Semilinear\\ clannish\\ block 
\end{tabular}
&
\begin{tabular}{c}
$\xymatrix@C=4.5em@R=6em{
& L \ar@(ld,rd)^{s_1}_{L^\theta\smallotimesL L} \ar@/^0.6pc/[dr]_{\gamma}^{L\smallotimesL L} & \\
L \ar@/^0.6pc/[ur]_{\alpha}^{L^{\theta^{-\xi_\beta}}\smallotimesL L} & & L \ar@/_0.6pc/[ll]_{\beta}^{L^{\theta^{\xi_\beta}}\smallotimesL L} 
}$\\
$\alpha\beta,\beta\gamma,\gamma\alpha$, $s_1^2-e_1$
\end{tabular}
 &
 \begin{tabular}{c}
 $\xymatrix@C=4.5em@R=6em{
 & L  \ar@/^0.6pc/[dr]_(.4){\gamma}^(.4){L\smallotimesL  L} &  \\
 L \ar@(u,l)^{s_2}_(.4){L^{\theta}\smallotimesL  L} \ar@/^0.6pc/[ur]_(.6){\alpha}^(.6){L\smallotimesL L} & & L \ar@(u,r)_{s_3}^(.4){L^\theta\smallotimesL L} \ar@/_0.6pc/[ll]_{\beta}^{L\smallotimesL L} 
}$\\
$\alpha\beta,\beta\gamma,\gamma\alpha$, $s_2^2-e_2,s_3^2-e_3$
\end{tabular}
& 
\begin{tabular}{c}
$\xymatrix@C=3em@R=6em{
 & L  \ar@/^0.6pc/[dr]_(.4){\gamma}^(.4){L^{\theta^{-\xi_\alpha}}\smallotimesL  L} &  \\
 L \ar@(u,l)^{s_2}_(.4){L\smallotimesL  L} \ar@/^0.6pc/[ur]_(.6){\alpha}^(.6){L^{\theta^{\xi_\alpha}}\smallotimesL L} & & L \ar@(u,r)_{s_3}^(.4){L^{\theta}\smallotimesL L} \ar@/_0.6pc/[ll]_{\beta}^{L\smallotimesL L} 
}$\\
$\alpha\beta,\beta\gamma,\gamma\alpha,$  $s_2^2-ue_2,s_3^2-e_3$
\end{tabular}
&
\begin{tabular}{c}
$\xymatrix@C=3em@R=6em{
 & L  \ar@/^0.6pc/[dr]_(.4){\gamma}^(.4){L^{\theta^{\xi_\gamma}}\smallotimesL  L} &  \\
 L \ar@(u,l)^{s_2}_(.4){L^\theta\smallotimesL  L} \ar@/^0.6pc/[ur]_(.6){\alpha}^(.6){L^{\theta^{-\xi_\gamma}}\smallotimesL L} & & L \ar@(u,r)_{s_3}^(.4){L\smallotimesL L} \ar@/_0.6pc/[ll]_{\beta}^{L\smallotimesL L} 
}$\\
$\alpha\beta,\beta\gamma,\gamma\alpha,$  $s_2^2-ue_2,s_3^2-e_3$
\end{tabular}
\\
\hline
$\xymatrix@R=2.4em{
M\ar@{|-}[d]\\ \,
}$ 
&
$\xymatrix@R=4em{ & M_1 \ar@(rd,ld)^{M_{s_1}} \ar@/^0.6pc/[dr]^{M_\gamma} & \\
 M_2 \ar@/^0.6pc/[ur]^{M_\alpha} & & M_3\ar@/_0.6pc/[ll]^{M_\beta}
}$ & 
$
\xymatrix@R=4em{
& M_1 \ar@/^0.6pc/[dr]^{M_\gamma} & \\
M_2 \ar@(u,l)_(.4){M_{s_2}}\ar@/^0.6pc/[ur]^{M_\alpha} & & M_3 \ar@(u,r)^(.4){M_{s_3}} \ar@/_0.6pc/[ll]^{M_\beta}
}
$
&
$
\xymatrix@R=4em{
& M_1 \ar@/^0.6pc/[dr]^{M_\gamma} & \\
M_2 \ar@(u,l)_(.4){M_{s_2}}\ar@/^0.6pc/[ur]^{M_\alpha} & & M_3 \ar@(u,r)^(.4){M_{s_3}} \ar@/_0.6pc/[ll]^{M_\beta}
}
$
&
$
\xymatrix@R=4em{
& M_1 \ar@/^0.6pc/[dr]^{M_\gamma} & \\
M_2 \ar@(u,l)_(.4){M_{s_2}}\ar@/^0.6pc/[ur]^{M_\alpha} & & M_3 \ar@(u,r)^(.4){M_{s_3}} \ar@/_0.6pc/[ll]^{M_\beta}
}
$
\\
\hline
$\xymatrix{
\,\ar@{->}[d] \\ \Phi(M)
}$ 
& 
$\xymatrix@C=1.8em@R=6em{
   & \ker(M_{s_1}\smallminus\myid_{M_1})  \ar@/^0.6pc/[dr]^(.6){M_\gamma|_{\ker}} &\\
  M_2 \ar@/^0.6pc/[ur]^{\frac{1}{2}(\myid_{M_1}\smallplus M_{s_1}) M_\alpha } & & M_3 \ar@/_0.8pc/[ll]^{M_{\beta}}
}$ & 
$
\xymatrix@C=2.2em@R=4em{
& M_{1} \ar@/^0.6pc/[dr]^{\quad\frac{1}{2}(\myid_{M_3}\smallplus M_{s_3}) M_\gamma} & \\
\ker(M_{s_2}\smallminus\myid_{M_2}) \ar@/^0.6pc/[ur]^{M_\alpha|_{\ker}\quad} & & \ker(M_{s_3}\smallminus\myid_{M_3}) \ar@/_0.6pc/[ll]_{\frac{1}{2}(M_{s_2}\smallplus\myid_{M_2}) M_\beta|_{\ker}} \ar@/^0.8pc/[ll]^{\frac{u}{2}(M_{s_2}\smallminus\myid_{M_2}) M_\beta|_{\ker}}
}
$
&
$
\xymatrix@C=1.8em@R=6em{
& M_{1} \ar@/^0.6pc/[dr]^(.4){\frac{1}{2}(\myid_{M_3}\smallplus M_{s_3}) M_\gamma} & \\ 
M_2 \ar@/^0.6pc/[ur]^(.6){M_\alpha} & & \ker(M_{s_3}\smallminus\myid_{M_3}) \ar@/_0.8pc/[ll]^{M_\beta|_{\ker}\qquad} 
}
$
&
$
\xymatrix@C=4em@R=6em{
& M_{1} \ar@/^0.6pc/[dr]^{M_\gamma} & \\
\ker(M_{s_2}\smallminus\myid_{M_2}) \ar@/^0.6pc/[ur]^{M_\alpha|_{\ker}} & & M_3 \ar@/_0.8pc/[ll]^{\qquad\frac{1}{2}(M_{s_2}\smallplus\myid_{M_2}) M_\beta} 
}
$
\\
\hline 
\end{tabular}}
       \captionof{table}{Morita equivalences 2, 4, 6 and 7, behavior of functors on objects
       \label{table-Morita-equivalences-2467}}
    \end{landscape}
}

As the reader should have noticed already, for each Jacobian block in Table \ref{table-Morita-equivalences-2467}, the vertices of the quiver get different fields attached, taken from $\{F,L,E\}$, which means that in the representations of such block the spaces assigned to the vertices are vector spaces over different fields, and the action of each arrow is semilinear over the intersection of the fields attached to its head and its tail. On the other hand, for each semilinear clannish block in that same table, all vertices get attached the same field $L$, which is an extension of $F$ and a subfield of $E$. This means that in the representations of such block the spaces assigned to the vertices are vector spaces over the same field $L$, and the action of each arrow is semilinear over $L$.

When associating to a given representation $M$ of a Jacobian block a representation $\Psi(M)$ of a semilinear clannish block, we replace each $F$-vector space $M_j$ with the $L$-vector space $L\otimes_FM_j$ and the $L$-semilinear endomorphism $\theta\otimes\myid_{M_j}:L\otimes_FM_j\rightarrow L\otimes_FM_j$, whereas each $L$-vector space $M_j$ is left unchanged, and each $E$-vector space $M_j$ is replaced with the $L$-vector space $M_j$ and the $L$-linear endomorphism $v\myid_{M_j}:M_j\rightarrow M_j$. Furthermore, for some arrows $a$ the $F$-linear map $M_a$ has to be replaced with an $L$-semilinear one, which we define through an extension or coextension of scalars of sorts: the domain or the codomain of $M_a$ has been already tensored with $L$, the corresponding  $L$-semilinear extension or coextension of $M_a$, denoted $\overleftarrow{M}_a$ or $\overrightarrow{M}_a$, respectively, is defined by an explicit formula, which appeared in the proof of Lemma \ref{lemma:simplyfing-reps-of-these-species}. Such formula is recalled in the row of Table \ref{table-Morita-equivalences-2467} labeled $\Psi(M)$. When the action of $M_a$ is already semilinear over $L$, no extension or coextension of scalars is needed, so the map $M_a$ is kept unchanged in the definition of~$\Psi(M)$.

On the other hand, when associating to a representation $N$ of a semilinear clannish block a representation $\Phi(N)$ of a Jacobian block, we use the $L$-semilinear endomorphisms $N_{s_j}$ that are not $L$-linear to realize the corresponding $L$-vector space $N_j$ canonically as $L\otimes_F\ker(N_{s_j}-\myid_{N_j})$ and this way replace $N_j$ with the $F$-vector space $\ker(N_{s_j}-\myid_{N_j})$ (this is Galois descent for vector spaces). Similarly, we use the endomorphisms $N_{s_j}$ that are $L$-linear to extend the left action of $L$ on $N_j$ canonically to a action of $E$ on $N_j$ that makes $N_j$ an $E$-vector space.

For the case of the $6^{\operatorname{th}}$ Jacobian and semilinear clannish blocks, we begin by checking that for $M\in \usualRA{A(Q,\bldD,\xi)}/J_0(W(Q,\bldD,\xi))\text{-}\mathbf{Mod}$ and $N\in K_{\bldsigma}\compactQhat/I\text{-}\mathbf{Mod}$ we indeed have $\Psi(M)\in K_{\bldsigma}\compactQhat/I\text{-}\mathbf{Mod}$ and $\Phi(N)\in \usualRA{A(Q,\bldD,\xi)}/J_0(W(Q,\bldD,\xi))\text{-}\mathbf{Mod}$:
\begin{align*}
    \Psi(M)_{\alpha} &\coloneqq  M_\alpha &&  \text{so $\Psi(M)_{\alpha}$ is $\sigma_\alpha$-linear}\\
    \Psi(M)_{\beta}(\ell\otimes m)&\coloneqq \ell M_{\beta}(m) && \text{so $\Psi(M)_{\beta}$ is $\sigma_\beta$-linear}\\
    \Psi(M)_{\gamma}(u m) &\coloneqq  ((\theta^{-\xi_\alpha}\otimes\myid_{M_3})\circ\overrightarrow{M_\gamma})(u m) && \\
    &= \frac{1}{2}(\theta^{-\xi_\alpha}\otimes\myid_{M_3})(1\otimes M_\gamma(u m)+u^{-1}\otimes M_\gamma(u^2m) ) && \\
    &= \frac{1}{2}(\theta^{-\xi_\alpha}\otimes\myid_{M_3})(1\otimes M_\gamma(u m)+u\otimes M_\gamma(m) ) && \text{since $u^2\in F$}\\
    &= \frac{1}{2}(1\otimes M_\gamma(u m)+\theta^{-\xi_\alpha}(u)\otimes M_\gamma(m) ) && \\
    &= \frac{\theta^{-\xi_\alpha}(u)}{2}(\theta^{-\xi_\alpha}(u)^{-1}\otimes M_\gamma(u m)+1\otimes M_\gamma(m) ) && \\
    &= \frac{\theta^{-\xi_\alpha}(u)}{2}(\theta^{-\xi_\alpha}\otimes\myid_{M_3})(u^{-1}\otimes M_\gamma(u m)+1\otimes M_\gamma(m) ) && \\
    &= \theta^{-\xi_\alpha}(u)((\theta^{-\xi_\alpha}\otimes\myid_{M_3})\circ\overrightarrow{M_\gamma})(m) && \\
    &= \theta^{-\xi_\alpha}(u)\Psi(M)_{\gamma}(m) && \text{so $\Psi(M)_{\gamma}$ is $\sigma_\gamma$-linear}\\
    \Psi(M)_{s_2} &\coloneqq  v\myid_{M_2} && \text{so $\Psi(M)_{s_2}$ is $\sigma_{s_2}$-linear}\\
    \Psi(M)_{s_3} &\coloneqq  \theta\otimes \myid_{M_3} && \text{so $\Psi(M)_{s_3}$ is $\sigma_{s_3}$-linear}
\end{align*}
\begin{align*}
    \Psi(M)_{\alpha\beta}(\ell\otimes m) &= M_\alpha\circ \overleftarrow{M_\beta} (\ell\otimes m)=M_\alpha(\ell M_\beta(m))=\theta^{\xi_\alpha}(\ell)M_\alpha (M_\beta(m))\\
    &= \theta^{\xi_\alpha}(\ell)M_{\partial_\gamma(W(Q,\bldD,\xi))}(m)  =0 \\
    \Psi(M)_{\beta\gamma}(m) &= (\overleftarrow{M_\beta}\circ (\theta^{-\xi_\alpha}\otimes\myid_{M_3})\circ\overrightarrow{M_\gamma})(m)\\
    &= \frac{1}{2}((\overleftarrow{M_\beta}\circ  (\theta^{-\xi_\alpha}\otimes\myid_{M_3}) )(1\otimes M_\gamma(m) +u^{-1}\otimes M_\gamma(u m)))\\
    &= \frac{1}{2}\overleftarrow{M_\beta}(1\otimes M_\gamma(m) +\theta^{-\xi_\alpha}(u^{-1})\otimes M_\gamma(u m))\\
    &= \frac{1}{2}(M_\beta (M_\gamma(m)) +\theta^{-\xi_\alpha}(u^{-1})M_\beta (M_\gamma(u m)))\\
    &= M_{\partial_{\alpha}(W(Q,\bldD,\xi))} (m) =0\\
    \Psi(M)_{\gamma\alpha}(m) &= ((\theta^{-\xi_\alpha}\otimes\myid_{M_3})\circ\overrightarrow{M_\gamma}\circ M_\alpha)(m) \\
    &= \frac{1}{2}(\theta^{-\xi_\alpha}\otimes\myid_{M_3}) (1\otimes M_\gamma(M_\alpha(m)) +u^{-1}\otimes M_\gamma(u M_\alpha(m)))\\
    &=\frac{1}{2}(\theta^{-\xi_\alpha}\otimes\myid_{M_3}) (1\otimes M_\gamma(M_\alpha(m)) +u^{-1}\otimes M_\gamma( M_\alpha(\theta^{-\xi_\alpha}(u)m)))\\
    &=\frac{1}{2}(1\otimes M_{\partial_\beta(W(Q,\bldD,\xi))}(m) +\theta^{-\xi_\alpha}(u^{-1})\otimes M_{\partial_\beta(W(Q,\bldD,\xi))}(\theta^{-\xi_\alpha}(u)m)) =0\\
    \Psi(M)_{s_2^2-ue_2} &= ((v\myid_{M_2})\circ(v\myid_{M_2}))-u\myid_{M_2} =0 \\
    \Psi(M)_{s_3^2-e_3} &= ((\theta\otimes \myid_{M_3})\circ(\theta\otimes \myid_{M_3}))-\myid_{M_3} = 0.
\end{align*}
Therefore, $\Psi(M)\in K_{\bldsigma}\compactQhat/I\text{-}\mathbf{Mod}$.

For $N\in K_{\bldsigma}\compactQhat/I\text{-}\mathbf{Mod}$, we turn the $L$-vector space $\Phi(N)_2\coloneqq N_2$ into an $E$-vector space by setting
$$
(\ell_0+\ell_1 v)n\coloneqq \ell_0n+\ell_1N_{s_2}(n) \qquad \text{for} \ \ell_0,\ell_1\in L \ \text{and} \ n\in N_2. 
$$
This uses that the map $N_{s_{2}}$ is $L$-linear. 
Furthermore, $\Phi(N)_3\coloneqq \ker(N_{s_3}-\myid_{N_3})$ certainly is an $F$-vector space since $N_{s_3}$ is $F$-linear, and since $(N_{s_3}-\myid_{N_3})\circ(\myid_{N_3}+N_{s_3})=N_{s_3}^2-\myid_{N_3}=0$, the image of $(\myid_{N_3}+N_{s_3})\circ N_\gamma$ is contained in $\Phi(N)_3$. Moreover,
\begin{align*}
    \Phi(N)_\alpha &\coloneqq  N_\alpha && \text{so $\Phi(N)_{\alpha}$ is $\theta^{\xi_\alpha}$-linear} \\
    \Phi(N)_\beta &\coloneqq N_\beta|_{\ker(N_{s_3}-\myid_{N_3})} && \text{so $\Phi(N)_{\beta}$ is $F$-linear}  \\
    \Phi(N)_\gamma &\coloneqq \frac{1}{2}(\myid_{N_3}+N_{s_3})\circ N_\gamma && \text{so $\Phi(N)_{\gamma}$ is $F$-linear}
\end{align*}
Hence, 
\begin{align*}
    \Phi(N)_{\partial_\alpha(W(Q,\bldD,\xi))} &= \frac{1}{2} (\Phi(N)_\beta \Phi(N)_\gamma + 
\theta^{-\xi_\alpha}(u^{-1})\Phi(N)_\beta \Phi(N)_\gamma u)\\
&= \frac{1}{4} (N_\beta|_{\Phi(N)_3} (\myid_{N_3}+N_{s_3}) N_\gamma + 
\theta^{-\xi_\alpha}(u^{-1})N_\beta|_{\Phi(N)_3} (\myid_{N_3}+N_{s_3}) N_\gamma u) \\
&= \frac{1}{4} (N_\beta| (\myid_{N_3}+N_{s_3}) N_\gamma + 
\theta^{-\xi_\alpha}(u^{-1})N_\beta| (\myid_{N_3}+N_{s_3}) \theta^{-\xi_\alpha}(u)N_\gamma )\\
&=\frac{1}{4} (N_\beta| (\myid_{N_3}+N_{s_3}) N_\gamma + 
\theta^{-\xi_\alpha}(u^{-1})N_\beta| (\theta^{-\xi_\alpha}(u) \myid_{N_3}+\theta^{1-\xi_\alpha}(u)N_{s_3}) N_\gamma )
\\
&=\frac{1}{4} (N_\beta| (\myid_{N_3}+N_{s_3}) N_\gamma + 
\theta^{-\xi_\alpha}(u^{-1})N_\beta| (\theta^{-\xi_\alpha}(u) \myid_{N_3}-\theta^{-\xi_\alpha}(u)N_{s_3}) N_\gamma ) 
\end{align*}
since $\theta(u)=-u$.
Thus, for all $n\in \Phi(N)_1\coloneqq N_1$ we have
\begin{align*}
    \Phi(N)_{\partial_\alpha(W(Q,\bldD,\xi))} (n)&= \frac{1}{4} (N_\beta| (\myid_{N_3}+N_{s_3}) N_\gamma (n) + 
\theta^{-\xi_\alpha}(u^{-1})\theta^{-\xi_\alpha}(u)N_\beta ( \myid_{N_3}-N_{s_3}) N_\gamma (n))\\
&= \frac{1}{4} (N_\beta| (\myid_{N_3}+N_{s_3}) N_\gamma (n) + 
N_\beta ( \myid_{N_3}-N_{s_3}) N_\gamma (n))\\
  &=\frac{1}{4} (N_\beta (\myid_{N_3}+N_{s_3}) N_\gamma + 
  N_\beta(\myid_{N_3}-N_{s_3}) N_\gamma )(n)\\
  &=\frac{1}{2}N_\beta N_\gamma(n)=0.
\end{align*}
Also, we have
\begin{align*}
\Phi(N)_{\partial_\beta(W(Q,\bldD,\xi))} &= \Phi(N)_\gamma \Phi(N)_\alpha= \frac{1}{2}(\myid_{N_3}+N_{s_3}) N_\gamma N_\alpha=0 \\
\Phi(N)_{\partial_\gamma(W(Q,\bldD,\xi))} &= \Phi(N)_\alpha \Phi(N)_\beta= N_\alpha N_\beta|_{\ker(N_{s_3}-\myid_{N_3})}=0.  
\end{align*}
Therefore, $\Phi(N)\in \usualRA{A(Q,\bldD,\xi)}/J_0(W(Q,\bldD,\xi))\text{-}\mathbf{Mod}$. We see that $\Psi$ and $\Phi$ are well-defined on objects.

At the level of morphisms, 
$\Psi:\usualRA{A(Q,\bldD,\xi)}/J_0(W(Q,\bldD,\xi))\text{-}\mathbf{Mod} \rightarrow K_{\bldsigma}\compactQhat/I\text{-}\mathbf{Mod}$ is defined by the rule
$$
\xymatrix{
&   & M_1  \ar@{.>}[dr]^{M_\gamma} \ar@/_1pc/[dd]|-{f_1} & & & & & M_1 \ar@/_1pc/[dd]|-{f_1}  \ar@{.>}[dr]|-{\quad (\theta^{-\xi_\alpha}\otimes\myid_{M_3})\circ\overrightarrow{M_\gamma}} &\\
M\ar[dd]^{f}& M_2  \ar@{.>}[ur]^{M_\alpha } \ar[dd]|-{f_2}  & & M_3 \ar@{.>}[ll]^{\qquad \quad  M_{\beta}}\ar[dd]|-{f_3}  & & \Psi(M)\ar[dd]^{\Psi(f)} & M_2 \ar@/^1.75pc/[dd]|-{f_2} \ar@{.>}[ur]|-{M_\alpha} \ar@{.>}@(dr,dl)^{v\myid_{M_{2}}} & & L\otimes_FM_3 \ar@/^2pc/[dd]|-{\myid_L\otimes f_3} \ar@{.>}[ll]^{\quad \overleftarrow{M_\beta}} \ar@{.>}@(dr,dl)^{\theta\otimes \myid_{M_{3}}}\\
&   & N_1  \ar@{.>}[dr]|-{N_\gamma} && \mapsto & & & N_1  \ar@{.>}[dr]|-{\quad (\theta^{-\xi_\alpha}\otimes\myid_{N_3})\circ\overrightarrow{N_\gamma}} &\\
N& N_2 \ar@{.>}[ur]|-{N_\alpha } & & N_3 \ar@{.>}[ll]|-{N_{\beta}}& & \Psi(N)  & N_2 \ar@{.>}[ur]|-{N_\alpha} \ar@{.>}@(dr,dl)^{v\myid_{N_{2}}} & & L\otimes_FN_3, \ar@{.>}[ll]^{\overleftarrow{N_\beta}} \ar@{.>}@(dr,dl)^{\theta\otimes \myid_{N_{3}}}
}
$$
whereas $\Phi:K_{\bldsigma}\compactQhat/I\text{-}\mathbf{Mod} \rightarrow \usualRA{A(Q,\bldD,\xi)}/J_0(W(Q,\bldD,\xi))\text{-}\mathbf{Mod}$ is defined by the rule
$$
\xymatrix{
& & M_1 \ar@/_1pc/[dd]|-{f_1}  \ar@{.>}[dr]|-{M_\gamma} & & & &   & & M_1  \ar@{.>}[drr]^{\frac{1}{2}(\myid_{M_3}+M_{s_3})\circ M_\gamma}  \ar@/_1.75pc/[dd]|-{f_1} && &\\
 M\ar[dd]^{f} & M_2 \ar@/^1.5pc/[dd]|-{f_2} \ar@{.>}[ur]|-{M_\alpha} \ar@{.>}@(dr,dl)^{M_{s_2}} & & M_3 \ar@/^1.5pc/[dd]|-{f_3} \ar@{.>}[ll]^{\quad M_\beta} \ar@{.>}@(dr,dl)^{M_{s_3}} & & \Phi(M)\ar[dd]^{\Phi(f)}& M_2  \ar@{.>}[urr]^{M_\alpha }  \ar[dd]|-{f_2} & & & &  \ker(M_{s_3}-\myid_{M_3}) \ar@{.>}[llll]^{\qquad M_\beta|_{\ker(M_{s_3}-\myid_{M_3})}}  \ar[dd]|-{f_3}  \\
 & & N_1  \ar@{.>}[dr]|-{N_\gamma} & & \mapsto & &   & & N_1  \ar@{.>}[drr]|-{\frac{1}{2}(\myid_{N_3}+N_{s_3})\circ N_\gamma}  && & \\
 N  & N_2 \ar@{.>}[ur]|-{N_\alpha} \ar@{.>}@(dr,dl)^{N_{s_2}} & & N_3 \ar@{.>}[ll]^{N_\beta} \ar@{.>}@(dr,dl)^{N_{s_3}} & & \Phi(N) & N_2 \ar@{.>}[urr]|-{N_\alpha } & & & &  \ker(N_{s_3}-\myid_{N_3}). \ar@{.>}[llll]^{N_\beta|_{\ker(N_{s_3}-\myid_{N_3})}}  
}
$$

Through a routine check, the reader can easily verify that $\Psi$ and $\Phi$ are covariant $F$-linear functors. We claim that they are equivalences of categories, mutually inverse up to isomorphism of functors. We will prove this claim by exhibiting invertible natural transformations
\begin{align*}
\varepsilon&:\myid_{\jacobalg{A(Q,\xi)}\text{-}\mathbf{Mod}}\rightarrow\Phi\circ\Psi\\
\eta&:\myid_{K_{\bldsigma}\widehat{Q}/I\text{-}\mathbf{Mod}}\rightarrow \Psi\circ\Phi.
\end{align*}

For $M\in \jacobalg{A(Q,\xi)}$-$\mathbf{Mod}$ define an $F$-linear function $\varepsilon_M:M\rightarrow\Phi\circ\Psi(M)$ as follows:
$$
\xymatrix{
   & M_1  \ar@{.>}[ddr]^{M_\gamma} \ar[rrrrr]^{\myid_{M_1}}  & & &&   & M_1  \ar@{.>}[ddr]|-{\qquad\qquad\qquad \frac{1}{2}(\myid_{L\otimes_FM_3}+\theta\otimes \myid_{M_{3}})\circ (\theta^{-\xi_\alpha}\otimes\myid_{M_3})\circ\overrightarrow{M_\gamma}}  &\\
 & & & M\ar[r]^{\varepsilon_M\qquad} & \Phi\circ\Psi(M)& & & \\
 M_2 \ar@{.>}[uur]^{M_\alpha } \ar@/_2pc/[rrrrr]_{\myid_{M_2}}  & & M_3 \ar@{.>}[ll]_{M_{\beta}}  \ar@/_2pc/[rrrrr]_{m\mapsto 1\otimes m} & && M_2 \ar@{.>}[uur]^{M_\alpha}   & & \ker(\theta\otimes \myid_{M_{3}}-\myid_{L\otimes_FM_3}). \ar@{.>}[ll]_{\overleftarrow{M_\beta}\ |_{\ker}\qquad}  
}
$$
And for $N\in  K_{\bldsigma}\compactQhat/I$-$\mathbf{Mod}$, define an $F$-linear function $\eta_N:N\rightarrow \Psi\circ\Phi(N)$ as follows:
$$
\xymatrix{& N_1  \ar@{.>}[ddr]^{N_\gamma} \ar[rrrrr]^{\myid_{N_1}}  & & & & & N_{1}  \ar@{.>}[ddr]|-{\qquad\qquad\qquad\qquad(\theta^{-\xi_\alpha}\otimes\myid_{\ker(N_{s_3}-\myid_{N_3})})\circ\overrightarrow{\frac{1}{2}(\myid_{N_3}+N_{s_3})\circ N_\gamma}}  &  \\
  &  & &N \ar[r]^{\eta_N\qquad} &  \Psi\circ\Phi(N) & &  &  \\
 N_2 \ar@{.>}@(dr,dl)^{N_{s_2}} \ar@{.>}[uur]^{N_\alpha}  \ar@/_3pc/[rrrrr]_{\myid_{N_2}}  & & N_3 \ar@{.>}@(dr,dl)_{N_{s_3}} \ar@{.>}[ll]_{ N_\beta}  \ar@/_3pc/[rrrrr]_{\qquad\qquad\qquad\frac{1}{2}\left(1\otimes \frac{1}{2}\left(\myid_{N_1}+N_{s_3}\right)+u^{-1}\otimes \frac{1}{2}\left(\myid_{N_1}+N_{s_3}\right)u\right)}  & & & N_2 \ar@{.>}@(dr,dl)_{N_{s_2}} \ar@{.>}[uur]|-{N_\alpha}  & & L\otimes_F\ker(N_{s_3}-\myid_{N_3}). \ar@{.>}@(dr,dl)^{\qquad\qquad\theta\otimes \myid_{\ker(N_{s_3}-\myid_{N_3})}} \ar@{.>}[ll]_{\overleftarrow{N_\beta|_{\ker(N_{s_3}-\myid_{N_3})}}\qquad\quad} 
}
$$

Again, a routine check shows that:
\begin{itemize}
\item $\varepsilon_M$ is an isomorphism of $\jacobalg{A(Q,\xi)}$-modules;
\item $\eta_N$ is an isomorphism of $K_{\bldsigma}\compactQhat/I$-modules;
\item $\varepsilon\coloneqq (\varepsilon_M)_{M\in\jacobalg{A(Q,\xi)}\text{-}\mathbf{Mod}}$ is a natural transformation $\myid_{\jacobalg{A(Q,\xi)}\text{-}\mathbf{Mod}}\rightarrow\Phi\circ\Psi$;
\item $\eta\coloneqq (\eta_M)_{M\in L_{\bldsigma}\widehat{Q}/I\text{-}\mathbf{Mod}}$ is a natural transformation $\myid_{K_{\bldsigma}\widehat{Q}/I\text{-}\mathbf{Mod}}\rightarrow \Psi\circ\Phi$.
\end{itemize}
Therefore, $\varepsilon$ and $\eta$ are $F$-linear isomorphisms of functors, and $\Psi$ and $\Phi$ are $F$-linear Morita equivalences between the Jacobian algebra $\jacobalg{A(Q,\xi)}$ and the semilinear clannish algebra $K_{\bldsigma}\compactQhat/I$. 
Observe that  $\Phi,\Psi$ restrict to equivalences between the full subcategories of finite-dimensional representations. 
\end{proof}

\begin{ex}
For $k=2$ in Proposition \ref{prop-morita-equivalences-for-blocks} take $F=\mathbb{R}$, $L=\mathbb{C}$, and consider the Jacobian block (left) and the semilinear clannish block (right)
\begin{center}
\begin{tabular}{ccc}
\begin{tabular}{c}
$\xymatrix{
& \mathbb{R} \ar@/^0.5pc/[dr]_{\gamma}^{\mathbb{C}\otimes_{\mathbb{R}} \mathbb{R}} & \\
\mathbb{C} \ar@/^0.5pc/[ur]_{\alpha}^{\mathbb{R}\otimes_{\mathbb{R}} \mathbb{C}} & & \mathbb{C} \ar@/^0.5pc/[ll]_{\beta}^{\mathbb{C}^{\theta^{\xi_\beta}}\otimes_{\mathbb{C}} \mathbb{C}} 
}$\\
\\
$J(Q,\mathbf{d},\xi)=\langle\beta\gamma, \alpha\beta, \frac{1}{2}(\gamma\alpha+(-1)^{\xi_\beta}u^{-1}\gamma\alpha u)\rangle$
\end{tabular}
& \qquad\qquad &
\begin{tabular}{c}
$\xymatrix{
& \mathbb{C} \ar@(ul,ur)_{s_1}^{\mathbb{C}^\theta\otimes_{\mathbb{C}} \mathbb{C}} \ar@/^0.5pc/[dr]_{\gamma}^{\mathbb{C}\otimes_{\mathbb{C}} \mathbb{C}} & \\
\mathbb{C} \ar@/^0.5pc/[ur]_{\alpha}^{\mathbb{C}^{\theta^{-\xi_\beta}}\otimes_{\mathbb{C}} \mathbb{C}} & & \mathbb{C} \ar@/^0.5pc/[ll]_{\beta}^{\mathbb{C}^{\theta^{\xi_\beta}}\otimes_{\mathbb{C}} \mathbb{C}} 
}$\\
\\
$I=\langle \alpha\beta,\beta\gamma,\gamma\alpha, s_1^2-e_1\rangle$
\end{tabular}
\end{tabular}
\end{center}
where $u\in \mathbb{C}$ satisfies $u^2=-1$. We have seen in Proposition \ref{prop-morita-equivalences-for-blocks} that the arising Jacobian algebra $\jacobalg{A(Q,\xi)}$ and semilinear clannish algebra $\mathbb{C}_{\bldsigma}\compactQhat/I$ are Morita equivalent. They are, however, not isomorphic as rings. 

To prove our claim, notice first that in each of $\jacobalg{A(Q,\xi)}$ and $\mathbb{C}_{\bldsigma}\compactQhat/I$, the center is the image of $\mathbb{R}$ under the corresponding diagonal embeddings $\mathbb{R}\hookrightarrow \jacobalg{A(Q,\xi)}$ and $\mathbb{R}\hookrightarrow\mathbb{C}_{\bldsigma}\compactQhat/I$. Since every ring isomorphism restricts to an isomorphism between centers, and since the only non-zero ring endomorphism of $\mathbb{R}$ is the identity, we deduce that any ring isomorphism between $\jacobalg{A(Q,\xi)}$ and $\mathbb{C}_{\bldsigma}\compactQhat/I$ would be forced to be $\mathbb{R}$-linear. However,
$$
\dim_{\mathbb{R}}(\jacobalg{A(Q,\xi)})=13 \quad \text{and} \quad \dim_{\mathbb{R}}(\mathbb{C}_{\bldsigma}\compactQhat/I)=16,
$$
so $\jacobalg{A(Q,\xi})$ and $\mathbb{C}_{\bldsigma}\compactQhat/I$ cannot be isomorphic as rings.

Alternatively, instead of a dimension count, one could notice that, on the one hand, the diagonal $\Delta:\mathbb{C}\hookrightarrow \mathbb{C}_{\bldsigma}\compactQhat/I$ embeds $\mathbb{C}$ as a unital subring of $\mathbb{C}_{\bldsigma}\compactQhat/I$ containing the center $Z(\mathbb{C}_{\bldsigma}\compactQhat/I)$, and on the other, it is possible to endow the real vector space $\mathbb{R}$ with the structure of left $\jacobalg{A(Q,\xi)}$-module. Were $\varphi:\mathbb{C}_{\bldsigma}\compactQhat/I\rightarrow \jacobalg{A(Q,\xi)}$ a ring isomorphism, \emph{a fortiori} $\mathbb{R}$-linear as we have seen, its restriction to $\Delta(\mathbb{C})$ would lift the real vector space structure of $\mathbb{R}$ to a complex vector space structure, under which we would then have $2\leq [\mathbb{C}:\mathbb{R}]\dim_{\mathbb{C}}(\mathbb{R})=\dim_{\mathbb{R}}(\mathbb{R})=1$.

Similar arguments show in general that for $k=2,4,6,7$ the $k^{\operatorname{th}}$ Jacobian block and the $k^{\operatorname{th}}$ semilinear clannish block cannot be isomorphic as rings.
\end{ex}

\section{Colored triangulations of surfaces with orbifold points}\label{sec:colored-triangs}

\subsection{Triangulations}
\label{subsec:surfaces-and-triangs}

\,

By a \emph{surface with marked points and orbifold points} we mean a triple~$\SSigma = (\Sigma, \marked, \orb)$ consisting of
\begin{itemize}
 \item  an oriented connected compact real surface $\Sigma$ with (possibly empty) boundary~$\partial \Sigma$,
 \item  a non-empty finite set $\marked \subseteq \Sigma$ meeting each connected component of~$\partial \Sigma$ at least once,
 \item  a (possibly empty) finite set $\orb \subseteq \Sigma \setminus (\partial \Sigma \cup \marked)$.
\end{itemize}
The points in $\marked$ are called \emph{marked points}, the points in $\orb$ are \emph{orbifold points}.
Marked points belonging to~$\Sigma \setminus \partial \Sigma$ are known as \emph{punctures}. We shall refer to $\SSigma$ simply as a \emph{surface}.

We will consider only \emph{unpunctured surfaces} with finitely many orbifold points, and \emph{once-punctured closed surfaces} with arbitrarily many orbifold points. Furthermore, we will always assume that $\surf$ is none of the following 8 surfaces:
\begin{itemize}
\item a once-punctured closed sphere with $|\orb|<4$;
\item the unpunctured disc with $|\marked| =1$ and $|\orb|=1$;
\item the unpunctured discs with $|\marked|\in\{1,2,3\}$ and $|\orb|=0$.
\end{itemize}
Our reasons for working only with unpunctured and once-punctured closed surfaces are:
\begin{itemize}
\item A consequence of Theorem \ref{thm:SPs-well-behaved-under-colored-flips-and-muts-arb-weights} below is that the species with potential associated to their colored triangulations are non-degenerate in the sense of Derksen-Weyman-Zelevinsky, a result that for surfaces with arbitrarily many punctures has been shown only when $\orb=\varnothing$ \cite{LF-QPsurfs4} and when the choice of weights $\omega:\orb\rightarrow\{1,4\}$ is the constant function that takes the value $1$ at every orbifold point \cite{GLF1};
\item for surfaces with many punctures, the non-degenerate potentials on the species associated to colored triangulations typically yield Jacobian algebras not Morita equivalent to semilinear clannish algebras. This is well-known in the case of surfaces without orbifold points.
\end{itemize}

\begin{defi}\cite[Section 4]{FeShTu-orbifolds}\label{def:triangulation} Let $\surf$ be a surface with marked points and orbifold points.
\begin{enumerate}\item An \emph{arc} on $\surf$, is a curve $i$ on $\Sigma$ such that:
\begin{itemize}
\item either both of the endpoints of $i$ belong to $\marked$, or $i$ connects a point of $\marked$ with a point of $\orb$;
\item $i$ does not intersect itself, except that its endpoints may coincide;
\item the points in $i$ that are not endpoints do not belong to $\marked\cup\orb\cup\partial\Sigma$;
\item if $i$ cuts out an unpunctured monogon, then such monogon contains at least two orbifold points;
\item if $i$ cuts out an unpunctured digon, then such digon contains at least one orbifold point.
\end{itemize}
\item If $i$ is an arc that connects a point of $\marked$ with a point of $\orb$, we will say that $i$ is a \emph{pending arc}; if it connects a point of $\marked$ to a point of $\marked$, we will say it is \emph{non-pending}.
\item Two arcs $i_1$ and $i_2$ are \emph{isotopic relative to $\marked\cup\orb$} if there exists a continuous function $H:[0,1]\times\Sigma\rightarrow\Sigma$ such that
    \begin{itemize}
    \item[(a)] $H(0,x)=x$ for all $x\in\Sigma$;
    \item[(b)] $H(1,i_1)=i_2$;
    \item[(c)] $H(t,m)=m$ for all $t\in I$ and all $m\in\marked\cup\orb$;
    \item[(d)] for every $t\in I$, the function $H_t:\Sigma\to\Sigma$ given by $x\mapsto H(t,x)$ is a homeomorphism.
    \end{itemize}
    Arcs will be considered up to isotopy relative to $\marked\cup\orb$, parametrization, and orientation.
\item Two isotopy classes $C_1$ and $C_2$ of arcs are \emph{compatible} if either
 \begin{itemize}
 \item $C_1=C_2$; or
 \item $C_1\neq C_2$ and there are arcs $i_1\in C_1$ and $i_2\in C_2$ such that $i_1$ and $i_2$ do not share an orbifold point as a common endpoint, and, except possibly for their endpoints, $i_1$ and $i_2$ do not intersect.
 \end{itemize}
 If $C_1$ and $C_2$ form a pair of compatible isotopy classes of arcs and we have elements $j_1\in C_1$ and $j_2\in C_2$, we will also say that $j_1$ and $j_2$ are compatible.
\item An \emph{ideal triangulation} of $\surf$ is any maximal collection $\tau$ of pairwise compatible arcs.
\end{enumerate}
\end{defi}

Thus, a non-pending arc goes from a point in $\marked$ to a point in $\marked$, whereas a pending arc connects a point in $\marked$ with a point in $\orb$. Loops based at a marked point and cutting off a monogon containing exactly one orbifold point are not considered to be arcs. In this paper, ideal triangulations will be often referred to simply as triangulations, and ideal triangles simply as triangles.

The following result states the basic properties of the \emph{flip}, which is a combinatorial move on ideal triangulations. Recall that we are considering only unpunctured surfaces and closed surfaces with exactly one puncture.

\begin{thm}[\cite{FeShTu-orbifolds}] Let $\surf$ be a surface with marked points and orbifold points.
\begin{enumerate}
\item If $\tau$ is an ideal triangulation of $\surf$ and $i\in\tau$, then there exists a unique arc $j$ on $\surf$ such that the set $\sigma=(\tau\setminus\{i\})\cup\{j\}$ is an ideal triangulation of $\surf$. We say that $\sigma$ is obtained from $\tau$ by the \emph{flip of $i\in\tau$}.
\item Any two ideal triangulations of $\surf$ can be obtained from each other by a finite sequence of flips.
\end{enumerate}
\end{thm}

In other words, for unpunctured surfaces and closed surfaces with exactly one puncture, every arc in an ideal triangulation can be flipped, and any two ideal triangulations are related by a chain of flips. 

\begin{ex} In Figure \ref{Fig:some_flips} we can see four triangulations of a hexagon with one orbifold point.
        \begin{figure}[!ht]
                \centering
                \includegraphics[scale=.175]{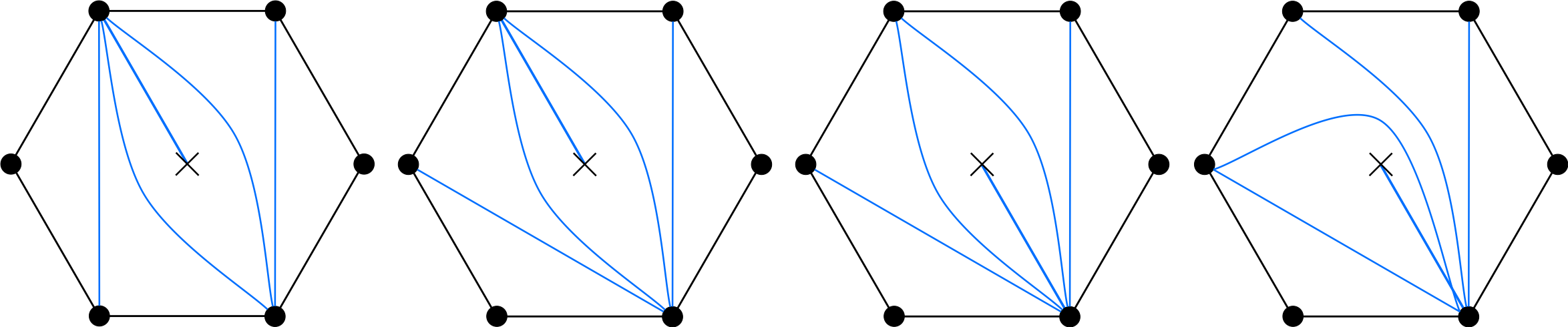}
                \caption{}
                \label{Fig:some_flips}
        \end{figure}
Every two consecutive triangulations are related by a flip.
\end{ex}

\begin{defi} Let $\surf$ be a surface with marked points and orbifold points, and let $\tau$ be an ideal triangulation of $\surf$.
\begin{enumerate}\item An \emph{ideal triangle} of $\tau$ is the topological closure of a connected component of the complement in $\Sigma$ of the union of the arcs in $\tau$.
\item An ideal triangle $\triangle$ is \emph{interior} if its intersection with the boundary of $\Sigma$ consists only of (possibly none) marked points. Otherwise it will be called \emph{non-interior}.
\item An \emph{orbifolded triangle} is an ideal triangle (not necessarily interior) that contains an orbifold point.
\end{enumerate}
\end{defi}

We now give a combinatorial description of ideal triangulations in terms of \emph{puzzle-piece decompositions}.
Consider the three ``puzzle pieces" shown in Figure \ref{Fig:unpunct_puzzle_pieces}.
        \begin{figure}[!ht]
                \centering
                \includegraphics[scale=.175]{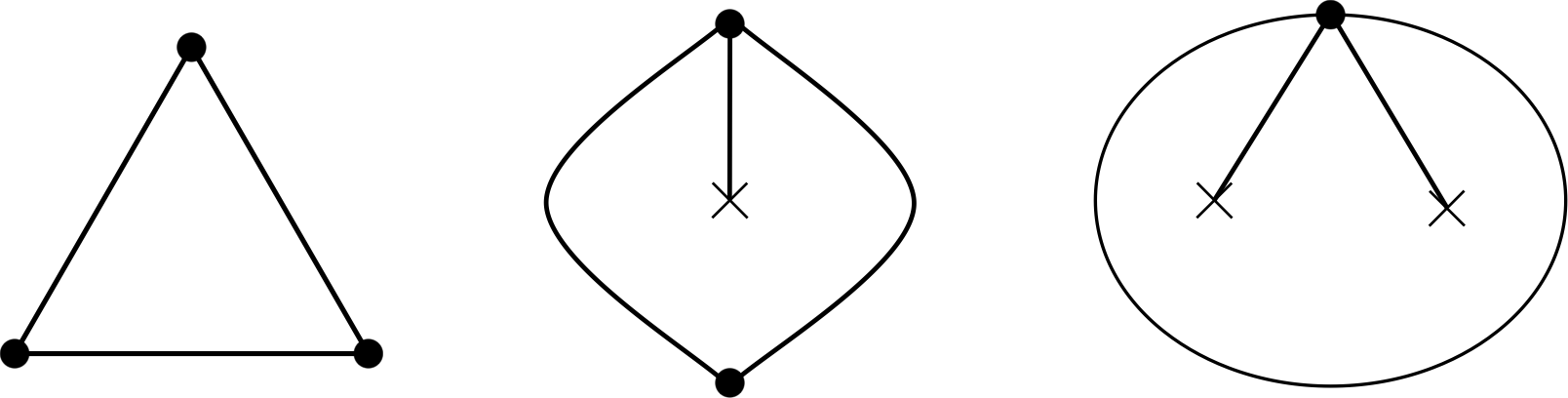}
                \caption{}
                \label{Fig:unpunct_puzzle_pieces}
        \end{figure}

Take several pairwise disjoint copies of these pieces, assign an orientation to each of the outer sides of these copies and fix a partial matching on the set of all outer sides of the copies taken, never matching two sides of the same copy. Then glue the puzzle pieces along the matched sides, making sure the orientations match. Though some partial matchings may not lead to an (ideal triangulation of an) oriented surface, we do have the following.

\begin{thm}\label{thm:existence-of-puzzle-piece-decomps} Any ideal triangulation $\tau$ of an oriented surface $\surf$ can be obtained from a suitable partial matching by means of the procedure just described.
\end{thm}

One way to see this is to start with an ideal triangulation $\tau_0$ of $(\Sigma,\mathbb{M},\varnothing)$, add the points in $\mathbb{O}$, say one by one, completing the given ideal triangulation each time a point is added, and then notice that 
\begin{enumerate}
    \item $\tau_0$ can be obtained from a puzzle-piece decomposition, say, by \cite[Remark 4.2]{Fomin-Shapiro-Thurston};
    \item every time a point from $\mathbb{O}$ is added and the ideal triangulation is completed, the puzzle-piece decomposition can be updated;
    \item possessing a puzzle-piece decomposition is a property of ideal triangulations which is invariant under flips. This can be easily shown through a case by case verification depending on whether the arc to be flipped sits inside a puzzle piece or is an arc shared by two puzzle pieces. The verification was carried out exhaustively in \cite[Figures 21 and 22]{GLF2}.
\end{enumerate}

\begin{defi}\label{def:puzzle-piece-decomps} Any partial matching giving rise to $\tau$ through the procedure just described will be called a \emph{puzzle-piece decomposition} of $\tau$.
\end{defi}

Theorem \ref{thm:existence-of-puzzle-piece-decomps} will play an essential role in the proof of our main result. Notice that the possibilities for how a triangle in a triangulation of one of the surfaces in our setting (unpunctured, or once-punctured closed) can look like are limited. More precisely, there are three types of triangles:
\begin{itemize}
 \itemsep 4pt \parskip 0pt \parsep 0pt

 \item
 \emph{Ordinary triangles}, i.e.\ triangles containing no orbifold points.

 \item
 \emph{Once orbifolded triangles}, i.e.\ triangles containing exactly one orbifold point.
.

 \item
 \emph{Twice orbifolded triangles}, i.e.\ triangles containing containing exactly two orbifold points.
\end{itemize}

Following \cite[Definition 3.2]{GLF2}, given a triangulation $\tau$ of $\SSigma$, we define a
 quiver $\overline{Q}(\tau)$ as follows:
 \begin{enumerate}
  \itemsep 3pt \parskip 0pt \parsep 0pt

  \item
  The vertices of $\overline{Q}(\tau)$ are the arcs in $\tau$, that is,  $\overline{Q}(\tau)_0 = \tau$.

  \item
  The arrows of $\overline{Q}(\tau)$ are induced by the triangles of $\tau$ and the orientation of $\surfnoM$:
  for each triangle $\triangle$ of $\tau$ and every pair $i, j \in \tau$ of arcs in $\triangle$ such that $j$ succeeds $i$ in $\triangle$ with respect to the orientation of $\Sigma$, we draw a single arrow
  from $i$ to $j$.
 \end{enumerate}
  Thus, for the three types of triangles depicted in Figure~\ref{Fig:unpunct_puzzle_pieces}, we draw arrows according to the rule depicted in Figure \ref{Fig:rule-for-arrows-of-overlineQ}, with the convention that no arrow incident to a boundary segment is drawn.
        \begin{figure}[!ht]
                \centering
                \includegraphics[scale=.1]{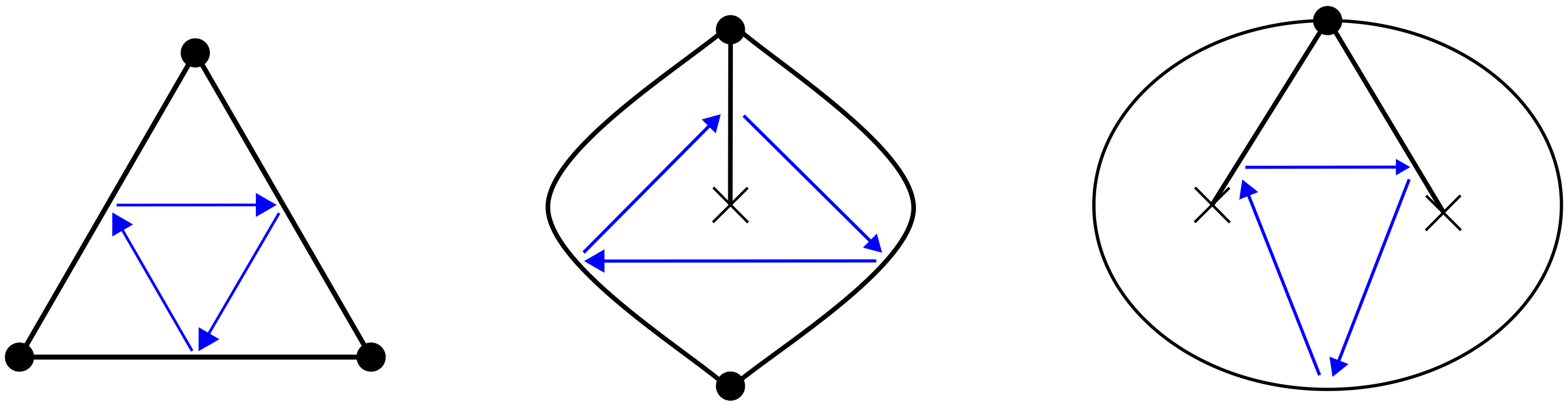}
                \caption{}
                \label{Fig:rule-for-arrows-of-overlineQ}
        \end{figure}

\subsection{Colored triangulations}
\label{subsec:colored-triangulations}

\,

Let $\SSigma$ be a surface with orbifold points, and let $\tau$ a triangulation of $\SSigma$. Through a slight modification of \cite[Equation (4.1)]{GLF2}, we define a family of sets $X_\bullet(\tau) = (X_n(\tau))_{n \in \mathbb{Z}_{\geq 0}}$ by setting $X_n(\tau) = \varnothing$ for $n \not\in \{0, 1, 2\}$ and
\begin{equation}\label{eq:def-sets-X}
 \begin{array}{ccc}
X_0(\tau)   =  \overline{Q}(\tau)_0,
&
X_1(\tau)   =  \overline{Q}(\tau)_1,
&
 X_2(\tau) = \{ \triangle \suchthat \text{$\triangle$ is an interior triangle of $\tau$} \}
 \end{array}
\end{equation}
(a triangle $\triangle$ is \emph{interior} if its intersection with the boundary of $\Sigma$ consists only of marked points).
We use $X_\bullet(\tau)$ to define a chain complex $\Ctau$ as follows:
\begin{equation}\label{eq:def-chain-complexes}
 \xymatrix{
  \Ctau :
  \:\:\:
  \cdots \ar[r] & 0
  \ar[r]^{\partial_3\phantom{/10pt/}}
  &
  \F_2 X_2(\tau)
  \ar[r]^{\partial_2}
  &
  \F_2 X_1(\tau)
  \ar[r]^{\partial_1}
  &
  \F_2 X_0(\tau)
  \ar[r]  \ar[r]^/9pt/{\partial_0}
  &
  0,
 }
\end{equation}
where $\F_2X$ stands for the vector space with basis $X$ over the two-element field $\F_2\coloneqq \Z/2\Z$.
The non-zero differentials are given on basis elements as follows:
\begin{equation}\label{eq:def-of-differentials}
 \arraycolsep 2pt
 \begin{array}{lcll}
  \partial_2(\triangle)
  &=&
  {{\alpha}} + {{\beta}} + {{\gamma}}
  &
  \hspace{0.5em}
  \text{if \smash{$\triangle \in X_2(\tau)$} induces
  $\alpha, \beta, \gamma \in \overline{Q}(\tau)_1$,}
  \\ [0.5em]
  \partial_1({{\alpha}})
  &=&
  h(\alpha) - t(\alpha)
  &
  \hspace{0.5em}
  \text{for \smash{${{\alpha}} \in X_1(\tau)$}.}
 \end{array}
\end{equation}

\begin{ex}\cite[Example 4.3]{GLF2} In Figure \ref{Fig:pentagon_two_orb_points} we can see two triangulations $\tau$ and $\sigma$ of the pentagon with two orbifold points, as well as the quivers $\overline{Q}(\tau)$ and $\overline{Q}(\sigma)$. We can also visualize the 2-dimensional cells belonging to the sets $X_2(\tau)$ and $X_2(\sigma)$.
\end{ex}

        \begin{figure}[!ht]
                \centering
                \includegraphics[scale=.175]{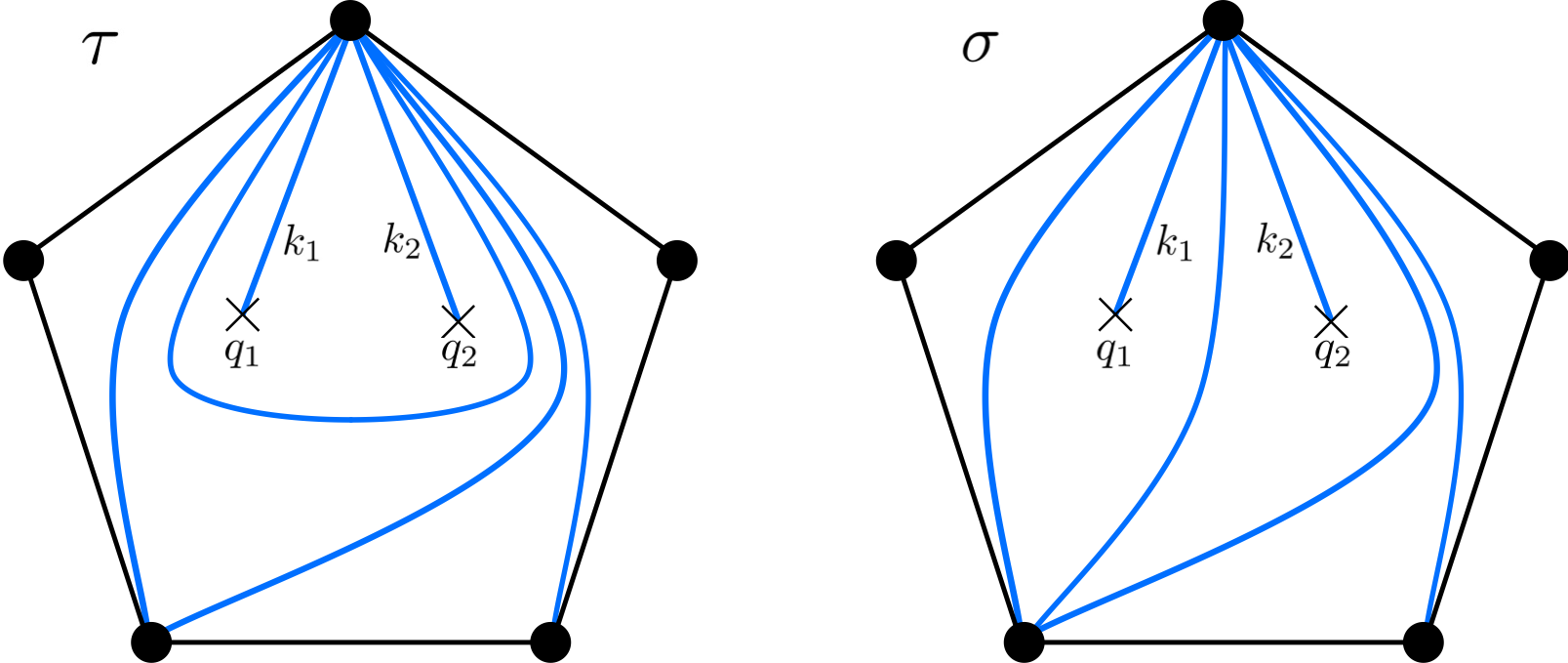}
                \includegraphics[scale=.175]{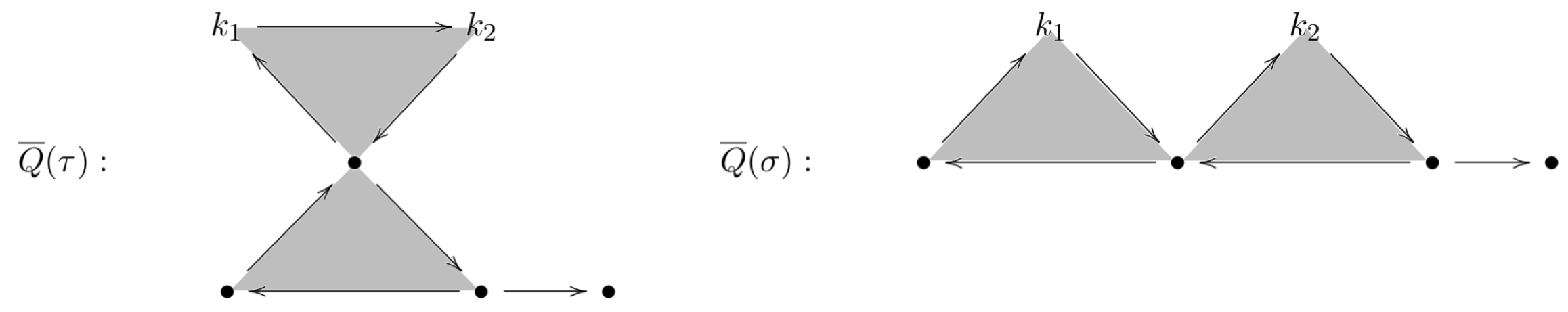}
                \caption{}
                \label{Fig:pentagon_two_orb_points}
        \end{figure}

\begin{defi}\label{defi:colored-triangulations}
 Let $\CZonetauwFtwo$ be the set of $1$-cocycles of the cochain complex $\CCtauwFtwo = \Hom_{\F_2}(\CtauwFtwo, \F_2)$.
 A \emph{colored triangulation} of $\SSigmaw$ is a pair~$(\tau, \xi)$ consisting of a triangulation $\tau$ of $\SSigma$ and a $1$-cocycle~$\xi \in \CZonetauwFtwo\subseteq C^1(\tau)=\Hom_{\F_2}(C_1(\tau), \F_2)$.
\end{defi}

\begin{remark}\label{rem:cocycle-condition}\begin{enumerate}\item The chain complex $C_\bullet(\tau,\omega)$ defined in \cite[Equation (4.1)]{GLF2} is a subcomplex of the chain complex $C_\bullet(\tau)$ we have defined above through \eqref{eq:def-sets-X}, \eqref{eq:def-chain-complexes} and \eqref{eq:def-of-differentials}. It is easy to see that the inclusion $C_\bullet(\tau,\omega)\hookrightarrow C_\bullet(\tau)$ is a homotopy equivalence.
\item The first cohomology group $H^1(C^\bullet(\tau))$ is isomorphic to $H^1(\Sigma\setminus\mathbb{M},\F_2)$, see \cite[Definition 3.6, Equations (4.1), (4.3) and (8.1), and Corollary 8.8]{GLF2}. Thus, for instance, if $\Sigma$ has positive genus, then $H^1(C^\bullet(\tau))\neq 0$.
\item By definition, $C_1(\tau)$ is the $\F_2$-vector space with basis $\overline{Q}(\tau)_1$. Let $\{{{\alpha}}^\vee\suchthat\alpha\in \overline{Q}(\tau)_1\}$ be the $\F_2$-vector space basis of $C^1(\tau)=\Hom_{\F_2}(C_1(\tau), \F_2)$ which is dual to $\overline{Q}(\tau)_1$.
Then, choosing a cocycle $\xi = \sum_{{{\alpha}}} \xi({{\alpha}}) {{\alpha}}^\vee \in \CZonetauwFtwo$ amounts to fixing, for each arrow \smash{${{\alpha}} \in \overline{Q}(\tau)_1$}, an element $\xi({{\alpha}}) \in \{0,1\} = \F_2$ in such a way that whenever ${{\alpha}}, {{\beta}}, {{\gamma}}$ are arrows of $\overline{Q}(\tau)$ induced by an interior triangle $\triangle$
 one has
 \[
  \xi({{\alpha}}) + \xi({{\beta}}) + \xi({{\gamma}}) = 0 \:\in\: \F_2 \,.
 \]
 See Section \ref{sec:on-need-of-cocycles} below for a brief discussion on the necessity to impose this cocycle condition.
\end{enumerate}
\end{remark}

\section{Jacobian and semi-linear clannish algebras associated to colored triangulations}
\label{sec:Jac-algs-and-semilinear-clan-algs-of-colored-triangs}

\subsection{The weighted quiver of a triangulation}
\label{subsec:weighted-quiver}

\,

As already mentioned in the Introduction, our input information will consist not only of a surface $\SSigma$, but of an assignment of a weight to each orbifold point.

\begin{defi}\label{def:surf-with-weighted-orb-pts}
 A \emph{surface with marked points and weighted orbifold points}~$\SSigmaw$, is a surface $\SSigma=\surf$ together with a function $\omega : \orb \to \{1, 4\}$.
\end{defi}

\begin{remark}\begin{enumerate}\item The idea of taking a function $\omega:\orb\rightarrow\{1,4\}$ as part of the input information comes from~\cite{FeShTu-orbifolds}.
\item If $\Gamma$ is a discrete subgroup of $\operatorname{PSL}_2(\mathbb{R})$ and $z$ is a point in the upper half plane $\mathbb{U}\subseteq\mathbb{C}$ which is fixed by a non-identity element of $\Gamma$, then the order of the stabilizer $\Gamma_z\subseteq \Gamma$ is said to be the \emph{order of $q=p(z)$ as an orbifold point} of $\mathbb{U}/\Gamma$, where $p:\mathbb{U}\rightarrow \mathbb{U}/\Gamma$ is the projection to the orbit space. In Teichmüller theory one typically fixes the topological type of $\mathbb{U}/\Gamma$ (that is, one fixes it as a topological manifold, but ignores any possible Riemann surface structure on it) as well as a set of prescribed orbifold points, together with their prescribed orders --integers greater than $1$, then considers all the discrete subgroups $G$ of $\operatorname{PSL}_2(\mathbb{R})$ such that $\mathbb{U}/G$ has the desired topological type and the prescribed orbifold points, with the prescribed orders. 
\item As such, the number $\omega(q)\in\{1,4\}$ is unrelated to the order of $q$ as an orbifold point, which plays no role in this paper.
 \end{enumerate}
\end{remark}

For the rest of the article,~$\SSigma_\omega=(\Sigma,\marked,\orb,\omega)$ will be part of our \emph{a priori} given input. For each triangulation $\tau$ of $\SSigma$, we shall define a weighted quiver $(Q(\tauw), \dtuple(\tauw))$.

\begin{defi}\label{def:Q(tau,omega)}\cite[Definition 3.3]{GLF2}
Let $\SSigmaw=(\surfnoM,\marked,\orb,\omega)$ be a surface with weighted orbifold points as in Definition \ref{def:surf-with-weighted-orb-pts}, and let $\tau$ a triangulation of $\SSigma$.
  For each arc $i \in Q_0(\tau)$ we define an integer $\dtauwi$, the \emph{weight of $i$ with respect to $\omega$}, by the rule
   \[
    \dtauwi
    \coloneqq 
    \begin{cases}
     2         & \text{if $i$ is a non-pending arc,} \\
     \omega(q) & \text{if $i$ is a pending arc with $q \in i \cap \orb$.}
    \end{cases}
   \]
We set $\dtuple(\tauw)=(\dtauwi)_{i\in\tau}$, and define the \emph{weighted quiver of $\tau$ with respect to $\omega$} to be the weighted quiver $(Q(\tauw), \dtuple(\tauw))$ on the vertex set $Q_0(\tauw)=\tau$, where $Q(\tauw)$ is the quiver obtained from $\overline{Q}(\tau)$ by adding an extra arrow $j\rightarrow i$ for each pair of pending arcs $i$ and $j$ that satisfy $\dtauwi = \dtauwi[j]$ and for which $\overline{Q}(\tau)$ has an arrow from $j$ to $i$.
\end{defi}

\begin{ex}\label{ex:all-possible-associated-quivers}\cite[Example 3.8]{GLF2} Consider the triangulations $\tau$ and $\sigma$ from Figure \ref{Fig:pentagon_two_orb_points}. The quivers $Q(\tauw)$ and $Q(\sigma,\omega)$ can be seen in Figure
        \begin{figure}[!ht]
                \centering
                \includegraphics[scale=.125]{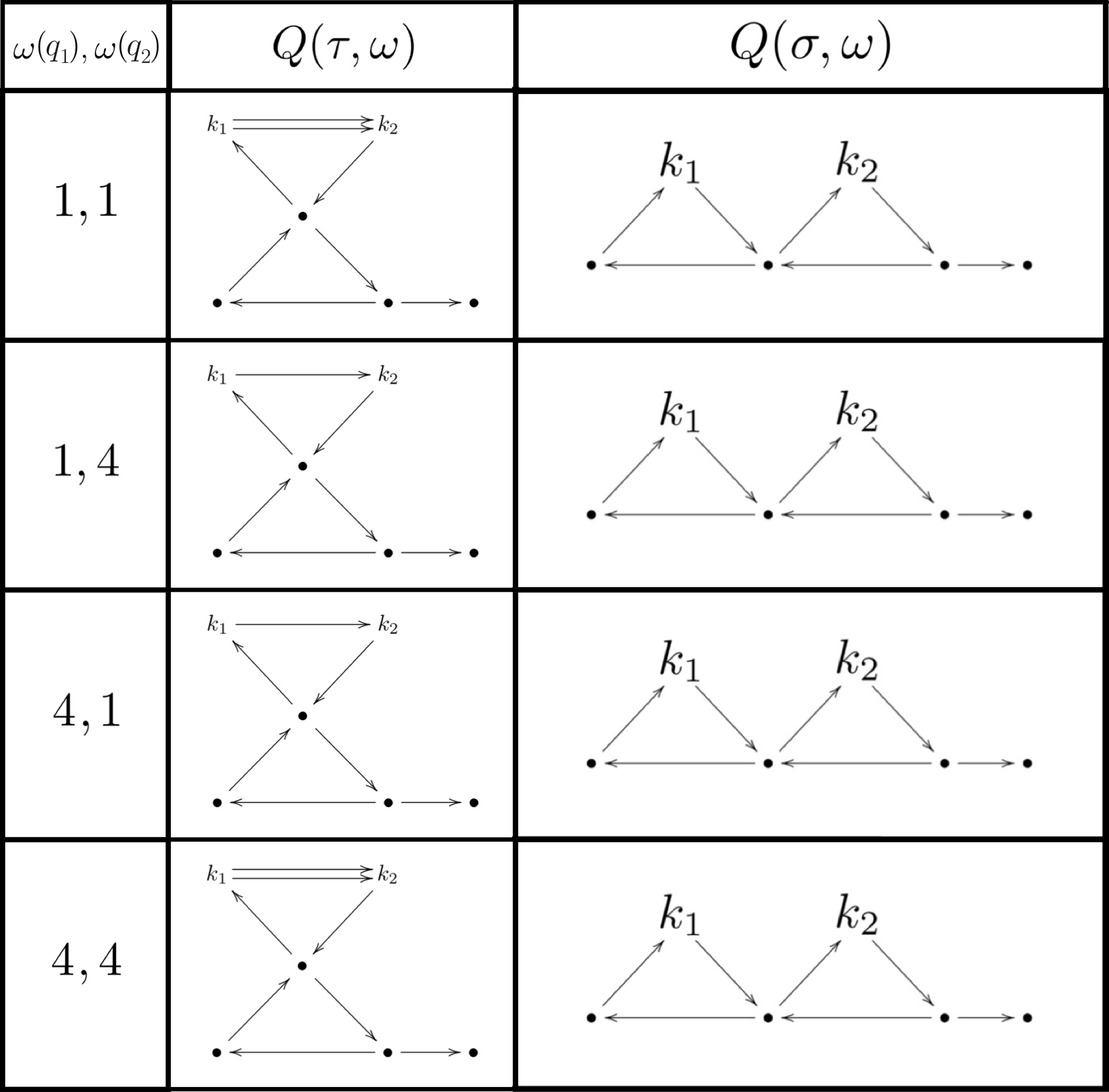}
                \caption{}
                \label{Fig:example_Qtauomega}
        \end{figure}
for all possible functions $\omega:\orb\rightarrow\{1,4\}$. No triangle of $\sigma$ contains more than one orbifold point, hence $Q(\sigma,\omega)=\overline{Q}(\sigma)$ for every function $\omega:\orb\rightarrow\{1,4\}$.
\end{ex}

\begin{ex}\label{ex-weighted-quivers-def-in-tables-also-come-from-triangulations}
For $k=1,\ldots,7$, the weighted quiver $(Q,\bldD)$ appearing in the column labeled ``Block $k$'' in Tables  \ref{table-Jacobian-blocks-1-to-5}, \ref{table-Jacobian-blocks-6-to-10}, \ref{table-semilinear-clannish-blocks-1-to-5} and \ref{table-semilinear-clannish-blocks-6-to-10} has the form $(Q(\tauw), \dtuple(\tauw))$ for some triangulation $\tau$ of a puzzle piece from Figure \ref{Fig:unpunct_puzzle_pieces} with weighted orbifold points.
\end{ex}

\subsection{Arbitrary weights: algebras defined over degree-4 field extensions}\label{subsec:algebras-for-arb-weights}

\,

Let $\SSigmaw=(\Sigma,\marked,\orb,\omega)$ be a surface with weighted orbifold points, and let $\tau$ be a triangulation of $\SSigma$. Set $d=\lcm\{d(\tau,\omega)_k\suchthat k\in\tau\}$, which is equal to either $2$ or $4$ (because $\tau$ contains at least one non-pending arc), and let $E/F$ be a degree-$d$ datum, and $L/F$ be the degree-$2$ datum contained in $E/F$ (see Subsection \ref{sec:specific-field-extensions}). Notice that
\begin{equation}\label{eq:d=2-or-d=4-depending-on-omega}
d=\begin{cases}
2 & \text{if $\orb=\varnothing$ or $\omega\equiv 1$},\\
4 & \text{otherwise;}
\end{cases}
\qquad
\text{hence}
\qquad
[E:L]=\begin{cases}
1 & \text{if $\orb=\varnothing$ or $\omega\equiv 1$},\\
2 & \text{otherwise;}
\end{cases}
\end{equation}

\begin{remark}\label{rem:orb-not-empty-and-omega-constant-1=>degree-2-datum}
\begin{enumerate}\item If $\orb\neq\varnothing$, then there are $2^{\orb}-1$ functions $\omega:\orb\rightarrow\{1,4\}$ that yield $d=4$, but only one that yields $d=2$, namely, the constant function $\omega\equiv 1$.
\item If $\omega\equiv 1$, then \eqref{eq:d=2-or-d=4-depending-on-omega} tells us that $E/F$ is a degree-$2$ datum and $E=L$ (even if $\orb=\varnothing$). In particular, when $\omega\equiv 1$, the forthcoming constructions and results are valid over the field extension $\mathbb{C}/\mathbb{R}$.
\end{enumerate}
\end{remark}

Following \cite[Section 6]{GLF2}, for each $k\in\tau$ we set $F_k/F$ to be the unique degree-$d(\tau,\omega)_k$ field subextension of $E/F$, and denote $G_k=\Gal(F_k/F)$. We also denote $G_{j,k}=\Gal(F_j\cap F_k/F)$ for $j,k\in\tau$. Thus:
$$
G_{j,k}=\begin{cases}
\{\myid_{E},\rho,\rho^2,\rho^3\} & \text{if $\lcm(d(\tau,\omega)_j,d(\tau,\omega)_k)=4$;}\\
\{\myid_{L},\theta\} & \text{if $\lcm(d(\tau,\omega)_j,d(\tau,\omega)_k)=2$;}\\
\{\myid_{F}\} & \text{if $\lcm(d(\tau,\omega)_j,d(\tau,\omega)_k)=1$.}
\end{cases}
$$

\subsubsection{\textbf{The Jacobian algebra of a colored triangulation}}\label{subsubsec:arb-weights-Jac-alg}

Let $(\tau,\xi)$ be a colored triangulation of $\SSigmaw$.
Exactly as in \cite[Definition 6.1]{GLF2}, we define a modulating function $g(\tau,\xi):Q(\tau,\omega)_1\rightarrow\bigcup_{j,k\in\tau}G_{j,k}$ as follows. Take an arrow $a\in Q(\tau,\xi)_1$.
\begin{enumerate}
\item If $d(\tau,\omega)_{h(a)}=1$ or $d(\tau,\omega)_{t(a)}=1$, set
$$
g(\tau,\xi)_{a}=\myid\in G_{h(a),t(a)}.
$$
\item If $d(\tau,\omega)_{h(a)}\neq 1\neq d(\tau,\omega)_{t(a)}$, and $d(\tau,\omega)_{h(a)}d(\tau,\omega)_{t(a)}<16$, set
    $$
    g(\tau,\xi)_{a}=\theta^{\xi(a)}\in\ G_{h(a),t(a)}.
    $$
\item If $d(\tau,\omega)_{h(a)}=4=d(\tau,\omega)_{t(a)}$, then $t(a)$ and $h(a)$ are pending arcs contained in a twice orbifolded triangle $\triangle$, and
\begin{enumerate}
\item the quiver $\overline{Q}(\tau)$ has exactly one arrow $t(a)\rightarrow h(a)$, induced by $\triangle$; let $\delta_0^\triangle$ be this arrow of $\overline{Q}(\tau)$; notice that we can evaluate $\xi$ at $\delta_0^\triangle$;
\item the quiver $Q(\tau,\omega)$ has exactly two arrows going from $t(a)$ to $h(a)$, one of which is $\delta_0^\triangle$; let $\delta_1^\triangle$ be the other such arrow of $Q(\tau,\omega)$; of course, $a\in\{\delta_0^\triangle,\delta_1^\triangle\}$;
\item $[E:F]=4$ and $F_{h(a)}=E=F_{t(a)}$; let $\ell$ be the unique element of $\{0,1\}$ whose residue class modulo 2 is $\xi(\delta_0^\triangle)\in\F_2\coloneqq \mathbb{Z}/2\mathbb{Z}$ (equivalently, let $\ell$ be the unique element of $\{0,1\}$ such that $\rho^\ell|_{L}=\theta^{\xi(\delta_0^\triangle)}=\rho^{\ell+2}|_{L}$).
\end{enumerate}
We set
    $$
    g(\tau,\xi)_{a}=\begin{cases}
    \rho^\ell & \text{if $a=\delta_0^\triangle$;}\\
    \rho^{\ell+2} & \text{if $a=\delta_1^\triangle$.}
    \end{cases}
    $$
\end{enumerate}

\begin{ex}\label{ex-modulated-quivers-from-puzlle-pieces-defined-table}
For $k=1,\ldots,7$, the weighted quiver $(Q,\bldD)$ and  the modulating function $Q_{1}\to\bigcup_{i,j}G_{i,j}$ appearing in the column labeled ``Block $k$'' in Tables \ref{table-Jacobian-blocks-1-to-5} and \ref{table-Jacobian-blocks-6-to-10} have the form $(Q(\tau,\xi),\bldD(\tau,\xi))$ and  $g(\tau,\xi)$, respectively, for some colored triangulation $(\tau,\xi)$ of a puzzle piece surface from Figure \ref{Fig:unpunct_puzzle_pieces}.
\end{ex}

For the next definition we refer the reader to \S\ref{subsubsec:weighted-quivers-and-modulations} and \S\ref{subsubsec:modulating-functions-def-of-concept}.

\begin{defi}\label{def:species-of-colored-triangulation}\cite[Definition 6.2]{GLF2} The \emph{species of the colored triangulation $(\tau,\xi)$} is the $F$-modulation of $(Q(\inputpairfromtriangulation), \mathbf{d}(\inputpairfromtriangulation))$ defined by setting 
\begin{align*}
(\mathbf{F},\mathbf{A}(\tau,\xi))&\coloneqq ((F_k)_{k\in\tau},(A(\tau,\xi)_a)_{a\in Q(\tau,\xi)_1}),  \qquad \text{where} \\
A(\tau,\xi)_a&\coloneqq F_{h(a)}^{g(\tau,\xi)_{a}}\otimes_{F_h(a)\cap F(ta)} F_{t(a)}. 
\end{align*}
\end{defi}

We write $R\coloneqq \times_{k\in\tau} F_k$ and $A(\tau,\xi)\coloneqq \bigoplus_{a\in Q(\tau,\xi)_1}A(\tau,\xi)_a$. It is clear that $R$ is a semisimple ring and $A(\tau,\xi)$ is an $R$-$R$-bimodule.
Detailed examples can be found in \cite[Examples 6.3 and 6.4]{GLF2}. The next proposition asserts that $(\mathbf{F},\mathbf{A}(\tau,\xi))$ is a species realization of one of the skew-symmetrizable matrices associated to $\tau$ by Felikson-Shapiro-Tumarkin \cite{FeShTu-orbifolds}, cf. \cite[Remark 3.5-(2)]{GLF2}. The proof is left to the reader.

\begin{prop}\label{prop:our-species-realize-FeShTu-matrices}\cite[Proposition 6.5]{GLF2} Let $\SSigmaw$ be a surface with weighted orbifold points, and $(\tau,\xi)$ a colored triangulation of $\SSigmaw$, where $\SSigma$ is either unpunctured or once-punctured closed. Let $B(\tau,\omega)=(b_{kj}(\tau,\omega))_{k,j}$ denote the skew-symmetrizable matrix that corresponds to the weighted quiver $(Q(\tau,\omega),\dtuple(\tau,\omega))$ under \cite[Lemma 2.3]{LZ}. For every pair $(k,j)\in\tau\times\tau$ we have:
\begin{enumerate}
\item $e_kA(\tau,\xi)e_j$ is an $F_k$-$F_j$-bimodule;
\item $\dim_{F_k}(e_kA(\tau,\xi)e_j)=[b_{kj}(\tau,\omega)]_+$ and $\dim_{F_j}(e_kA(\tau,\xi)e_j)=[-b_{jk}(\tau,\omega)]_+$, where $[b]_+\coloneqq \max(b,0)$;
\item there is an $F_j$-$F_k$-bimodule isomorphism $\Hom_{F_k}(e_kA(\tau,\xi)e_j,F_k)\cong\Hom_{F_j}(e_kA(\tau,\xi)e_j,F_j)$.
\end{enumerate}
\end{prop}

\begin{remark}
Notice that in the situation $d(\tau,\omega)_{h(a)}=4=d(\tau,\omega)_{t(a)}$ above, writing $a:j\rightarrow k$ we have
$$
F_{k}^{\theta^{\xi_{\delta_0}}}\otimes_{L} F_{j} = E^{\theta^{\xi_{\delta_0}}}\otimes_LE
\cong 
(E^{\rho^l}\otimes_{E} E)
\oplus 
(E^{\rho^{l+2}}\otimes_{E} E)=
(F_{k}^{g(\tau,\xi)_{\delta_0}}\otimes_{F_k\cap F_j} F_{j})
\oplus 
(F_{k}^{g(\tau,\xi)_{\delta_1}}\otimes_{F_k\cap F_j} F_{j})
$$
and $E^{\rho^l}\otimes_{E} E
\not\cong 
E^{\rho^{l+2}}\otimes_{E} E$ as $E$-$E$-bimodules.
\end{remark}

We now move towards the definition of a natural potential $W(\tau,\xi)\in\usualRA{A(\tau,\xi)}$. There are some obvious cycles on $A(\tau,\xi)$, that we point to explicitly.

\begin{defi}\label{def:cycles-from-triangles}\cite[Definitions 6.7, 6.8 and 6.9]{GLF2}
Let $(\tau,\xi)$ be a colored triangulation of $\SSigmaw$ and $\triangle$ be an interior triangle of $\tau$.
\begin{enumerate}
    \item If $\triangle$ does not contain any orbifold point, then, with the notation from the picture on the upper left in Figure \ref{Fig:triangles_quivers}), we set $W^\triangle(\tau,\xi);=\alpha^\triangle\beta^\triangle\gamma^\triangle$;
    \item if $\triangle$ contains exactly one orbifold point, let $k$ be the unique pending arc of $\tau$ contained in $\triangle$. Using the notation from the picture on the upper right in Figure \ref{Fig:triangles_quivers}, we set $W^\triangle(\tau,\xi)=\alpha^\triangle\beta^\triangle\gamma^\triangle$, regardless of whether $d(\tau,\omega)_k=1$ or $d(\tau,\omega)_k=4$;
    \item it $\triangle$ contains exactly two orbifold points, let $k_1$ and $k_2$ be the two pending arcs of $\tau$ contained in $\triangle$, and assume that they are configured as in Figure \ref{Fig:triangles_quivers}.
\begin{itemize}
\item If $d(\tau,\omega)_{k_1}=1=d(\tau,\omega)_{k_2}$, then, with the notation of the picture on the bottom left in Figure \ref{Fig:triangles_quivers}, we set $W^\triangle(\tau,\xi)=\delta_0^\triangle\beta^\triangle\gamma^\triangle+\delta_1^\triangle\beta^\triangle u\gamma^\triangle$.
\item If $d(\tau,\omega)_{k_1}=1$ and $d(\tau,\omega)_{k_2}=4$, then, with the notation of the picture on the bottom right in Figure \ref{Fig:triangles_quivers}, we set $W^\triangle(\tau,\xi)=\alpha^\triangle\beta^\triangle\gamma^\triangle$.
\item If $d(\tau,\omega)_{k_1}=4$ and $d(\tau,\omega)_{k_2}=1$, then, with the notation of the picture on the bottom right in Figure \ref{Fig:triangles_quivers}, we set $W^\triangle(\tau,\xi)=\alpha^\triangle\beta^\triangle\gamma^\triangle$.
\item If $d(\tau,\omega)_{k_1}=4$ and $d(\tau,\omega)_{k_2}=4$, then, with the notation of the picture on the bottom left in Figure \ref{Fig:triangles_quivers}, we set $W^\triangle(\tau,\xi)=(\delta_0^\triangle+\delta_1^\triangle)\beta^\triangle\gamma^\triangle$.
\end{itemize}
\end{enumerate}
\end{defi}

\begin{figure}[!ht]
                \centering
                \includegraphics[scale=.5]{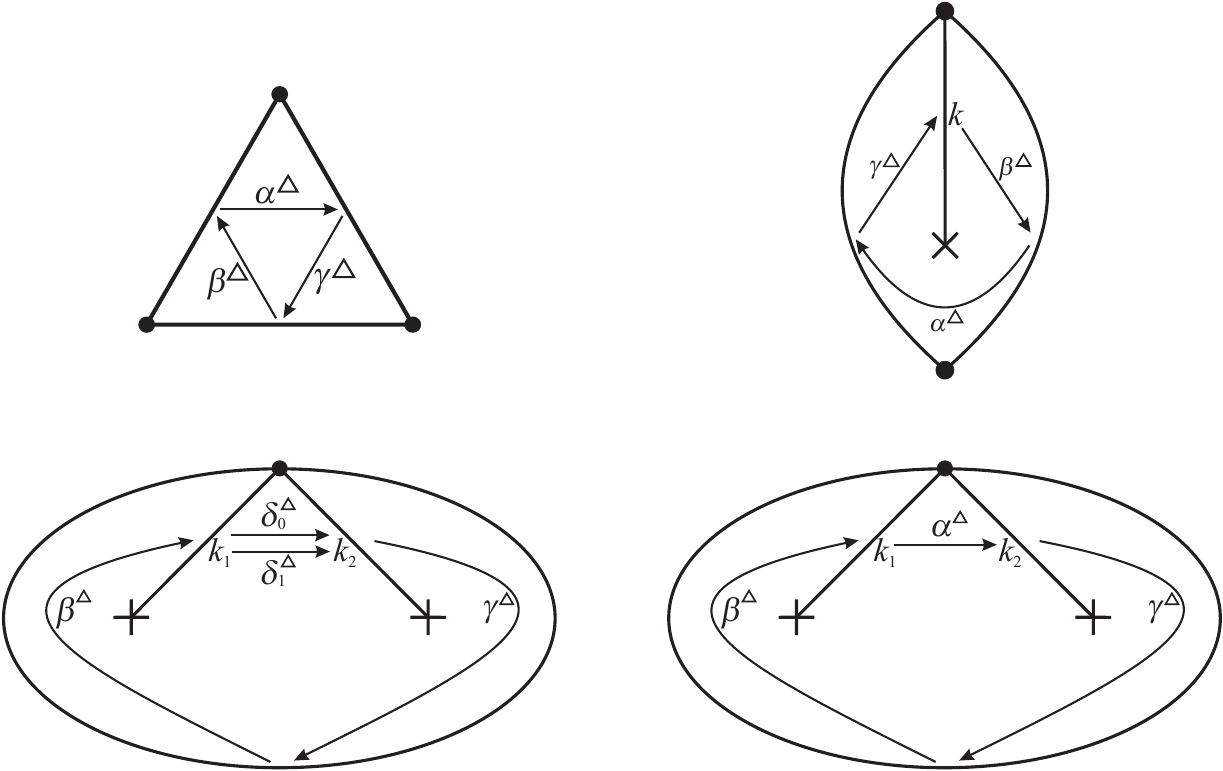}\caption{Notation for the definition of $W^\triangle(\tau,\xi)$.}
                \label{Fig:triangles_quivers}
        \end{figure}

For the next definition, we remind the reader that $\SSigma=\surf$ is assumed to be either unpunctured or once-punctured closed.

\begin{defi}\label{def:W(tau,xi)}\cite[Definition 6.10]{GLF2}\label{def:Jacobian-alg-of-colored-triang} Let $\SSigmaw=(\Sigma,\marked,\orb,\omega)$ be a surface with weighted orbifold points, and $(\tau,\xi)$ a colored triangulation of $\SSigmaw=(\Sigma,\marked,\orb,\omega)$.
\begin{enumerate}
\item
The \emph{potential associated to $(\tau,\xi)$} is
$$
W(\tau,\xi)\coloneqq \sum_{\triangle}W^\triangle(\tau,\xi)\in \usualRA{A(\tau,\xi)} \subseteq\compRA{A(\tau,\xi)},
$$
where the sum runs over all interior triangles $\triangle$ of $\tau$;
\item the \emph{Jacobian algebra associated to $(\tau,\xi)$} is the quotient
$$
\jacobalg{(A(\tau,\xi),W(\tau,\xi))}\coloneqq \compRA{A(\tau,\xi)}/J(W(\tau,\xi)),
$$
where the \emph{Jacobian ideal} $J(W(\tau,\xi))$ is defined according to Definition \ref{def:paths-potentials-cyclic-derivatives-jacobian-algebras}, cf. \cite[Definition 3.11]{GLF1}.
\end{enumerate}
\end{defi}

For detailed examples of the basic arithmetic in $\usualRA{A(\tau,\xi)}$ and in the Jacobian algebra $\jacobalg{(A(\tau,\xi),W(\tau,\xi))}$, we kindly refer the reader to \cite[Example 6.11, Example 6.12 and Section 13]{GLF2} and \cite[Example 6.2.18]{GeuenichPHD}. See Example \ref{ex:jac-algs-pentagon-2orb-pts-arb-weights} below as well.  

\begin{thm}\cite[Theorems 10.1 and 10.2]{GLF2}\label{thm:Jac-algs-are-fin-dim}
    Let $\SSigmaw$ be an unpunctured surface with weighted orbifold points.
    \begin{enumerate}
    \item
    For every colored triangulation $(\tau,\xi)$ of $\SSigmaw$, the Jacobian algebra $\jacobalg{(A(\tau,\xi),W(\tau,\xi))}$ is $F$-linearly isomorphic to $\usualRA{A(\tau,\xi)}/J_0(W(\tau,\xi))$ and its dimension over the ground field $F$ is finite.
    \item For every pari $(\tau,\xi_{1})$ and $(\tau,\xi_{2})$ of colored triangulations of $\SSigmaw$ with same underlying triangulation $\tau$, the following statements are equivalent:
\begin{enumerate}
 \item $[\xi_{1}] = [\xi_{2}]$ in the first cohomology group $H^{1}(C^{\bullet}(\tau))$;
\item the Jacobian algebras $\jacobalg{A(\tau, \xi_{1}), W(\tau, \xi_{1})}$ and $\jacobalg{A(\tau, \xi_{2}), W(\tau, \xi_{2})}$ are isomorphic through an $F$-linear ring isomorphism acting as the identity on the set of idempotents $\{e_{k} \suchthat k \in \tau\}$.
\end{enumerate}
    \end{enumerate}
\end{thm}

\begin{remark} When $\SSigma$ is once-punctured closed,  $[\xi_{1}] = [\xi_{2}]$ in cohomology implies $\jacobalg{A(\tau, \xi_{1}), W(\tau, \xi_{1})}\cong\jacobalg{A(\tau, \xi_{2}), W(\tau, \xi_{2})}$ through an $F$-linear ring isomorphism acting as the identity on $\{e_{k} \suchthat k \in \tau\}$.
\end{remark}

\begin{ex}\label{ex:jac-algs-pentagon-2orb-pts-arb-weights} Consider the triangulations $\tau$ and $\sigma$ of the pentagon with two orbifold points shown in Figure \ref{Fig:pentagon_two_orb_points}. Therein we can visualize not only the quivers $\overline{Q}(\tau)$ and $\overline{Q}(\sigma)$, but all the cells conforming the bases of the chain complexes $C_\bullet(\tau)$ and $C_\bullet(\sigma)$: the shaded regions are the 2-cells, the arrows are the 1-cells, and the vertices of $\overline{Q}(\tau)$ and $\overline{Q}(\sigma)$ are the 0-cells.
Take arbitrary $1$-cocycles $\xi\in Z^1(\tau)\subseteq C^1(\tau)$ and $\phi\in Z^1(\sigma)\subseteq C^1(\sigma)$, and an arbitrary choice of weights $\omega:\orb\rightarrow\{1,4\}$. In view of Theorem \ref{thm:Jac-algs-are-fin-dim}, in Table \ref{table-jac-algs-pentagon-2orb-pts-arb-weights} we can visualize the Jacobian algebras $\jacobalg{A(\tau,\xi),W(\tau,\xi))}$ and $\jacobalg{A(\sigma,\phi),W(\sigma,\phi)}$, for we see the species with potential $(A(\tau,\xi),W(\tau,\xi))$ and $(A(\sigma,\phi),W(\sigma,\phi))$, as well as all the cyclic derivatives of the potentials $W(\tau,\xi)$ and $W(\sigma,\phi)$.

\begin{center}
{\tiny
\begin{tabular}{|c|c|c|}
\hline
$\omega(q_1),\omega(q_2)$ & $\jacobalg{A(\tau,\xi),W(\tau,\xi))}$ & $\jacobalg{A(\sigma,\phi),W(\sigma,\phi)}$\\
\hline
$1,1$ &
\begin{tabular}{c}
$\xymatrix{
F \ar@/_0.5pc/[rr]_{F\otimes_FF} \ar@/^0.5pc/[rr]^{F\otimes_FF} & & F \ar[dl]^{L\otimes_FF} & \\
 & L \ar[ul]^{F\otimes_FL} \ar[dr]^{L^{\theta^{\xi_\beta}}\otimes_LL} & \\
 L \ar[ur]^{L^{\theta^{\xi_\gamma}}\otimes_LL} & & L \ar[ll]^{L^{\theta^{\xi_\alpha}}\otimes_LL} \ar[r]_{L^{\theta^{\xi_\nu}}\otimes_LL} & L
}$ \\
$W(\tau,\xi)=\alpha\beta\gamma+\delta_0\varepsilon\eta+\delta_1\varepsilon u\eta$ \\
\begin{tabular}{ll}
$\partial_\alpha W(\tau,\xi)=\beta\gamma$ & $\partial_{\delta_0}W(\tau,\xi)=\varepsilon\eta$ 
 \\
$\partial_\beta W(\tau,\xi)=\gamma\alpha$ & $\partial_{\delta_1}W(\tau,\xi)=\varepsilon u\eta$  
 \\
$\partial_\gamma W(\tau,\xi)=\alpha\beta$  & 
$\partial_{\varepsilon}W(\tau,\xi)=\eta\delta_0+u\eta\delta_1$ 
 \\
$\partial_{\nu}W(\tau,\xi)=0$& $\partial_{\eta}W(\tau,\xi)=\delta_0\varepsilon+\delta_1\varepsilon u$ 
\end{tabular}
\end{tabular}
& 
\begin{tabular}{c}
$\xymatrix{
 & F \ar[dr]|-{L\otimes_FF} &  & F \ar[dr]^{L\otimes_FF} & & \\
L \ar[ur]^{F\otimes_FL} &  & L \ar[ll]^{L^{\theta^{\phi_\beta}}\otimes_LL} \ar[ur]|-{F\otimes_FL} & & L \ar[ll]^{L^{\theta^{\phi_\varepsilon}}\otimes_L L} \ar[r]_{L^{\theta^{\phi_\nu}}\otimes_LL} & L 
}$ \\ 
$W(\sigma,\phi)=\alpha\beta\gamma+\delta\varepsilon\eta$  \\
\begin{tabular}{ll}
$\partial_\alpha W(\sigma,\phi)=\beta\gamma$ & $\partial_{\delta}W(\sigma,\phi)=\varepsilon\eta$ 
 \\
$\partial_\beta W(\sigma,\phi)=\frac{1}{2}(\gamma\alpha$ & $\partial_{\eta}W(\sigma,\phi)=\delta\varepsilon$  \\
$\phantom{\partial_\beta W(\sigma,\phi)=}+(-1)^{\phi_\beta}u^{-1}\gamma\alpha u)$ & $\partial_{\nu}W(\sigma,\phi)=0$ \\
$\partial_\gamma W(\sigma,\phi)=\alpha\beta$  & 
 \\
$\partial_{\varepsilon}W(\sigma,\phi)=\frac{1}{2}(\eta\delta$   & \\
$\phantom{\partial_{\varepsilon}W(\sigma,\phi)=}+(-1)^{\phi_\varepsilon}u^{-1}\eta\delta u)$&
\end{tabular}
\end{tabular}
\\
\hline
$1,4$ &
\begin{tabular}{c}
$\xymatrix{
F \ar[rr]^{E\otimes_FF} & & E \ar[dl]^{L^{\theta^{\xi_\eta}}\otimes_LE} & \\
 & L \ar[ul]^{F\otimes_FL} \ar[dr]^{L^{\theta^{\xi_\beta}}\otimes_LL} & \\
 L \ar[ur]^{L^{\theta^{\xi_\gamma}}\otimes_LL} & & L \ar[ll]^{L^{\theta^{\xi_\alpha}}\otimes_LL} \ar[r]_{L^{\theta^{\xi_\nu}}\otimes_LL} & L
}$ \\
$W(\tau,\xi)=\alpha\beta\gamma+\delta\varepsilon\eta$ \\
\begin{tabular}{ll}
$\partial_\alpha W(\tau,\xi)=\beta\gamma$ & $\partial_{\delta}W(\tau,\xi)=\varepsilon\eta$ 
 \\
$\partial_\beta W(\tau,\xi)=\gamma\alpha$ & 
$\partial_{\varepsilon}W(\tau,\xi)=\eta\delta$ 
 \\
$\partial_\gamma W(\tau,\xi)=\alpha\beta$  & 
$\partial_{\eta}W(\tau,\xi)= \frac{1}{2}(\delta\varepsilon$  \\
$\partial_{\nu}W(\tau,\xi)=0$ & $\phantom{\partial_{\eta}W(\tau,\xi)=}+(-1)^{\xi_\eta}u^{-1}\delta\varepsilon u)$
\end{tabular}
\end{tabular}
& 
\begin{tabular}{c}
$\xymatrix{
 & F \ar[dr]|-{L\otimes_FF} &  & E \ar[dr]^{L^{\theta^{\phi_\eta}}\otimes_LE} & & \\
L \ar[ur]^{F\otimes_FL} &  & L \ar[ll]^{L^{\theta^{\phi_\beta}}\otimes_LL} \ar[ur]|-{E^{\theta^{\phi_\delta}}\otimes_LL} & & L \ar[ll]^{L^{\theta^{\phi_\varepsilon}}\otimes_L L} \ar[r]_{L^{\theta^{\phi_\nu}}\otimes_LL} & L 
}$ \\ 
$W(\sigma,\phi)=\alpha\beta\gamma+\delta\varepsilon\eta$ \\
\begin{tabular}{ll}
$\partial_\alpha W(\sigma,\phi)=\beta\gamma$ & $\partial_{\delta}W(\sigma,\phi)=\varepsilon\eta$ 
 \\
$\partial_\beta W(\sigma,\phi)=\frac{1}{2}(\gamma\alpha$ & $\partial_{\varepsilon}W(\sigma,\phi)=\eta\delta$  
 \\
 $\phantom{\partial_\beta W(\sigma,\phi)=}+(-1)^{\phi_\beta}u^{-1}\gamma\alpha u)$ & 
$\partial_{\eta}W(\sigma,\phi)=\delta\varepsilon$ 
 \\
$\partial_\gamma W(\sigma,\phi)=\alpha\beta$ & $\partial_{\nu}W(\sigma,\phi)=0$
\end{tabular}
\end{tabular}
\\
\hline
$4,1$ &
\begin{tabular}{c}
$\xymatrix{
E \ar[rr]^{F\otimes_FE} & & F \ar[dl]^{L\otimes_FF} & \\
 & L \ar[ul]^{E^{\theta^{\xi_\varepsilon}}\otimes_LL} \ar[dr]^{L^{\theta^{\xi_\beta}}\otimes_LL} & \\
 L \ar[ur]^{L^{\theta^{\xi_\gamma}}\otimes_LL} & & L \ar[ll]^{L^{\theta^{\xi_\alpha}}\otimes_LL} \ar[r]_{L^{\theta^{\xi_\nu}}\otimes_LL} & L
}$ \\
$W(\tau,\xi)=\alpha\beta\gamma+\delta\varepsilon\eta$ \\
\begin{tabular}{ll}
$\partial_\alpha W(\tau,\xi)=\beta\gamma$ & $\partial_{\delta}W(\tau,\xi)=\varepsilon\eta$ 
 \\
$\partial_\beta W(\tau,\xi)=\gamma\alpha$ & 
$\partial_{\varepsilon}W(\tau,\xi)=\frac{1}{2}(\eta\delta$ 
 \\
$\partial_\gamma W(\tau,\xi)=\alpha\beta$  & 
$\phantom{\partial_{\varepsilon}W(\tau,\xi)=}+(-1)^{\xi_\varepsilon}u^{-1}\eta\delta u)$
 \\
$\partial_{\nu}W(\tau,\xi)=0$& $\partial_{\eta}W(\tau,\xi)=\delta\varepsilon$ 
\end{tabular}
\end{tabular}
& 
\begin{tabular}{c}
$\xymatrix{
 & E \ar[dr]|-{\quad L^{\theta^{\phi_\gamma}}\otimes_LE} &  & F \ar[dr]^{L\otimes_FF} & & \\
L \ar[ur]|-{E^{\theta^{\phi_{\alpha}}}\otimes_LL\quad} &  & L \ar[ll]^{L^{\theta^{\phi_\beta}}\otimes_LL} \ar[ur]|-{\quad F\otimes_FL} & & L \ar[ll]^{L^{\theta^{\phi_\varepsilon}}\otimes_L L} \ar[r]_{L^{\theta^{\phi_\nu}}\otimes_LL} & L 
}$ \\ 
$W(\sigma,\phi)=\alpha\beta\gamma+\delta\varepsilon\eta$ \\
\begin{tabular}{ll}
$\partial_\alpha W(\sigma,\phi)=\beta\gamma$ & $\partial_{\delta}W(\sigma,\phi)=\varepsilon\eta$ 
 \\
$\partial_\beta W(\sigma,\phi)=\gamma\alpha$ & $\partial_{\varepsilon}W(\sigma,\phi)=\frac{1}{2}(\eta\delta$  
 \\
$\partial_\gamma W(\sigma,\phi)=\alpha\beta$  & 
$\phantom{\partial_{\varepsilon}W(\sigma,\phi)=}+(-1)^{\xi_\varepsilon}u^{-1}\eta\delta u)$
 \\
$\partial_{\nu}W(\sigma,\phi)=0$ & $\partial_{\eta}W(\sigma,\phi)=\delta\varepsilon$ 
\end{tabular}
\end{tabular}
\\
\hline
$4,4$ &
\begin{tabular}{c}
$\xymatrix{
E \ar@/_0.5pc/[rr]_{E^{\rho^{\ell+2}}\otimes_EE} \ar@/^0.5pc/[rr]^{E^{\rho^\ell}\otimes_EE} & & E \ar[dl]^{L^{\theta^{\xi_\eta}}\otimes_LE} & \\
 & L \ar[ul]^{E^{\theta^{\xi_\varepsilon}}\otimes_LL} \ar[dr]^{L^{\theta^{\xi_\beta}}\otimes_LL} & \\
 L \ar[ur]^{L^{\theta^{\xi_\gamma}}\otimes_LL} & & L \ar[ll]^{L^{\theta^{\xi_\alpha}}\otimes_LL} \ar[r]_{L^{\theta^{\xi_\nu}}\otimes_LL} & L
}$ \\
$W(\tau,\xi)=\alpha\beta\gamma+(\delta_0+\delta_1)\varepsilon\eta$ \\
\begin{tabular}{ll}
$\partial_\alpha W(\tau,\xi)=\beta\gamma$ & $\partial_{\delta_0}W(\tau,\xi)=\frac{1}{2}(\varepsilon\eta$ 
 \\
$\partial_\beta W(\tau,\xi)=\gamma\alpha$ & $\phantom{\partial_{\delta_0}W(\tau,\xi)=}+\rho^{-l}(v^{-1})\varepsilon\eta v)$
 \\
$\partial_\gamma W(\tau,\xi)=\alpha\beta$  & 
$\partial_{\delta_1}W(\tau,\xi)=\frac{1}{2}(\varepsilon\eta$  
 \\
$\partial_{\nu}W(\tau,\xi)=0$ & $\phantom{\partial_{\delta_1}W(\tau,\xi)=}+\rho^{-l-2}(v^{-1})\varepsilon\eta v)$\\
 & $\partial_{\varepsilon}W(\tau,\xi)=\eta(\delta_0+\delta_1)$ \\
 & $\partial_{\eta}W(\tau,\xi)=(\delta_0+\delta_1)\varepsilon$ 
\end{tabular}
\end{tabular}
& 
\begin{tabular}{c}
$\xymatrix{
 & E \ar[dr]|-{\quad L^{\theta^{\phi_\gamma}}\otimes_LE} &  & E \ar[dr]^{L^{\theta^{\phi_\eta}}\otimes_LE} & & \\
L \ar[ur]|-{E^{\theta^{\phi_\alpha}}\otimes_LL\quad } &  & L \ar[ll]^{L^{\theta^{\phi_\beta}}\otimes_LL} \ar[ur]|-{\qquad E^{\theta^{\phi_\delta}}\otimes_LL} & & L \ar[ll]^{L^{\theta^{\phi_\varepsilon}}\otimes_L L} \ar[r]_{L^{\theta^{\phi_\nu}}\otimes_LL} & L 
}$ \\ 
$W(\sigma,\phi)=\alpha\beta\gamma+\delta\varepsilon\eta$  \\
\begin{tabular}{ll}
$\partial_\alpha W(\sigma,\phi)=\beta\gamma$ & $\partial_{\delta}W(\sigma,\phi)=\varepsilon\eta$ 
 \\
$\partial_\beta W(\sigma,\phi)=\gamma\alpha$ & $\partial_{\varepsilon}W(\sigma,\phi)=\eta\delta$  
 \\
$\partial_\gamma W(\sigma,\phi)=\alpha\beta$  & 
$\partial_{\eta}W(\sigma,\phi)=\delta\varepsilon$ 
 \\
& $\partial_{\nu}W(\sigma,\phi)=0$
\end{tabular}
\end{tabular}
\\
\hline
\end{tabular}
\captionof{table}{Jacobian algebras from Example \ref{ex:jac-algs-pentagon-2orb-pts-arb-weights} \label{table-jac-algs-pentagon-2orb-pts-arb-weights}}
}
\end{center}
\end{ex}

The initial interest in the Jacobian algebras $\jacobalg{A(\tau, \xi), W(\tau, \xi)}$ stems from the relation to cluster combinatorics provided by Proposition \ref{prop:our-species-realize-FeShTu-matrices} and the following result on their good behavior under mutations of species with potential.

\begin{thm}\cite[Theorem 7.1]{GLF2}\label{thm:SPs-well-behaved-under-colored-flips-and-muts-arb-weights} Let $\SSigmaw$ be a surface with weighted orbifold points, either unpunctured or once-punctured closed, and $(\tau,\xi)$ and $(\sigma,\phi)$ be colored triangulations of $\SSigma$.  If $(\sigma,\phi)$ can be obtained from $(\tau,\xi)$ by the colored flip of an arc $k \in\tau$, then the species with potential $(A(\sigma,\phi),W(\sigma,\phi))$ and $\mu_k(A(\tau, \xi), W(\tau, \xi))$ are right-equivalent.
\end{thm}

 The notion of \emph{colored flip} of colored triangulations is defined in \cite[Definition 5.8]{GLF2} (see also \cite[Examples 5.9 and 5.10]{GLF2}), while those of \emph{right equivalence} and \emph{mutation} of species with potential are defined in \cite[Definitions 3.11 and 3.19]{GLF1} (see also \cite[Remark 3.20]{GLF1} and \cite[Section 2.1]{GLF2}). The latter two notions were of course inspired by the corresponding ones introduced by Derksen-Weyman-Zelevinsky in \cite{DWZ1}.

\subsubsection{\textbf{The semilinear clannish algebra of a colored triangulation}}\label{subsubsec:arb-weights-semilin-clan-alg}

Fix a degree-$4$ datum $E/F$, and let $(\tau,\xi)$ be a colored triangulation of a surface with weighted orbifold points $\SSigmaw=(\surfnoM,\marked,\orb,\omega)$. 
We associate to $(\tau,\xi)$ a semilinear clannish algebra $K_{\bldsigma(\tau,\xi)}\compactQhat (\tau)/I(\tau,\xi)$ as follows.

Define $\compactQhat (\tau)$ to be the quiver obtained from $\overline{Q}(\tau)$ by adding one loop at each pending arc of $\tau$, which we assume to be \emph{special}. 
Since $\overline{Q}(\tau)$ is loop-free, this means we are taking the set $\bbS(\tau)$ of special loops in $\compactQhat (\tau)$ to be the set of all loops in $\compactQhat (\tau)$, or said another way, $\compactQhat (\tau)_{1}=\overline{Q}(\tau)_{1}\sqcup \bbS(\tau)$.

Let $K=L$. 
To each arrow $a\in\compactQhat (\tau)_1$ we define a field automorphism $\sigma(\tau,\xi)_a\in \Gal(L/F)\subseteq \Aut(L)$ by
$$
\sigma(\tau,\xi)_a\coloneqq \begin{cases}
\theta^{\xi_a} & \text{if $\overline{Q}(\tau)_1=\compactQhat (\tau)_1\setminus \bbS(\tau)$};\\
\theta & \text{if $h(a)=t(a)$ and $d(\tau,\omega)_{h(a)}=1$};\\
\myid_L & \text{if $h(a)=t(a)$ and $d(\tau,\omega)_{h(a)}=4$}.
\end{cases}
$$
This information determines already a semilinear path algebra $K_{\bldsigma(\tau,\xi)}\compactQhat (\tau)$. 
Furthermore, to each loop $s\in \bbS(\tau)$ of $\compactQhat (\tau)$ with head and tail $k$, we attach the quadratic polynomial
$$
q_s(x)\coloneqq \begin{cases}
x^2-1\in L[x;\theta] & \text{if $d(\tau,\omega)_k=1$};\\
x^2-u\in L[x] & \text{if $d(\tau,\omega)_k=4$}.
\end{cases}
$$
This information determines the set  of special relations, defined by
\[
\begin{array}{cc}
S(\tau,\xi)=\{q(s)\mid s\in \bbS(\tau)=\compactQhat (\tau)_1\setminus \overline{Q}(\tau)_1 \}, & e_{k}K_{\bldsigma(\tau,\xi)}\compactQhat (\tau)e_{k}\ni q(s)= \begin{cases}
s^2-e_{k} & \text{if $d(\tau,\omega)_k=1$};\\
s^2-ue_{k} & \text{if $d(\tau,\omega)_k=4$}.
\end{cases}
\end{array}
\]
We define the two-sided ideal $I(\tau,\xi)=\langle Z(\tau,\xi)\cup S(\tau,\xi)\rangle$ in $K_{\bldsigma(\tau,\xi)}\compactQhat (\tau)$ by defining the set $Z(\tau,\xi)$ of zero-relations, as follows.
Suppose $\triangle$ is a triangle in $\tau$, say of one of the forms depicted in Figure \ref{Fig:rule-for-arrows-of-overlineQ}. 
Each such $\triangle$ gives rise to three distinct arrows of $\overline{Q}(\tau)$ subject to certain conditions, namely
\[
\begin{array}{cccc}
\alpha^{\triangle}, \beta^{\triangle},\gamma^{\triangle}\in \compactQhat (\tau)_1\setminus \bbS(\tau)=\overline{Q}(\tau)_1, & h(\alpha^{\triangle})=t(\gamma^{\triangle}), & h(\gamma^{\triangle})=t(\beta^{\triangle}),
& h(\beta^{\triangle})=t(\alpha^{\triangle}).
\end{array}
\] 
We now let $Z(\tau,\xi)$ be the union of the sets $Z(\tau,\xi,\triangle)=\{\alpha^{\triangle}\beta^{\triangle},\beta^{\triangle}\gamma^{\triangle},\gamma^{\triangle}\alpha^{\triangle}\}$ taken over all such $\triangle$.

\begin{ex}
For $k=1,\ldots,7$, the rings appearing in 
Tables \ref{table-semilinear-clannish-blocks-1-to-5} and \ref{table-semilinear-clannish-blocks-6-to-10} have the form $L_{\bldsigma(\tau,\xi)}\compactQhat (\tau)/I(\tau,\xi)$ where $(\tau,\xi)$ is a colored triangulation of a puzzle piece surface $\SSigma$ from Figure \ref{Fig:unpunct_puzzle_pieces}. Compare with Example \ref{ex-modulated-quivers-from-puzlle-pieces-defined-table}.
\end{ex}

\begin{prop}
\label{prop-algebras-are-semilinear-clannish}
If $(\tau,\xi)$ is a triangulation of a surface $\SSigmaw$ then $A=L_{\bldsigma(\tau,\xi)}\compactQhat (\tau)/I(\tau,\xi)$ is a semilinear clannish algebra which is  normally-bound, non-singular and of semisimple type.
\end{prop}
\begin{proof}
Let $\compactQhat=\compactQhat (\tau)$. 
In what follows we consider an element $i\in \compactQhat_{0}$ from the set $\compactQhat (\tau)_{0}$ of arcs in the triangulation $(\tau,\xi)$ of $\SSigmaw=(\Sigma,\marked,\orb,\omega)$. 
We fix notation for such an arc $i$ which depends on cases.
\begin{enumerate}[(a)]
    \item If $i$ is the edge of only one triangle in $(\tau,\xi)$, we denote it $\triangle(i)$;  let $e(i),f(i)\in \compactQhat_{0}$ be the arcs defining the remaining edges of $\triangle(i)$; and we write $\kappa(i),\lambda(i),\mu(i)$ for the (ordinary) arrows in $\overline{Q}(\tau)_1=\compactQhat_{1}\setminus \bbS(\tau)$ with $t(\kappa(i))=i=h(\lambda(i))$,  $t(\lambda(i))=e(i)=h(\mu(i))$ and $t(\mu(i))=f(i)=h(\kappa(i))$. 
    \item If $i$ is the edge of two distinct triangles in $(\tau,\xi)$ we denote them $\triangle_{-}(i)$ and $\triangle_{+}(i)$; let $e_{\pm}(i),f_{\pm}(i)\in \compactQhat_{0}$ be the other arcs defining edges of $\triangle_{\pm}(i)$; and let  $\kappa_{\pm}(i),\lambda_{\pm}(i),\mu_{\pm}(i)\in\compactQhat_{1}\setminus \bbS(\tau)$ where  $t(\kappa_{\pm}(i))=i=h(\lambda_{\pm}(i))$,  $t(\lambda_{\pm}(i))=e_{\pm}(i)=h(\mu_{\pm}(i))$ and $t(\mu_{\pm}(i))=f_{\pm}(i)=h(\kappa_{\pm}(i))$. 
\end{enumerate}
Note that: exactly one of (a) or (b) holds; $Z(\tau,\xi,\triangle(i))=\{\kappa(i)\lambda(i),\lambda(i)\mu(i),\mu(i)\kappa(i)\}$ in case (a); and $Z(\tau,\xi,\triangle_{\pm}(i))=\{\kappa_{\pm}(i)\lambda_{\pm}(i),\lambda_{\pm}(i)\mu_{\pm}(i),\mu_{\pm}(i)\kappa_{\pm}(i)\}$ in case (b).
     As observed in Example \ref{ex-main-example-for-semilinear-clannish-blocks}, it is straightforward to check conditions (S) and (i)--(iii) hold from Definition \ref{defi-semilinear-clannish-algebra}.

(Q) Note firstly that there is at most one special loop incident at $i$, and if there is one, $i$ must be a pending arc, and we must be in case (a) above. 
In case (a) the arrow $\lambda(i)$ (respectively,  $\kappa(i)$) is the unique ordinary arrow with head (respectively, tail) $i$. Hence (Q) holds in case (a), whether or not $i$ is pending.

In case (b) $i$ must be non-pending, meaning there are no special loops at $i$. Hence (Q) holds in case (b) as well, since the arrows with head (respectively, tail) $i$ are precisely  $\lambda_{\pm}(i)$ (respectively,  $\kappa_{\pm}(i)$). 

(Z) Let $y$ be an ordinary arrow, and hence an element of $\overline{Q}(\tau)_1=\compactQhat_{1}\setminus \bbS(\tau)$, and write $h(y)=i$. Hence in case (a) we have $y=\lambda(i)$, in which case $\kappa(i)y\in Z(\tau,\xi,\triangle(i))$. 
Likewise in case (b) we have, after relabeling, $y=\lambda_{+}(i)$, and therefore $\kappa_{+}(i)y\in Z(\tau,\xi,\triangle_{+}(i))$. 
By having shown condition (Q) holds, and since no special loop occurs in a path from $Z(\tau,\xi)$, we now have that (Z) holds. 
\end{proof}

\begin{ex}\label{ex:semilin-clannish-algs-pentagon-2orb-pts-arb-weights} Consider the triangulations $\tau$ and $\sigma$ of the pentagon with two orbifold points from Figure \ref{Fig:pentagon_two_orb_points}. 
Take arbitrary $1$-cocycles $\xi\in Z^1(\tau)\subseteq C^1(\tau)$ and $\phi\in Z^1(\sigma)\subseteq C^1(\tau)$, and take arbitrary weights $\omega:\orb\rightarrow\{1,4\}$. In Table \ref{table-semilin-clannish-algs-pentagon-2orb-pts-arb-weights} we visualize the semilinear clannish algebras $L_{\bldsigma(\tau,\xi)}\compactQhat (\tau)/I(\tau,\xi)$ and $L_{\bldsigma(\sigma,\phi)}\compactQhat (\sigma)/I(\sigma,\phi)$.
\begin{center}
{\tiny
\begin{tabular}{|c|c|c|}
\hline
$\omega(q_1),\omega(q_2)$ & $L_{\bldsigma(\tau,\xi)}\compactQhat (\tau)/I(\tau,\xi)$ & $L_{\bldsigma(\sigma,\phi)}\compactQhat (\sigma)/I(\sigma,\phi)$\\
\hline
$1,1$ &
\begin{tabular}{c}
$\xymatrix{
L \ar@(l,u)^(.4){L^{\theta}\otimes_LL} \ar[rr]^{L^{\theta^{\xi_\delta}}} & & L \ar@(u,r)_(.4){\qquad\qquad\qquad\qquad\qquad\quad L^{\theta}\otimes_LL} \ar[dl]^{L^{\theta^{\xi_\eta}}\otimes_LL} & \\
 & L \ar[ul]^{L^{\theta^{\xi_\varepsilon}}\otimes_FL} \ar[dr]^{L^{\theta^{\xi_\beta}}\otimes_LL} & \\
 L \ar[ur]^{L^{\theta^{\xi_\gamma}}\otimes_LL} & & L \ar[ll]^{L^{\theta^{\xi_\alpha}}\otimes_LL} \ar[r]_{L^{\theta^{\xi_\nu}}\otimes_LL} & L
}$ \\
\begin{tabular}{l}
$I(\tau,\xi)=\langle\alpha\beta,\beta\gamma,\gamma\alpha,\delta\varepsilon,\varepsilon\eta,\eta\delta,$\\ 
$\phantom{I(\tau,\xi)=}s_1^2-e_1,s_2^2-e_2\rangle$  
\end{tabular}
\end{tabular}
& 
\begin{tabular}{c}
$\xymatrix{
 & L \ar@(lu,ru)^{L^{\theta}\otimes_LL} \ar[dr]|-{\quad L^{\theta^{\phi_\gamma}}\otimes_LL} &  & L \ar@(lu,ru)^{L^{\theta}\otimes_LL} \ar[dr]^{L^{\theta^{\phi_\eta}}\otimes_LL} & & \\
L \ar[ur]|-{L^{\theta^{\phi_\alpha}}\otimes_LL\quad } &  & L \ar[ll]^{L^{\theta^{\phi_\beta}}\otimes_LL} \ar[ur]|-{\qquad L^{\theta^{\phi_\delta}}\otimes_LL} & & L \ar[ll]^{L^{\theta^{\phi_\varepsilon}}\otimes_L L} \ar[r]_{L^{\theta^{\phi_\nu}}}\otimes_LL & L 
}$ \\ 
\begin{tabular}{l}
$I(\sigma,\phi)=\langle\alpha\beta,\beta\gamma,\gamma\alpha,\delta\varepsilon,\varepsilon\eta,\eta\delta,$\\ 
$\phantom{I(\sigma,\phi)=}s_1^2-e_1,s_2^2-e_2\rangle$  
\end{tabular}
\end{tabular}
\\
\hline
$1,4$ &
\begin{tabular}{c}
$\xymatrix{
L \ar@(l,u)^(.4){L^{\theta}\otimes_LL} \ar[rr]^{L^{\theta^{\xi_\delta}}\otimes_LL} & & L \ar@(u,r)_(.4){\qquad\qquad\qquad\qquad\qquad\quad L\otimes_LL} \ar[dl]^{L^{\theta^{\xi_\eta}}\otimes_LL} & \\
 & L \ar[ul]^{L^{\theta^{\xi_\varepsilon}}\otimes_FL} \ar[dr]^{L^{\theta^{\xi_\beta}}\otimes_LL} & \\
 L \ar[ur]^{L^{\theta^{\xi_\gamma}}\otimes_LL} & & L \ar[ll]^{L^{\theta^{\xi_\alpha}}\otimes_LL} \ar[r]_{L^{\theta^{\xi_\nu}}\otimes_LL} & L
}$ \\
\begin{tabular}{l}
$I(\tau,\xi)=\langle\alpha\beta,\beta\gamma,\gamma\alpha,\delta\varepsilon,\varepsilon\eta,\eta\delta,$\\ 
$\phantom{I(\tau,\xi)=}s_1^2-e_1,s_2^2-ue_2\rangle$  
\end{tabular}
\end{tabular}
& 
\begin{tabular}{c}
$\xymatrix{
 & L \ar@(lu,ru)^{L^{\theta}\otimes_LL} \ar[dr]|-{\quad L^{\theta^{\phi_\gamma}}\otimes_LL} &  & L \ar@(lu,ru)^{L\otimes_LL} \ar[dr]^{L^{\theta^{\phi_\eta}}\otimes_LL} & & \\
L \ar[ur]|-{L^{\theta^{\phi_\alpha}}\otimes_LL\quad} &  & L \ar[ll]^{L^{\theta^{\phi_\beta}}\otimes_LL} \ar[ur]|-{\qquad L^{\theta^{\phi_\delta}}\otimes_LL} & & L \ar[ll]^{L^{\theta^{\phi_\varepsilon}}\otimes_L L} \ar[r]_{L^{\theta^{\phi_\nu}}}\otimes_LL & L 
}$ \\ 
\begin{tabular}{l}
$I(\sigma,\phi)=\langle\alpha\beta,\beta\gamma,\gamma\alpha,\delta\varepsilon,\varepsilon\eta,\eta\delta,$\\ 
$\phantom{I(\sigma,\phi)=}s_1^2-e_1,s_2^2-ue_2\rangle$  
\end{tabular}
\end{tabular}
\\
\hline
$4,1$ &
\begin{tabular}{c}
$\xymatrix{
L \ar@(l,u)^(.4){L\otimes_LL} \ar[rr]^{L^{\theta^{\xi_\delta}}\otimes_LL} & & L \ar@(u,r)_(.4){\qquad\qquad\qquad\qquad\qquad\quad L^{\theta}\otimes_LL} \ar[dl]^{L^{\theta^{\xi_\eta}}\otimes_LL} & \\
 & L \ar[ul]^{L^{\theta^{\xi_\varepsilon}}\otimes_FL} \ar[dr]^{L^{\theta^{\xi_\beta}}\otimes_LL} & \\
 L \ar[ur]^{L^{\theta^{\xi_\gamma}}\otimes_LL} & & L \ar[ll]^{L^{\theta^{\xi_\alpha}}\otimes_LL} \ar[r]_{L^{\theta^{\xi_\nu}}\otimes_LL} & L
}$ \\
\begin{tabular}{l}
$I(\tau,\xi)=\langle\alpha\beta,\beta\gamma,\gamma\alpha,\delta\varepsilon,\varepsilon\eta,\eta\delta,$\\ 
$\phantom{I(\tau,\xi)=}s_1^2-ue_1,s_2^2-e_2\rangle$  
\end{tabular}
\end{tabular}
& 
\begin{tabular}{c}
$\xymatrix{
 & L \ar@(lu,ru)^{L\otimes_LL} \ar[dr]|-{\quad L^{\theta^{\phi_\gamma}}\otimes_LL} &  & L \ar@(lu,ru)^{L^{\theta}\otimes_LL} \ar[dr]^{L^{\theta^{\phi_\eta}}\otimes_LL} & & \\
L \ar[ur]|-{L^{\theta^{\phi_\alpha}}\otimes_LL\quad } &  & L \ar[ll]^{L^{\theta^{\phi_\beta}}\otimes_LL} \ar[ur]|-{\qquad L^{\theta^{\phi_\delta}}\otimes_LL} & & L \ar[ll]^{L^{\theta^{\phi_\varepsilon}}\otimes_L L} \ar[r]_{L^{\theta^{\phi_\nu}}}\otimes_LL & L 
}$ \\ 
\begin{tabular}{l}
$I(\sigma,\phi)=\langle\alpha\beta,\beta\gamma,\gamma\alpha,\delta\varepsilon,\varepsilon\eta,\eta\delta,$\\ 
$\phantom{I(\sigma,\phi)=}s_1^2-ue_1,s_2^2-e_2\rangle$  
\end{tabular}
\end{tabular}
\\
\hline
$4,4$ &
\begin{tabular}{c}
$\xymatrix{
L \ar@(l,u)^(.4){L\otimes_LL} \ar[rr]^{L^{\theta^{\xi_\delta}}} & & L \ar@(u,r)_(.4){\qquad\qquad\qquad\qquad\qquad\quad L\otimes_LL} \ar[dl]^{L^{\theta^{\xi_\eta}}\otimes_LL} & \\
 & L \ar[ul]^{L^{\theta^{\xi_\varepsilon}}\otimes_FL} \ar[dr]^{L^{\theta^{\xi_\beta}}\otimes_LL} & \\
 L \ar[ur]^{L^{\theta^{\xi_\gamma}}\otimes_LL} & & L \ar[ll]^{L^{\theta^{\xi_\alpha}}\otimes_LL} \ar[r]_{L^{\theta^{\xi_\nu}}\otimes_LL} & L
}$ \\
\begin{tabular}{l}
$I(\tau,\xi)=\langle\alpha\beta,\beta\gamma,\gamma\alpha,\delta\varepsilon,\varepsilon\eta,\eta\delta,$\\ 
$\phantom{I(\tau,\xi)=}s_1^2-ue_1,s_2^2-ue_2\rangle$  
\end{tabular}
\end{tabular}
& 
\begin{tabular}{c}
$\xymatrix{
 & L \ar@(lu,ru)^{L\otimes_LL} \ar[dr]|-{\quad L^{\theta^{\phi_\gamma}}\otimes_LL} &  & L \ar@(lu,ru)^{L\otimes_LL} \ar[dr]^{L^{\theta^{\phi_\eta}}\otimes_LL} & & \\
L \ar[ur]|-{L^{\theta^{\phi_\alpha}}\otimes_LL\quad } &  & L \ar[ll]^{L^{\theta^{\phi_\beta}}\otimes_LL} \ar[ur]|-{\qquad L^{\theta^{\phi_\delta}}\otimes_LL} & & L \ar[ll]^{L^{\theta^{\phi_\varepsilon}}\otimes_L L} \ar[r]_{L^{\theta^{\phi_\nu}}}\otimes_LL & L 
}$ \\ 
\begin{tabular}{l}
$I(\sigma,\phi)=\langle\alpha\beta,\beta\gamma,\gamma\alpha,\delta\varepsilon,\varepsilon\eta,\eta\delta,$\\ 
$\phantom{I(\sigma,\phi)=}s_1^2-ue_1,s_2^2-ue_2\rangle$  
\end{tabular}
\end{tabular}
\\
\hline
\end{tabular}
\captionof{table}{Semilinear clannish algebras from Example \ref{ex:semilin-clannish-algs-pentagon-2orb-pts-arb-weights} 
\label{table-semilin-clannish-algs-pentagon-2orb-pts-arb-weights}}
}
\end{center}
\end{ex}

\subsection{Constant weights: algebras defined over \mathinsection{\bbC/\R}{}}\label{subsec:constant-weights-algs-def-over-C/R}

\,

As pointed out in Remark \ref{rem:orb-not-empty-and-omega-constant-1=>degree-2-datum}, if $\orb\neq \varnothing$ and $\omega\equiv 1$, then one may simply work over a degree-$2$ datum (not necessarily extendable to a degree-$4$ datum, e.g. $\mathbb{C}/\mathbb{R}$), and all the constructions and results from Subsections \ref{subsubsec:arb-weights-Jac-alg} and \ref{subsubsec:arb-weights-semilin-clan-alg} are valid.

\begin{ex}\label{ex:jac-algs-and-semilin-clan-algs-pentagon-2orb-pts-constant-weight=1} Consider the triangulations $\tau$ and $\sigma$ of the pentagon with two orbifold points shown in Figure \ref{Fig:pentagon_two_orb_points}, and the constant function $\omega:\orb\rightarrow\{1,4\}$ with value $1$. Let  $\xi\in Z^1(\tau)\subseteq C^1(\tau)$ and $\phi\in Z^1(\sigma)\subseteq Z^1(\sigma)$ be arbitrary $1$-cocycles. Working over the degree-2 datum $\mathbb{C}/\mathbb{R}$ (thus, $\theta:\mathbb{C}\rightarrow\mathbb{C}$ is the usual complex conjugation and the square of $u=i\in\mathbb{C}$ is $-1\in\mathbb{R}$), in the following two tables we can visualize the Jacobian algebras $\jacobalg{A}(\tau,\xi),W(\tau,\xi))$ and $\jacobalg{A}(\sigma,\phi),W(\sigma,\phi)$ as well as the semiinear clannish algebras $\mathbb{C}_{\bldsigma(\tau,\xi)}\compactQhat (\tau)/I(\tau,\xi)$ and $\mathbb{C}_{\bldsigma(\sigma,\phi)}\compactQhat (\sigma)/I(\sigma,\phi)$.
\begin{center}
{\small
\begin{tabular}{|c|c|c|}
\hline
$\omega(q_1),\omega(q_2)$ & $\jacobalg{A}(\tau,\xi),W(\tau,\xi))$ & $\jacobalg{A}(\sigma,\phi),W(\sigma,\phi)$\\
\hline
$1,1$ &
\begin{tabular}{c}
$\xymatrix{
\mathbb{R} \ar@/_0.5pc/[rr]_{\mathbb{R}\otimes_\mathbb{R}\mathbb{R}} \ar@/^0.5pc/[rr]^{\mathbb{R}\otimes_\mathbb{R}\mathbb{R}} & & \mathbb{R} \ar[dl]^{\mathbb{C}\otimes_\mathbb{R}\mathbb{R}} & \\
 & \mathbb{C} \ar[ul]^{\mathbb{R}\otimes_\mathbb{R}\mathbb{C}} \ar[dr]^{\mathbb{C}^{\theta^{\xi_\beta}}\otimes_\mathbb{C}\mathbb{C}} & \\
 \mathbb{C} \ar[ur]^{\mathbb{C}^{\theta^{\xi_\gamma}}\otimes_\mathbb{C}\mathbb{C}} & & \mathbb{C} \ar[ll]^{\mathbb{C}^{\theta^{\xi_\alpha}}\otimes_\mathbb{C}\mathbb{C}} \ar[r]_{\mathbb{C}^{\theta^{\xi_\nu}}\otimes_\mathbb{C}\mathbb{C}} & \mathbb{C}
}$ \\
$W(\tau,\xi)=\alpha\beta\gamma+\delta_0\varepsilon\eta+\delta_1\varepsilon i\eta$ \\
\begin{tabular}{ll}
$\partial_\alpha W(\tau,\xi)=\beta\gamma$ & $\partial_{\delta_0}W(\tau,\xi)=\varepsilon\eta$ 
 \\
$\partial_\beta W(\tau,\xi)=\gamma\alpha$ & $\partial_{\delta_1}W(\tau,\xi)=\varepsilon i\eta$  
 \\
$\partial_\gamma W(\tau,\xi)=\alpha\beta$  & 
$\partial_{\varepsilon}W(\tau,\xi)=\eta\delta_0+i\eta\delta_1$ 
 \\
$\partial_{\nu}W(\tau,\xi)=0$& $\partial_{\eta}W(\tau,\xi)=\delta_0\varepsilon+\delta_1\varepsilon i$ 
\end{tabular}
\end{tabular}
& 
\begin{tabular}{c}
$\xymatrix{
 & \mathbb{R} \ar[dr]|-{\mathbb{C}\otimes_\mathbb{R}\mathbb{R}} &  & \mathbb{R} \ar[dr]^{\mathbb{C}\otimes_\mathbb{R}\mathbb{R}} & & \\
\mathbb{C} \ar[ur]^{\mathbb{R}\otimes_\mathbb{R}\mathbb{C}} &  & \mathbb{C} \ar[ll]^{\mathbb{C}^{\theta^{\phi_\beta}}\otimes_{\mathbb{C}}\mathbb{C}} \ar[ur]|-{\mathbb{R}\otimes_\mathbb{R}\mathbb{C}} & & \mathbb{C} \ar[ll]^{\mathbb{C}^{\theta^{\phi_\varepsilon}}\otimes_\mathbb{C} \mathbb{C}} \ar[r]_{\mathbb{C}^{\theta^{\phi_\nu}}\otimes_\mathbb{C}\mathbb{C}} & \mathbb{C} 
}$ \\ 
$W(\sigma,\phi)=\alpha\beta\gamma+\delta\varepsilon\eta$  \\
\begin{tabular}{ll}
$\partial_\alpha W(\sigma,\phi)=\beta\gamma$ & $\partial_{\delta}W(\sigma,\phi)=\varepsilon\eta$ 
 \\
$\partial_\beta W(\sigma,\phi)=\frac{1}{2}(\gamma\alpha$ & $\partial_{\eta}W(\sigma,\phi)=\delta\varepsilon$  \\
$\phantom{\partial_\beta W(\sigma,\phi)=}-(-1)^{\phi_\beta}i\gamma\alpha i)$ & $\partial_{\nu}W(\sigma,\phi)=0$ \\
$\partial_\gamma W(\sigma,\phi)=\alpha\beta$  & 
 \\
$\partial_{\varepsilon}W(\sigma,\phi)=\frac{1}{2}(\eta\delta$   & \\
$\phantom{\partial_{\varepsilon}W(\sigma,\phi)=}-(-1)^{\phi_\varepsilon}i\eta\delta i)$&
\end{tabular}
\end{tabular}
\\
\hline
\end{tabular}
}
\end{center}

\begin{center}
{\small
\begin{tabular}{|c|c|c|}
\hline
$\omega(q_1),\omega(q_2)$ & $\mathbb{C}_{\bldsigma(\tau,\xi)}\compactQhat (\tau)/I(\tau,\xi)$ & $\mathbb{C}_{\bldsigma(\sigma,\phi)}\compactQhat (\sigma)/I(\sigma,\phi)$\\
\hline
$1,1$ &
\begin{tabular}{c}
$\xymatrix{
\mathbb{C} \ar@(l,u)^(.4){\mathbb{C}^\theta\otimes_{\mathbb{C}}\mathbb{C}} \ar[rr]^{\mathbb{C}^{\theta^{\xi_\delta}}\otimes_\mathbb{C}\mathbb{C}} & & \mathbb{C} \ar@(u,r)_(.4){\qquad\qquad\qquad\qquad \mathbb{C}^\theta\otimes_{\mathbb{C}}\mathbb{C}} \ar[dl]^{\mathbb{C}^{\theta^{\xi_\eta}}\otimes_\mathbb{C}\mathbb{C}} & \\
 & \mathbb{C} \ar[ul]^{\mathbb{C}^{\theta^{\xi_{\varepsilon}}}\otimes_\mathbb{C}\mathbb{C}} \ar[dr]^{\mathbb{C}^{\theta^{\xi_\beta}}\otimes_\mathbb{C}\mathbb{C}} & \\
 \mathbb{C} \ar[ur]^{\mathbb{C}^{\theta^{\xi_\gamma}}\otimes_\mathbb{C}\mathbb{C}} & & \mathbb{C} \ar[ll]^{\mathbb{C}^{\theta^{\xi_\alpha}}\otimes_\mathbb{C}\mathbb{C}} \ar[r]_{\mathbb{C}^{\theta^{\xi_\nu}}\otimes_\mathbb{C}\mathbb{C}} & \mathbb{C}
}$ \\
\begin{tabular}{l}
$I(\tau,\xi)=\langle\alpha\beta,\beta\gamma,\gamma\alpha,\delta\varepsilon,\varepsilon\eta,\eta\delta,$\\ 
$\phantom{I(\sigma,\phi)=}s_1^2-e_1,s_2^2-e_2\rangle$  
\end{tabular}
\end{tabular}
& 
\begin{tabular}{c}
$\xymatrix{
 & \mathbb{C} \ar@(lu,ru)^{\mathbb{C}^\theta\otimes_{\mathbb{C}}\mathbb{C}} \ar[dr]|-{\mathbb{C}^{\theta^{\phi_\gamma}}\otimes_{\mathbb{C}}\mathbb{C}} &  & \mathbb{C} \ar@(lu,ru)^{\mathbb{C}^\theta\otimes_{\mathbb{C}}\mathbb{C}} \ar[dr]^{\mathbb{C}^{\theta^{\phi_\eta}}\otimes_\mathbb{C}\mathbb{C}} & & \\
\mathbb{C} \ar[ur]|-{\mathbb{C}^{\theta^{\phi_\alpha}}\otimes_{\mathbb{C}}\mathbb{C}\quad } &  & \mathbb{C} \ar[ll]^{\mathbb{C}^{\theta^{\phi_\beta}}\otimes_{\mathbb{C}}\mathbb{C}} \ar[ur]|-{\quad \mathbb{C}^{\theta^{\phi_{\delta}}\otimes_{\mathbb{C}}\mathbb{C}}} & & \mathbb{C} \ar[ll]^{\mathbb{C}^{\theta^{\phi_\varepsilon}}\otimes_{\mathbb{C}} \mathbb{C}} \ar[r]_{\mathbb{C}^{\theta^{\phi_\nu}}\otimes_\mathbb{C}\mathbb{C}} & \mathbb{C} 
}$ \\ 
\begin{tabular}{l}
$I(\sigma,\phi)=\langle\alpha\beta,\beta\gamma,\gamma\alpha,\delta\varepsilon,\varepsilon\eta,\eta\delta,$\\ 
$\phantom{I(\sigma,\phi)=}s_1^2-e_1,s_2^2-e_2\rangle$  
\end{tabular}
\end{tabular}
\\
\hline
\end{tabular}
}
\end{center}
\end{ex}

Thus, the rest of this short subsection will be devoted to giving a small modification of the constructions from Subsections \ref{subsubsec:arb-weights-Jac-alg} and \ref{subsubsec:arb-weights-semilin-clan-alg} that allow to work over a degree-$2$ datum also when $\omega\equiv 4$.

Let $\SSigmaw=(\Sigma,\marked,\orb,\omega)$ be a surface with weighted orbifold points, with $\omega:\orb\rightarrow\{1,4\}$ constant taking the value 4, let $\tau$ be a triangulation of $\SSigma$. For each $k\in\tau$, set
\begin{equation}\label{eq:omega=4=>half-the-weights}
\delta(\tau,\omega)_k\coloneqq 
\frac{d(\tau,\omega)_k}{2}=\begin{cases}1 & \text{if $k$ is not a pending arc;}\\
2 & \text{if $k$ is a pending arc.}
\end{cases}.
\end{equation}

Set $\delta\coloneqq \lcm\{\delta(\tau,\omega)_k\suchthat k\in\tau\}$, and
let $L/F$ be a degree-$\delta$ datum. Thus, $\delta=2$ if $\orb\neq\varnothing$, and $\delta=1$ if $\orb=\varnothing$. Notice that one may take $L/F$ to be $\mathbb{C}/\mathbb{R}$ if $\delta=2$, or $L=F=\mathbb{C}$ if $\delta=1$.

For each $k\in\tau$ we set $F_k/F$ to be the unique degree-$\delta(\tau,\omega)_k$ field subextension of $L/F$, and denote $G_k=\Gal(F_k/F)$. We also denote $G_{j,k}=\Gal(F_j\cap F_k/F)$ for $j,k\in\tau$. Thus:
$$
G_{j,k}=\begin{cases}
\{\myid_{L},\theta\} & \text{if both $j$ and $k$ are pending arcs;}\\
\{\myid_{F}\} & \text{if at least one of $j$ and $k$ is not a pending arc.}
\end{cases}
$$

Since we are assuming $\omega\equiv 4$, we have $\delta(\tau,\omega)_k=1$ for every non-pending arc $k$. Thus, our construction of a Jacobian algebra and a clannish algebra will be independent of any 1-cocycle $\xi$. For this reason, in Subsections \ref{subsubsec:constant-weights-Jac-alg} and \ref{subsubsec:constant-weights-semilin-clan-alg}, we shall work with plain triangulations $\tau$ instead of colored triangulations $(\tau,\xi)$.

\subsubsection{\textbf{Jacobian algebra}}\label{subsubsec:constant-weights-Jac-alg}
We are working under the assumptions and notations described in the first few paragraphs of the current Subsection \ref{subsec:constant-weights-algs-def-over-C/R}.
Let $\tau$ be a triangulation of $\SSigmaw$.
We define a modulating function $c(\tau):Q(\tau,\omega)_1\rightarrow\bigcup_{j,k\in\tau}G_{j,k}$ as follows. Take an arrow $a\in Q(\tau,\xi)_1$.
\begin{enumerate}
\item If $\min\{\delta(\tau,\omega)_{h(a)},\delta(\tau,\omega)_{t(a)}\}=1$, set
$$
c(\tau)_{a}=\myid\in G_{h(a),t(a)}.
$$
\item If $\min\{\delta(\tau,\omega)_{h(a)},\delta(\tau,\omega)_{t(a)}\}=2$, then both $h(a)$ and $t(a)$ are pending arcs, and the quiver $Q(\tau,\omega)$ has exactly two arrows going from $t(a)$ to $h(a)$, say $\beta_0$ and $\beta_1$. We set
    $$
    c(\tau)_{a}=\begin{cases}
    \myid_L & \text{if $a=\beta_0$;}\\
    \theta & \text{if $a=\beta_1$.}
    \end{cases}
    $$
\end{enumerate}

\begin{ex}\label{ex-modulated-quivers-from-puzlle-pieces-defined-table-constantweights}
For $k=8,9,10$, the weighted quiver $(Q,\bldD)$ and the modulating function $Q_{1}\to\bigcup_{i,j}G_{i,j}$ appearing in the column labeled ``Block $k$'' in Table \ref{table-Jacobian-blocks-6-to-10} have the form $(Q(\tau,\omega),\blddelta(\tau,\omega))$ and  $c(\tau)$, respectively, for some triangulation $\tau$ of a puzzle piece surface from Figure \ref{Fig:unpunct_puzzle_pieces}, where $\blddelta(\tau,\omega)$ is the tuple defined by \eqref{eq:omega=4=>half-the-weights}.
\end{ex}

With the modulating function $c(\tau)$ we form the species
\begin{align*}
(\mathbf{F},\boldsymbol{A}(\tau))&\coloneqq ((F_k)_{k\in\tau},(A(\tau)_a)_{a\in Q(\tau,\omega)_1}),  \qquad \text{where} \\
A(\tau)_a&\coloneqq F_{h(a)}^{c(\tau)_{a}}\otimes_{F_{h(a)}\cap F_{t(a)}} F_{t(a)}. 
\end{align*}

We write $R\coloneqq \times_{k\in\tau} F_k$ and $A(\tau)\coloneqq \bigoplus_{a\in Q(\tau,\omega)_1}A(\tau)_a$. It is clear that $R$ is a semisimple ring and $A(\tau)$ is an $R$-$R$-bimodule.

One easily verifies that the pair $(\mathbf{F},\boldsymbol{A}(\tau))$ satisfies Proposition \ref{prop:our-species-realize-FeShTu-matrices} too, i.e., we are obtaining a species realization of one of the $2^{|\orb|}$ skew-symmetrizable matrices associated to $\tau$ by Felikson-Shapiro-Tumarkin \cite{FeShTu-orbifolds}, cf. \cite[Remark 3.5-(2)]{GLF2}.

\begin{remark}\label{rem:decomposing-Lotimes_FL} If $\min\{\delta(\tau,\omega)_{h(a)},\delta(\tau,\omega)_{t(a)}\}=2$, so that $h(a)$ and $t(a)$ are pending and $Q(\tau,\omega)$ has exactly two arrows going from $j\coloneqq t(a)$ to $k\coloneqq h(a)$, namely $\beta_0$ and $\beta_1$ (one of them being $a$ of course), then 
$$
F_{k}\otimes_{F} F_{j} = L\otimes_FL
\cong 
(L^{\myid_L}\otimes_{L} L)
\oplus 
(L^{\theta}\otimes_{L} L)=
(F_{k}^{c(\tau)_{\beta_0}}\otimes_{F_k\cap F_j} F_{j})
\oplus 
(F_{k}^{c(\tau)_{\beta_1}}\otimes_{F_k\cap F_j} F_{j})
$$
and $L^{\myid_L}\otimes_{L} L
\not\cong 
L^{\theta}\otimes_{L} L$ as $L$-$L$-bimodules.
\end{remark}

We now move towards the definition of a natural potential $W(\tau)\in\usualRA{A}(\tau)$. There are some obvious cycles on $A(\tau)$, that we point to explicitly.

\begin{defi}\label{def:cycles-from-triangles-constantweights}
Let $\tau$ be a triangulation of $\SSigmaw$ and $\triangle$ be an interior triangle of $\tau$.
\begin{enumerate}
    \item If $\triangle$ does not contain any orbifold point, then, with the notation from the picture on the upper left in Figure \ref{Fig:triangles_quivers}), we set $W^\triangle(\tau);=\alpha^\triangle\beta^\triangle\gamma^\triangle$;
    \item if $\triangle$ contains exactly one orbifold point, let $k$ be the unique pending arc of $\tau$ contained in $\triangle$. Using the notation from the picture on the upper right in Figure \ref{Fig:triangles_quivers}, we set $W^\triangle(\tau)=\alpha^\triangle\beta^\triangle\gamma^\triangle$;
    \item if $\triangle$ contains exactly two orbifold points, let $k_1$ and $k_2$ be the two pending arcs of $\tau$ contained in $\triangle$, and assume that they are configured as in Figure \ref{Fig:triangles_quivers}. Then, with the notation of the picture on the bottom left in Figure \ref{Fig:triangles_quivers}, we set $W^\triangle(\tau)=(\delta_0^\triangle+\delta_1^\triangle)\beta^\triangle\gamma^\triangle$.
\end{enumerate}
\end{defi}

For the next definition, we remind the reader that $\SSigma=\surf$ is assumed to be either unpunctured or once-punctured closed.

\begin{defi}\label{def:W(tau,xi)-constantweights} Let $\SSigmaw=(\Sigma,\marked,\orb,\omega)$ be a surface with weighted orbifold points, with $\omega\equiv4$, and let $\tau$ a triangulation of $\SSigmaw=(\Sigma,\marked,\orb,\omega)$.
\begin{enumerate}
\item
The \emph{potential associated to $\tau$} is
$$
W(\tau)\coloneqq \sum_{\triangle}W^\triangle(\tau)\in \usualRA{A}(\tau) \subseteq\compRA{A}(\tau),
$$
where the sum runs over all interior triangles $\triangle$  of $\tau$;
\item the \emph{Jacobian algebra associated to $\tau$} is the quotient
$$
\jacobalg{A}(\tau),W(\tau)\coloneqq \compRA{A}(\tau)/J(W(\tau)),
$$
where the \emph{Jacobian ideal} $J(W(\tau))\subseteq \compRA{A}(\tau)$ is defined according to \cite[Definition 3.11]{GLF1}.
\end{enumerate}
\end{defi}

For detailed examples of the basic arithmetic in $\usualRA{A}(\tau)$ and in the Jacobian algebra $\jacobalg{A}(\tau),W(\tau)$, we kindly refer the reader to  \cite[Example 4.8 and Section 9]{GLF1}.

\begin{remark} 
\begin{enumerate}
    \item if $\orb\neq\varnothing$ and $\omega\equiv 4$, then $(A(\tau),W(\tau))$ is the species with potential associated to $\tau$ in \cite{GLF1} (although therein punctures are allowed, whereas here they are excluded);
    \item if $\orb=\varnothing$ and $\marked\subseteq\partial\Sigma$, then $(A(\tau),W(\tau))$ is the quiver with potential defined in \cite{LF-QPsurfs1}, and $\jacobalg{A}(\tau),W(\tau)$ is the gentle algebra studied in \cite{ABCP}.
\end{enumerate}    
\end{remark}

The same argument as the one given in the proof of \cite[Theorem 10.2]{GLF2} can be applied to obtain a proof of the next result.

\begin{thm}\label{thm:Jac-algs-are-fin-dim-constantweights}
    Let $\SSigmaw$ be an unpunctured surface with weighted orbifold points, with $\omega\equiv4$. For every triangulation $\tau$ of $\SSigmaw$, the Jacobian algebra $\jacobalg{A}(\tau),W(\tau)$ is $F$-linearly isomorphic to $\usualRA{A}(\tau)/J_0(W(\tau))$ and its dimension over the ground field $F$ is finite.
\end{thm}

\begin{ex}\label{ex:jac-algs-pentagon-2orb-pts-constant-weight=4} Consider the triangulations $\tau$ and $\sigma$ of the pentagon with two orbifold points shown in Figure \ref{Fig:pentagon_two_orb_points}, and the constant function $\omega:\orb\rightarrow\{1,4\}$ with value $4$. Working over the degree-2 datum $\mathbb{C}/\mathbb{R}$ (thus, $\theta:\mathbb{C}\rightarrow\mathbb{C}$ is the usual complex conjugation and the square of $u=i\in\mathbb{C}$ is $-1\in\mathbb{R}$), in the following table we can visualize the Jacobian algebras $\jacobalg{A}(\tau),W(\tau))$ and $\jacobalg{A}(\sigma),W(\sigma)$.
\begin{center}
{\small
\begin{tabular}{|c|c|c|}
\hline
$\omega(q_1),\omega(q_2)$ & $\jacobalg{A}(\tau),W(\tau))$ & $\jacobalg{A}(\sigma),W(\sigma)$\\
\hline
$4,4$ &
\begin{tabular}{c}
$\xymatrix{
\mathbb{C} \ar@/_0.5pc/[rr]_{\mathbb{C}^{\mathbb{\theta}}\otimes_\mathbb{C}\mathbb{C}} \ar@/^0.5pc/[rr]^{\mathbb{C}\otimes_\mathbb{C}\mathbb{C}} & & \mathbb{C} \ar[dl]^{\mathbb{R}\otimes_\mathbb{R}\mathbb{C}} & \\
 & \mathbb{R} \ar[ul]^{\mathbb{C}\otimes_\mathbb{R}\mathbb{R}} \ar[dr]^{\mathbb{R}\otimes_\mathbb{R}\mathbb{R}} & \\
 \mathbb{R} \ar[ur]^{\mathbb{R}\otimes_\mathbb{R}\mathbb{R}} & & \mathbb{R} \ar[ll]^{\mathbb{R}\otimes_\mathbb{R}\mathbb{R}} \ar[r]_{\mathbb{R}\otimes_\mathbb{R}\mathbb{R}} & \mathbb{R}
}$ \\
$W(\tau,\xi)=\alpha\beta\gamma+(\delta_0+\delta_1)\varepsilon\eta$ \\
\begin{tabular}{ll}
$\partial_\alpha W(\tau,\xi)=\beta\gamma$ & $\partial_{\delta_0}W(\tau,\xi)=\frac{1}{2}(\varepsilon\eta$ 
 \\
$\partial_\beta W(\tau,\xi)=\gamma\alpha$ & $\phantom{\partial_{\delta_0}W(\tau,\xi)=}-i\varepsilon\eta i)$
 \\
$\partial_\gamma W(\tau,\xi)=\alpha\beta$  & 
$\partial_{\delta_1}W(\tau,\xi)=\frac{1}{2}(\varepsilon\eta$  
 \\
$\partial_{\nu}W(\tau,\xi)=0$ & $\phantom{\partial_{\delta_1}W(\tau,\xi)=}+i\varepsilon\eta i)$\\
 & $\partial_{\varepsilon}W(\tau,\xi)=\eta(\delta_0+\delta_1)$ \\
 & $\partial_{\eta}W(\tau,\xi)=(\delta_0+\delta_1)\varepsilon$ 
\end{tabular}
\end{tabular}
& 
\begin{tabular}{c}
$\xymatrix{
 & \mathbb{C} \ar[dr]|-{\mathbb{R}\otimes_\mathbb{R}\mathbb{C}} &  & \mathbb{C} \ar[dr]^{\mathbb{R}\otimes_\mathbb{R}\mathbb{C}} & & \\
\mathbb{R} \ar[ur]|-{\mathbb{C}\otimes_\mathbb{R}\mathbb{R}} &  & \mathbb{R} \ar[ll]^{\mathbb{R}\otimes_{\mathbb{R}}\mathbb{R}} \ar[ur]|-{\mathbb{C}\otimes_\mathbb{R}\mathbb{R}} & & \mathbb{R} \ar[ll]^{\mathbb{R}\otimes_\mathbb{R} \mathbb{R}} \ar[r]_{\mathbb{R}\otimes_\mathbb{R}\mathbb{R}} & \mathbb{R} 
}$ \\ 
$W(\sigma,\phi)=\alpha\beta\gamma+\delta\varepsilon\eta$  \\
\begin{tabular}{ll}
$\partial_\alpha W(\sigma,\phi)=\beta\gamma$ & $\partial_{\delta}W(\sigma,\phi)=\varepsilon\eta$ 
 \\
$\partial_\beta W(\sigma,\phi)=\gamma\alpha$ & $\partial_{\varepsilon}W(\sigma,\phi)=\eta\delta$  
 \\
$\partial_\gamma W(\sigma,\phi)=\alpha\beta$  & 
$\partial_{\eta}W(\sigma,\phi)=\delta\varepsilon$ 
 \\
& $\partial_{\nu}W(\sigma,\phi)=0$
\end{tabular}
\end{tabular}
\\
\hline
\end{tabular}
}
\end{center}
\end{ex}

\begin{thm}\cite[Theorem 8.4]{GLF1}\label{thm:SPs-well-behaved-under-flips-and-muts-constant-weight=4} Let $\SSigmaw$ be an unpunctured surface with weighted orbifold points, with $\omega\equiv 4$, and let $\tau$ and $\kappa$ be colored triangulations of $\SSigma$.  If $\kappa$ can be obtained from $\tau$ by the flip of an arc $k \in\tau$, then the species with potential $(A(\kappa),W(\kappa))$ and $\mu_k(A(\tau), W(\tau))$ are right-equivalent.
\end{thm}

\subsubsection{\textbf{Semilinear clannish algebra}}\label{subsubsec:constant-weights-semilin-clan-alg}
We maintain the assumptions and notations described in the first few paragraphs of the current Subsection \ref{subsec:constant-weights-algs-def-over-C/R}.
Let $\tau$ be a triangulation of $\SSigma$. 
We associate to $\tau$ a semilinear clannish algebra $K_{\bldsigma(\tau)}\compactQhat (\tau)/I(\tau)$ as follows.

Exactly as in Subsection \ref{subsubsec:arb-weights-semilin-clan-alg}, we set $\compactQhat (\tau)$ to be the quiver obtained from $\overline{Q}(\tau)$ by adding one special loop at each pending arc of $\tau$.

We set $K\coloneqq F$.
To every arrow $a\in\compactQhat (\tau)_1$ we attach the trivial field automorphism $\sigma(\tau,\xi)_a\coloneqq \myid_F\in \Gal(K/F)\subseteq \Aut(K)$.

This information determines already a path algebra $K_{\bldsigma(\tau)}\compactQhat (\tau)$. 
Furthermore, to each loop $s\in \bbS(\tau)$ of $\compactQhat (\tau)$ with head and tail $k$, we attach the quadratic polynomial
$$
q_s(x)\coloneqq 
x^2-u^2\in F[x].
$$
This information determines the set  of special relations, defined by
\[
\begin{array}{cc}
S(\tau)=\{q(s)\mid s\in \bbS(\tau)=\compactQhat (\tau)_1\setminus \overline{Q}(\tau)_1 \}, & e_{k}K_{\bldsigma(\tau)}\compactQhat (\tau)e_{k}\ni q(s)= 
s^2-u^2e_{k}.
\end{array}
\]
We define the two-sided ideal $I(\tau)=\langle Z(\tau)\cup S(\tau)\rangle$ in $K_{\bldsigma(\tau)}\compactQhat (\tau)$ by defining the set $Z(\tau)$ of zero-relations, as follows.
Suppose $\triangle$ is a triangle in $\tau$, say of one of the forms depicted in Figure \ref{Fig:rule-for-arrows-of-overlineQ}. 
Each such $\triangle$ gives rise to three distinct arrows of $\overline{Q}(\tau)$  subject to certain conditions, namely
\[
\begin{array}{cccc}
\alpha^{\triangle}, \beta^{\triangle},\gamma^{\triangle}\in \compactQhat (\tau)_1\setminus \bbS(\tau)=\overline{Q}(\tau)_1, & h(\alpha^{\triangle})=t(\gamma^{\triangle}), & h(\gamma^{\triangle})=t(\beta^{\triangle}),
& h(\beta^{\triangle})=t(\alpha^{\triangle}).
\end{array}
\] 
We now let $Z(\tau)$ be the union of the sets $Z(\tau,\triangle)=\{\alpha^{\triangle}\beta^{\triangle},\beta^{\triangle}\gamma^{\triangle},\gamma^{\triangle}\alpha^{\triangle}\}$ taken over all such $\triangle$.

\begin{ex}
For $k=8,9,10$, the ring appearing in the column labeled ``Block $k$'' in Table \ref{table-semilinear-clannish-blocks-6-to-10} has the form $K_{\bldsigma(\tau)}\compactQhat (\tau)/I(\tau)$ where $\tau$ is triangulation of a puzzle piece surface $\SSigma$ from Figure \ref{Fig:unpunct_puzzle_pieces}. 
\end{ex}

A minor modification of the proof of Proposition \ref{prop-algebras-are-semilinear-clannish} proves the next result.

\begin{prop}
\label{prop-algebras-are-semilinear-clannish-constantweights}
Let $\SSigmaw$ be a surface with weighted orbifold points, with $\omega\equiv4$. For every triangulation $\tau$ of $\SSigma$, $F_{\bldsigma(\tau)}\compactQhat (\tau)/I(\tau)$ is a clannish $F$-algebra which is  normally-bound, non-singular and of semisimple type.
\end{prop}

\begin{ex}\label{ex:semilin-clannish-algs-pentagon-2orb-pts-constant-weights} Consider the triangulations $\tau$ and $\sigma$ of the pentagon with two orbifold points shown in Figure \ref{Fig:pentagon_two_orb_points}, and the constant function $\omega:\orb\rightarrow\{1,4\}$ taking the value $4$. Working over the degree-2 datum $\mathbb{C}/\mathbb{R}$ (thus, $\theta:\mathbb{C}\rightarrow\mathbb{C}$ is the usual complex conjugation and the square of $u=i\in\mathbb{C}$ is $-1\in\mathbb{R}$), in the following table we can visualize the clannish algebras $\mathbb{R}_{\bldsigma(\tau,\xi)}\compactQhat (\tau)/I(\tau,\xi)$ and $\mathbb{R}_{\bldsigma(\sigma,\phi)}\compactQhat (\sigma)/I(\sigma,\phi)$.
\begin{center}
{\small
\begin{tabular}{|c|c|c|}
\hline
$\omega(q_1),\omega(q_2)$ & $\jacobalg{A}(\tau),W(\tau))$ & $\jacobalg{A}(\sigma),W(\sigma)$\\
\hline
$4,4$ &
\begin{tabular}{c}
$\xymatrix{
\mathbb{R} \ar@(l,u)^(.4){\mathbb{R}\otimes_{\mathbb{R}}\mathbb{R}} \ar[rr]^{\mathbb{R}\otimes_\mathbb{R}\mathbb{R}} & & \mathbb{R} \ar@(u,r)_(.4){\qquad\qquad\qquad\quad\mathbb{R}\otimes_{\mathbb{R}}\mathbb{R}} \ar[dl]^{\mathbb{R}\otimes_\mathbb{R}\mathbb{R}} & \\
 & \mathbb{R} \ar[ul]^{\mathbb{R}\otimes_\mathbb{R}\mathbb{R}} \ar[dr]^{\mathbb{R}\otimes_\mathbb{R}\mathbb{R}} & \\
 \mathbb{R} \ar[ur]^{\mathbb{R}\otimes_\mathbb{R}\mathbb{R}} & & \mathbb{R} \ar[ll]^{\mathbb{R}\otimes_\mathbb{R}\mathbb{R}} \ar[r]_{\mathbb{R}\otimes_\mathbb{R}\mathbb{R}} & \mathbb{R}
}$ \\
\begin{tabular}{l}
$I(\tau,\xi)=\langle\alpha\beta,\beta\gamma,\gamma\alpha,\delta\varepsilon,\varepsilon\eta,\eta\delta,$\\ 
$\phantom{I(\sigma,\phi)=}s_1^2+e_1,s_2^2+e_2\rangle$  
\end{tabular}
\end{tabular}
& 
\begin{tabular}{c}
$\xymatrix{
 & \mathbb{R} \ar@(lu,ru)^{\mathbb{R}\otimes_{\mathbb{R}}\mathbb{R}} \ar[dr]|-{\mathbb{R}\otimes_\mathbb{R}\mathbb{R}} &  & \mathbb{R} \ar@(lu,ru)^{\mathbb{R}\otimes_{\mathbb{R}}\mathbb{R}} \ar[dr]^{\mathbb{R}\otimes_\mathbb{R}\mathbb{R}} & & \\
\mathbb{R} \ar[ur]|-{\mathbb{R}\otimes_\mathbb{R}\mathbb{R}} &  & \mathbb{R} \ar[ll]^{\mathbb{R}\otimes_{\mathbb{R}}\mathbb{R}} \ar[ur]|-{\mathbb{R}\otimes_\mathbb{R}\mathbb{R}} & & \mathbb{R} \ar[ll]^{\mathbb{R}\otimes_\mathbb{R} \mathbb{R}} \ar[r]_{\mathbb{R}\otimes_\mathbb{R}\mathbb{R}} & \mathbb{R} 
}$ \\ 
\begin{tabular}{l}
$I(\sigma,\phi)=\langle\alpha\beta,\beta\gamma,\gamma\alpha,\delta\varepsilon,\varepsilon\eta,\eta\delta,$\\ 
$\phantom{I(\sigma,\phi)=}s_1^2+e_1,s_2^2+e_2\rangle$  
\end{tabular}
\end{tabular}
\\
\hline
\end{tabular}
}
\end{center}
\end{ex}

\vspace{2mm}

\section{Main result: Morita equivalence between Jacobian and semilinear clannish algebras}
\label{sec-morita-equivalence-between-algebras-from-triangulations}

Let us describe how we shall glue the building blocks from Section \ref{sec:building-blocks} together in order to construct `bigger' algebras. This way of gluing algebras along vertices was first introduced by Br\"{u}stle in~\cite{Brustle-kit} (in broader generality). The \emph{$\rho$-block decompositions} from \cite{GLFS-schemes} can be thought of as a ``reverse-engineering'' of the gluing process.

\begin{enumerate}
\item Take finitely many disjoint copies of blocks from one and only one of the following four sets:
\begin{enumerate}
    \item\label{item:Jac-blocks-1-to-7} the Jacobian blocks $1,\ldots,7$ in Tables \ref{table-Jacobian-blocks-1-to-5} and \ref{table-Jacobian-blocks-6-to-10};
    \item\label{item:Jac-blocks-8-9-10} the Jacobian blocks $8,9,10$ in Table \ref{table-Jacobian-blocks-6-to-10};
    \item\label{item:semilin-clan-blocks-1-to-7} the semilinear clannish blocks $1,\ldots,7$ in Tables \ref{table-semilinear-clannish-blocks-1-to-5} and \ref{table-semilinear-clannish-blocks-6-to-10};
   \item\label{item:semilin-clan-blocks-8-9-10} the semilinear clannish blocks $8,9,10$ in Table \ref{table-semilinear-clannish-blocks-6-to-10}.
\end{enumerate}
As said in the opening paragraphs of Section \ref{sec:building-blocks}, in each of these copies, some entries of the weight triple $\bldD=(d_1,d_2,d_3)$ appear enclosed in a small circle. This means that the corresponding vertex is an \emph{outlet}, allowed to be matched and glued to another outlet. Notice that there is never a loop based at an outlet, and that on all the outlets of the block copies chosen appears the same field (it is $L$ if the block copies are taken from \eqref{item:Jac-blocks-1-to-7} or \eqref{item:Jac-blocks-8-9-10}, and it is $F$ if the block copies are taken from \eqref{item:semilin-clan-blocks-1-to-7} or \eqref{item:semilin-clan-blocks-8-9-10}).
\item Fix a partial matching of this set of outlets, never matching two outlets of the same block copy.
\item Glue the puzzle pieces along the matched outlets.
\item After the gluing, choose an arbitrary subset of the set of outlets that were not matched (hence also not glued) to any other outlets, and delete this subset and the arrows incident to its elements.
\end{enumerate}

\begin{remark}
\label{remark-on-gluing-procedures}
    Recall that, as outlined in \S\ref{subsec:surfaces-and-triangs}, our meaning of triangulation in this paper is precisely that of an \emph{ideal triangulation} as defined in \cite{GLF1}.   
    We explain here, how the notion of a \emph{puzzle}-\emph{piece} \emph{decomposition} coming from \cite[Definition~2.8]{GLF1}, corresponds with the  gluing of building blocks (from Section \ref{sec:building-blocks}) above.

    Item (1) above corresponds to equipping oneself with several copies of the puzzle pieces from Figure \ref{Fig:unpunct_puzzle_pieces}. 
    In this way, choosing a side of a triangle correspond to choosing a vertex of a quiver of a block. 
    Furthermore, the \emph{outer side} of a triangle corresponds to an outlet.  
    Thus item (2) above corresponds to equipping oneself with a partial matching of the outer sides of puzzle pieces from Figure \ref{Fig:unpunct_puzzle_pieces}, and item (3) corresponds to gluing them. 
    
     For the sake of Proposition \ref{prop:our-algs-are-block-decomposable} and Theorem \ref{thm:main-result-of-paper}, it is important to note  \cite[Theorem~2.7]{GLF1}, which says that every triangulation can be obtained by a suitable partial matching, as described above. 
     Table \ref{table-surface-blocks-1-to-5} (respectively, \ref{table-surface-blocks-6-to-10}) describes how the blocks in Tables \ref{table-Jacobian-blocks-1-to-5} and \ref{table-semilinear-clannish-blocks-1-to-5} (respectively, \ref{table-Jacobian-blocks-6-to-10}  and \ref{table-semilinear-clannish-blocks-6-to-10}) are given by weighted surfaces $\SSigmaw=(\Sigma,\marked,\orb,\omega)$. 
\end{remark}

\afterpage{%
    \clearpage
    \thispagestyle{empty}
    \begin{landscape}
        \centering 
        \hspace{-1cm}
\renewcommand{\arraystretch}{1.25}
{\small
\begin{tabular}{|c|c|c|c|c|c|}
\hline
 & Block 1 & Block 2 & Block 3 & Block 4 & Block 5  \\
\hline
\begin{tabular}{c}
Weight triple\\ 
$\left(\begin{array}{ccc} & d_1 &\\ d_2 & & d_3\end{array}\right)$\end{tabular} & 
$\left(\begin{array}{ccc} & \circled{2} &\\ \circled{2} & & \circled{2} \end{array}\right)$&
$\left(\begin{array}{ccc} & 1 &\\ \circled{2} & & \circled{2} \end{array}\right)$
&
$\left(\begin{array}{ccc} & 4 &\\ \circled{2} & & \circled{2}\end{array}\right)$ & $\left(\begin{array}{ccc} & \circled{2} &\\ 1 & & 1\end{array}\right)$ & $\left(\begin{array}{ccc} & \circled{2} &\\ 4 & & 4\end{array}\right)$ \\
\hline
\begin{tabular}{c}
Triangulation\\
$(\Sigma,\marked,\orb)$
\\{}\\{}
\end{tabular}
&
\begin{tabular}{c}
\includegraphics[trim={0 0 42.5cm 0},clip,scale=.125]{Fig_unpunct_puzzle_pieces.png}
\end{tabular}
&
\begin{tabular}{c}
\includegraphics[trim={19.5cm 0 22.5cm -0.5cm},clip,scale=.125]{Fig_unpunct_puzzle_pieces.png}
\end{tabular}
&
\begin{tabular}{c}
\includegraphics[trim={19.5cm 0 22.5cm -0.5cm},clip,scale=.125]{Fig_unpunct_puzzle_pieces.png}
\end{tabular}
&
\begin{tabular}{c}
\includegraphics[trim={39.5cm 0 0 -1cm},clip,scale=.125]{Fig_unpunct_puzzle_pieces.png}
\end{tabular}
&
\begin{tabular}{c}
\includegraphics[trim={39.5cm 0 0 -1cm},clip,scale=.125]{Fig_unpunct_puzzle_pieces.png}
\end{tabular}
\\
\hline
\begin{tabular}{c}
Jacobian\\
algebra\\
mnemotechnics\end{tabular}
& 
$
\xymatrix@C=0.7em@R=6em{
& L \ar[dr]_(0.6){\gamma}^(.4){L^{\theta^{\xi_\gamma}}\smallotimesL L} & \\
L \ar[ur]_(0.4){\alpha}^(.6){L^{\theta^{\xi_\alpha}}\smallotimesL L} & & L \ar[ll]_{\beta}^{L^{\theta^{\xi_\beta}}\smallotimesL L} 
}$
&
$\xymatrix@C=0.7em@R=6em{
& F \ar[dr]_(.6){\gamma}^(.4){L\smallotimesF F} & \\
L \ar[ur]_(.4){\alpha}^(.6){F\smallotimesF L} & & L \ar[ll]_{\beta}^{L^{\theta^{\xi_\beta}}\smallotimesL L} 
}$
&
$\xymatrix@C=0.7em@R=6em{
 & E  \ar[dr]_(.6){\gamma}^(.4){L^{\theta^{\xi_\gamma}}\smallotimesL E} &\\
 L \ar[ur]_(.4){\alpha}^(.6){E^{\theta^{\xi_\alpha}}\smallotimesL L} & & L\ar[ll]_{\beta}^{L^{\theta^{\xi_\beta}}\smallotimesL L}
}$
 &{
 $\xymatrix@R=2.5em{
  & L  \ar[ddr]_(.4){\gamma}^(.4){F\smallotimesF L} &  \\
  & &  \\
 F \ar[uur]_(.6){\alpha}^(.6){L\smallotimesF F}  & &  F \ar@/_0.7pc/[ll]^{\beta_0}_{F\smallotimesF F} \ar@/^0.9pc/[ll]_{\beta_1}^{F\smallotimesF F}
}$}
 &{
 $\xymatrix@R=2.5em{
  & L  \ar[ddr]_(.4){\gamma}^(.4){E^{\theta^{\xi_\gamma}}\smallotimesL L} &  \\
  & &  \\
 E \ar[uur]_(.6){\alpha}^(.6){L^{\theta^{\xi_\alpha}}\smallotimesL E}  & &  E \ar@/_0.7pc/[ll]^{\beta_0}_{E^{\rho^l}\smallotimesE E} \ar@/^0.9pc/[ll]_{\beta_1}^{E^{\rho^{l+2}}\smallotimesE E}
}$}
\\
\hline
\begin{tabular}{c}
Semilinear\\
clannish\\
algebra\\
mnemotechnics\end{tabular}
& 
$\xymatrix@C=1.2em@R=4.5em{& L \ar[dr]_(.6){\gamma}^(.4){L^{\theta^{\xi_\gamma}}\smallotimesL L} & \\
L \ar[ur]_(.4){\alpha}^(.6){L^{\theta^{\xi_\alpha}}\smallotimesL L} & & L \ar[ll]_{\beta}^{L^{\theta^{\xi_\beta}}\smallotimesL L}}$
&
$\xymatrix@C=1.2em@R=4.5em{
& L \ar@(lu,ru)_{s_1}^{L^\theta\smallotimesL L} \ar[dr]_(.6){\gamma}^(.4){L\smallotimesL L} & \\
L \ar[ur]_(.4){\alpha}^(.6){L^{\theta^{-\xi_\beta}}\smallotimesL L} & & L \ar[ll]_{\beta}^{L^{\theta^{\xi_\beta}}\smallotimesL L} 
}$
&
$\xymatrix@C=1.2em@R=4.5em{
& L \ar@(lu,ru)_{s_1}^{L\smallotimesL L} \ar[dr]_(.6){\gamma}^(.4){L^{\theta^{\xi_\gamma}}\smallotimesL L} & \\
L \ar[ur]_(.4){\alpha}^(.6){L^{\theta^{\xi_\alpha}}\smallotimesL L} & & L \ar[ll]_{\beta}^{L^{\theta^{\xi_\beta}}\smallotimesL L} 
}$
 &
 $\xymatrix@C=1.7em@R=4.5em{
 & L  \ar[dr]_(.6){\gamma}^(.4){L\smallotimesL  L} &  \\
 L \ar@(dr,dl)_{s_2}^{L^{\theta}\smallotimesL  L} \ar[ur]_(.4){\alpha}^(.6){L\smallotimesL L} & & L \ar@(dr,dl)_{s_3}^{L^\theta\smallotimesL L} \ar[ll]_{\beta}^{L\smallotimesL L} 
}$
 &
$\xymatrix@C=1.7em@R=4.5em{
 & L  \ar[dr]_(.6){\gamma}^(.4){L^{\theta^{\xi_\gamma}}\smallotimesL  L} &  \\
 L \ar@(dr,dl)_{s_2}^{L\smallotimesL  L} \ar[ur]_(.4){\alpha}^(.6){L^{\theta^{\xi_\alpha}}\smallotimesL L} & & L \ar@(dr,dl)_{s_3}^{L\smallotimesL L} \ar[ll]_{\beta}^{L^{\theta^{\xi_\beta}}\smallotimesL L} 
}$ 
\\
\hline
\end{tabular}}
       \captionof{table}{Triangulations and mnemotechnics for blocks 1 to 5 \label{table-surface-blocks-1-to-5}}
    \end{landscape}
    \clearpage
}

\afterpage{%
    \clearpage
    \thispagestyle{empty}
    \begin{landscape}
        \centering 
        \hspace{-1cm}
\renewcommand{\arraystretch}{1.25}
{\small
\begin{tabular}{|c|c|c|c|c|c|}
\hline
 & Block 6 & Block 7 & Block 8 & Block 9 & Block 10  \\
\hline
\begin{tabular}{c}
Weight triple\\ 
$\left(\begin{array}{ccc} & d_1 &\\ d_2 & & d_3\end{array}\right)$\end{tabular} 
& 
$\left(\begin{array}{ccc} & \circled{2} &\\ 4 & & 1\end{array}\right)$
&
$\left(\begin{array}{ccc} & \circled{2} &\\ 1 & & 4\end{array}\right)$
&
$\left(\begin{array}{ccc} & \circled{1} &\\ \circled{1} & & \circled{1}\end{array}\right)$ 
& 
$\left(\begin{array}{ccc} & 2 &\\ \circled{1} & & \circled{1}\end{array}\right)$ 
& 
$\left(\begin{array}{ccc} & \circled{1} &\\ 2 & & 2\end{array}\right)$ \\
\hline
\begin{tabular}{c}
Triangulation\\
$(\Sigma,\marked,\orb)$
\\{}\\{}
\end{tabular}
&
\begin{tabular}{c}
\includegraphics[trim={39.5cm 0 0 -1cm},clip,scale=.125]{Fig_unpunct_puzzle_pieces.png}
\end{tabular}
&
\begin{tabular}{c}
\includegraphics[trim={39.5cm 0 0 -1cm},clip,scale=.125]{Fig_unpunct_puzzle_pieces.png}
\end{tabular}
&
\begin{tabular}{c}
\includegraphics[trim={0 0 42.5cm 0},clip,scale=.125]{Fig_unpunct_puzzle_pieces.png}
\end{tabular}
&
\begin{tabular}{c}
\includegraphics[trim={19.5cm 0 22.5cm -0.5cm},clip,scale=.125]{Fig_unpunct_puzzle_pieces.png}
\end{tabular}
&
\begin{tabular}{c}
\includegraphics[trim={39.5cm 0 0 -1cm},clip,scale=.125]{Fig_unpunct_puzzle_pieces.png}
\end{tabular}
\\
\hline
\begin{tabular}{c}
Jacobian\\
algebra\\
mnemotechnics\end{tabular}
& 
$\xymatrix@C=0.7em@R=6em{
 & L  \ar[dr]_(.6){\gamma}^(.4){F\smallotimesF L} & \\
 E \ar[ur]_(.4){\alpha}^(.6){L^{\theta^{\xi_\alpha}}\smallotimesL E} & & F \ar[ll]_{\beta}^{E\smallotimesF F}
}$
&
$\xymatrix@C=0.7em@R=6em{
 & L  \ar[dr]_(.6){\gamma}^(.4){E^{\theta^{\xi_\gamma}}\smallotimesL L} & \\
 F \ar[ur]_(.4){\alpha}^(.6){L\smallotimesF F} & & E \ar[ll]_{\beta}^{F\smallotimesF E}
}$
&
$\xymatrix@C=0.7em@R=6em{
& F \ar[dr]_(.6){\gamma}^(.4){F\smallotimesF F} & \\
F \ar[ur]_(.4){\alpha}^(.6){F\smallotimesF F} & & F \ar[ll]_{\beta}^{F\smallotimesF F} 
}$ & $\xymatrix@C=0.7em@R=6em{
& L \ar[dr]_(.6){\gamma}^{F\smallotimesF L} & \\
F \ar[ur]_(.4){\alpha}^{L\smallotimesF F} & & F \ar[ll]_{\beta}^{F\smallotimesF F} 
}$ &  $\xymatrix@C=1.2em@R=2.5em{
 & & F  \ar[ddrr]_(.4){\gamma}^(.4){L\smallotimesF F} & & \\
 & & & & \\
 L \ar[uurr]_(.6){\alpha}^(.6){F\smallotimesF L} & & & & L \ar@/_1pc/[llll]^{\beta_0}_{L\smallotimesL L} \ar@/^1pc/[llll]_{\beta_1}^{L^\theta\smallotimesL L}
}$
\\
\hline
\begin{tabular}{c}
Semilinear\\
clannish\\
algebra\\
mnemotechnics\end{tabular}
& 
$\xymatrix@C=1.2em@R=4.5em{
 & L  \ar[dr]_(.6){\gamma}^(.4){L^{\theta^{-\xi_\alpha}}\smallotimesL  L} &  \\
 L \ar@(dr,dl)_{s_2}^{L\smallotimesL  L} \ar[ur]_(.4){\alpha}^(.6){L^{\theta^{\xi_\alpha}}\smallotimesL L} & & L \ar@(dr,dl)_{s_3}^{L^{\theta}\smallotimesL L} \ar[ll]_{\beta}^{L\smallotimesL L} 
}$
&
$\xymatrix@C=1.2em@R=4.5em{
 & L  \ar[dr]_(.6){\gamma}^(.4){L^{\theta^{\xi_\gamma}}\smallotimesL  L} &  \\
 L \ar@(dr,dl)_{s_2}^{L^\theta\smallotimesL  L} \ar[ur]_(.4){\alpha}^(.6){L^{\theta^{-\xi_\gamma}}\smallotimesL L} & & L \ar@(dr,dl)_{s_3}^{L\smallotimesL L} \ar[ll]_{\beta}^{L\smallotimesL L} 
}$
&
$\xymatrix@C=1.2em@R=4.5em{
& F \ar[dr]_(.6){\gamma}^(.4){F\smallotimesF F} & \\
F \ar[ur]_(.4){\alpha}^(.6){F\smallotimesF F} & & F \ar[ll]_{\beta}^{F\smallotimesF F} 
}$
 &
$\xymatrix@C=1.2em@R=4.5em{
& F \ar@(lu,ru)_{s_1}^{F\smallotimesF F} \ar[dr]_(.6){\gamma}^(.4){F\smallotimesF F} & \\
F \ar[ur]_(.4){\alpha}^(.6){F\smallotimesF F} & & F \ar[ll]_{\beta}^{F\smallotimesF F} 
}$
 &
$\xymatrix@C=1.2em@R=4.5em{
 & F  \ar[dr]_(.6){\gamma}^(.4){F\smallotimesF F} &  \\
 F \ar@(dr,dl)_{s_2}^{F\smallotimesF F} \ar[ur]_(.4){\alpha}^(.6){F\smallotimesF F} & & F \ar@(dr,dl)_{s_3}^{F\smallotimesF F} \ar[ll]_{\beta}^{F\smallotimesF F} 
}$ 
\\
\hline
\end{tabular}}
       \captionof{table}{Triangulations and mnemotechnics for blocks 6 to 10 \label{table-surface-blocks-6-to-10}}
    \end{landscape}
    \clearpage
}

The next result is completely in sync with the combinatorial block decompositions from \cite[\S13]{Fomin-Shapiro-Thurston} and \cite[\S3]{FeShTu-orbifolds}. The proof is straightforward; see Remark \ref{remark-on-gluing-procedures} above.

\begin{prop}\label{prop:our-algs-are-block-decomposable} All the Jacobian algebras introduced in Subsections \ref{subsubsec:arb-weights-Jac-alg} and \ref{subsubsec:constant-weights-Jac-alg}, 
as well as all the semilinear clannish algebras defined in Subsections \ref{subsubsec:arb-weights-semilin-clan-alg} and \ref{subsubsec:constant-weights-semilin-clan-alg}, are $F$-linearly isomorphic to algebras that can be obtained through the process just described.
\end{prop}

We have arrived at the main result of the paper.

\begin{thm}\label{thm:main-result-of-paper} Let $\SSigmaw=(\Sigma,\marked,\orb,\omega)$ be a surface with weighted orbifold points, either unpunctured or once-punctured closed, and let $(\tau,\xi)$ be a colored triangulation of $\SSigma$.
\begin{enumerate}
    \item For $\omega:\orb\rightarrow\{1,4\}$ arbitrary, the Jacobian algebra $\jacobalg{A(\tau,\xi),W(\tau,\xi)}$ defined in Subsection \ref{subsubsec:arb-weights-Jac-alg} and the semilinear clannish algebra $L_{\bldsigma(\tau,\xi)}\compactQhat (\tau)/I(\tau,\xi)$ defined in Subsection \ref{subsubsec:arb-weights-semilin-clan-alg} are Morita-equivalent through $F$-linear functors.
    \item If $\omega\equiv 4$, then the Jacobian algebra $\jacobalg{A}(\tau),W(\tau)$ defined in Subsection \ref{subsubsec:constant-weights-Jac-alg} and the clannish $F$-algebra $F_{\bldsigma(\tau)}\compactQhat (\tau)/I(\tau)$ defined in Subsection \ref{subsubsec:constant-weights-semilin-clan-alg} are isomorphic through an $F$-linear ring isomorphism.
\end{enumerate}
\end{thm}

\begin{proof}
   As noted in Remark \ref{remark-on-gluing-procedures}, $(\Sigma,\marked,\orb)$ may be obtained by gluing a suitable partial matching of outer sides of puzzle pieces. By Proposition \ref{prop:our-algs-are-block-decomposable}, both the Jacobian algebra and the semilinear clannish algebra can be obtained through a corresponding gluing of copies of Jacobian blocks, resp. copies of semilinear clannish blocks. Since both the Jacobian algebra and the semilinear clannish algebra are associated to the same (colored) triangulation, these two block decompositions can be consistently taken in such a way that there is a bijection between the set of copies of Jacobian blocks and the set of copies of semilinear clannish blocks, with the following two properties:
    \begin{itemize}
        \item[(a)] every time a copy of a Jacobian block corresponds to a copy of a semilinear clannish block under the bijection,  there is a $k=1,\ldots,10$ such that the Jacobian block copy lies on the $k^{\operatorname{th}}$ column of Tables \ref{table-Jacobian-blocks-1-to-5} and \ref{table-Jacobian-blocks-6-to-10}, and the semilinear clannish block copy lies on the $k^{\operatorname{th}}$ column of Tables \ref{table-semilinear-clannish-blocks-1-to-5} and \ref{table-semilinear-clannish-blocks-6-to-10}.
        \item[(b)] the bijection takes pairs of matched-and-glued Jacobian block copies (and corresponding outlets) to pairs of matched-and-glued semilinear clannish block copies (and corresponding outlets), and viceversa.
    \end{itemize}

    Thus, there are Morita equivalences between the blocks of the Jacobian algebra and the blocks of the semilinear clannish algebra by Propositions \ref{prop-isomorphisms-between-some-of-the-blocks} and \ref{prop-morita-equivalences-for-blocks}. 

    We have noticed above that there is never a loop based at an outlet, and that on all the outlets of the block copies chosen appears the same field. This, and the explicit definition of the Morita equivalences appearing in the proofs of Propositions \ref{prop-isomorphisms-between-some-of-the-blocks} and \ref{prop-morita-equivalences-for-blocks} (see Table \ref{table-Morita-equivalences-2467}), show that these Morita equivalences can be glued as well to produce a Morita equivalence between $\jacobalg{A(\tau,\xi),W(\tau,\xi)}$ and $L_{\bldsigma(\tau,\xi)}\compactQhat (\tau)/I(\tau,\xi)$. This proves the first statement.

    The second statement follows by the same reasoning, after noticing that in Proposition \ref{prop-isomorphisms-between-some-of-the-blocks} we have an $F$-linear isomorphism between the $k^{\operatorname{th}}$ Jacobian block and the $k^{\operatorname{th}}$ semilinear clannish block for $k=8,9,10$.
\end{proof}

\section{Indecomposable representations for blocks}\label{sec:indecs-for-blocks}

Indecomposable finite-dimensional modules over clannish algebras were classified by Crawley-Boevey in~\cite{CB-clans}. 
Such modules are either \emph{string modules}, defined by walks in the quiver, or \emph{band modules}, given by cyclic walks.  
The class of string modules and that of band modules each split into so-called \emph{asymmetric} and \emph{symmetric} subclasses. 
The symmetry is a  reflection of the relevant walk about a special loop. 
Crawley-Boevey's classification has been generalized to semilinear clannish algebras in \cite{BTCB}. 
We recall this result in Theorem \ref{main-theorem-from-BTCB}.

Recall Definitions \ref{def:types-of-clannish} and \ref{defi-semilinear-clannish-algebra}. 
Throughout all of Section \ref{sec:indecs-for-blocks} we fix a division ring $K$ and a semilinear clannish algebra $A=K_{\bldsigma}\compactQhat/I$ which is non-singular, normally-bound and of semisimple type. 

\subsection{
Asymmetric and symmetric strings and bands. 
}

\,

The main theorem in \cite{BTCB} gives a classification of the indecomposable modules, finite-dimensional over $K$, for $A$. 
These indecomposables,  \emph{strings} and \emph{bands},  are defined in  \cite[\S2.4--\S2.6,~\S3]{BTCB}. 
They are described in terms of certain \emph{words} in an alphabet defined by the arrows of the quiver $Q$ subject to the set $Z$ of zero-relations and the set $\bbS$ of special loops. 
Such words are defined explicitly in \cite[\S2.4]{BTCB}, where it is explained what it means for a string or band to be \emph{symmetric} or \emph{asymmetric}. 
The next result describes the words that occur for the semilinear clannish algebras we are considering from Tables \ref{table-semilinear-clannish-blocks-1-to-5} and \ref{table-semilinear-clannish-blocks-6-to-10}.

\begin{prop}
\label{strings-and-bands-for-building-blocks} 

The strings for a  semilinear clannish block from Tables \ref{table-semilinear-clannish-blocks-1-to-5} and \ref{table-semilinear-clannish-blocks-6-to-10} are given by Table \ref{table-semilinear-clannish-blocks-1-to-10-rep-type}.

\begin{center}
\renewcommand{\arraystretch}{1.25}
{\small
\begin{tabular}{|c|c|c|c|}
\hline
$w$ & Blocks 1 and 8 & Blocks 2, 3 and 9 & Blocks 4, 5, 6, 7 and 10  \\
\hline
\begin{tabular}{c}
Ordinary quiver\\
$\compactQhat $
\end{tabular}&
$\xymatrix{
& 1 \ar[dr]^{\gamma} & \\
2 \ar[ur]^{\alpha} & & 3, \ar[ll]^{\beta} 
}$
&
$\xymatrix{
 & 1 \ar@(lu,ru)^{s_1} \ar[dr]^{\gamma} &\\
 2 \ar[ur]^{\alpha} & & 3\ar[ll]^{\beta}
}$
&
$\xymatrix{
 & 1  \ar[dr]^{\gamma} &  \\
 2 \ar@(u,l)_{s_2} \ar[ur]^{\alpha} & & 3 \ar@(u,r)^{s_3} \ar[ll]^{\beta} 
}$
\\
\hline
\begin{tabular}{l}
Equivalence classes of \\
asymmetric strings
\end{tabular} 
&
\begin{tabular}{l}
$1_{1}$ \\ $1_{2}$ \\ $1_{3}$ \\ $\alpha$ 
\\ $\beta$ \\ $\gamma$
\end{tabular} 
&
\begin{tabular}{l}
$1_{2}$ \\ $1_{3}$ \\ $\beta$ \\ $s_{1}^{*}\alpha$
\\ $\gamma s_{1}^{*}$ \\ $\gamma s_{1}^{*}\alpha$
\end{tabular} 
&
\begin{tabular}{ll}
\begin{tabular}{l}
$1_{1}$\\
$(s_{3}^{*}\beta^{-1}s_{2}^{*}\beta)^{n}s_{3}^{*}\gamma$ \\
$\alpha s_{2}^{*}(\beta s_{3}^{*}\beta^{-1}s_{2}^{*})^{n}$\\ 
$s_{2}^{*}(\beta s_{3}^{*}\beta^{-1}s_{2}^{*})^{n}\beta s_{3}^{*}$ \\
$s_{2}^{*}(\beta s_{3}^{*}\beta^{-1}s_{2}^{*})^{n}\beta s_{3}^{*}\gamma$ \\
$\alpha s_{2}^{*}(\beta s_{3}^{*}\beta^{-1}s_{2}^{*})^{n}\beta s_{3}^{*}$ \\
$\alpha s_{2}^{*}(\beta s_{3}^{*}\beta^{-1}s_{2}^{*})^{n}\beta s_{3}^{*}\gamma$
\end{tabular} 
&
$(n\geq 0)$
\end{tabular} \\
\hline
\begin{tabular}{l}
Equivalence classes of \\
symmetric strings
\end{tabular} 
&
None.
&
\begin{tabular}{l}
$s_{1}^{*}$ \\ $\alpha^{-1}s_{1}^{*}\alpha$ \\  $\gamma s_{1}^{*} \gamma^{-1}$
\end{tabular} 
&
\begin{tabular}{ll}
\begin{tabular}{l}
$s_{2}^{*}(\beta s_{3}^{*} \beta^{-1} s_{2}^{*})^{n}$ \\
$s_{3}^{*}(\beta^{-1}s_{2}^{*}\beta s_{3}^{*})^{n}$ \\
$\alpha s_{2}^{*}(\beta s_{3}^{*} \beta^{-1} s_{2}^{*})^{n}\alpha^{-1}$ \\ 
$\gamma^{-1}s_{3}^{*}(\beta^{-1}s_{2}^{*}\beta s_{3}^{*})^{n}\gamma$
\end{tabular}
$ (n\geq 0)$
\end{tabular}
\\
\hline
\end{tabular}}
 \captionof{table}{Strings for semilinear clannish blocks 1 to 10 \label{table-semilinear-clannish-blocks-1-to-10-rep-type}}
 \vspace{3mm}
\end{center}

 Furthermore, 
 for each of the blocks there are no asymmetric bands, and the following statements hold.
\begin{enumerate}
    \item For blocks 1, 2, 3, 8 and 9 there are no symmetric bands.
    \item For blocks 4, 5, 6, 7 and 10 every symmetric band is equivalent to 
    \[
    {}^{\infty}(\beta s_{3}^{*}\beta^{-1}s_{2}^{*})^{\infty}=\dots \beta s_{3}^{*}\beta^{-1}s_{2}^{*} \mid \beta s_{3}^{*}\beta^{-1}s_{2}^{*} \beta s_{3}^{*}\beta^{-1}s_{2}^{*} \dots
    \]
\end{enumerate}
\end{prop}
\begin{proof}
From the choice of $Q$, $\bbS$ and $Z$ used in defining $A$ it follows that the words above constitute a complete list of the strings and bands for $A$. 
To see this, note that any word which is \emph{relation}-\emph{admissible} and \emph{end}-\emph{admissible} must be a sequence that alternates between an ordinary arrow (which here is one of the arrows in the $3$-cycle) and any special loop (so any of the loops). 
Moreover, if a word ends on a vertex at which there is a special loop $s$, the last letter of the word must end with $s^{*}$.

To see that the words above are pairwise non-equivalent, consider cases. 
For strings, note that distinct strings $w$ and $w'$ are equivalent provided $w'$ is found by inverting the letters of $w$ and reversing their order, where letters of the form $s^*$ with $s\in \bbS$ are self-inverse. 
For bands, note that there can only be four distinct infinite words, all of which are \emph{shifts} of one another, meaning they must be equivalent. 
\end{proof}

\subsection{Modules over semilinear clannish algebras}
\label{subsec-modules-for-clannish}

\,

For each string or band $w$ one defines a ring $R_{w}$ and  a $A$-$R_{w}$-bimodule $M(C_{w})$. 
The ring $R_{w}$ is one of four \emph{parameterising} rings, depending on whether $w$ is a string or a band, and depending on whether $w$ is symmetric or asymmetric. 

\begin{defi}
Let $w$ be a word which is either a string or a band, and either symmetric or asymmetric. 
    \begin{enumerate}[(a)]
        \item \cite[\S3.1]{BTCB} If $w$ is an \emph{asymmetric string} then $R_{w}=K$. 
        \item \cite[\S3.2]{BTCB} If $w$ is a \emph{symmetric} string then  $R_{w}=K[x;\nu]/\langle q(x)\rangle$ for an automorphism $\nu$  of $K$ and a monic, normally-bound, non-singular quadratic skew-polynomial $q(x)\in K[x;\nu]$  of semisimple type. 
         \item \cite[\S3.3]{BTCB} If $w$ is an \emph{asymmetric band} then $R_{w}=K[x,x^{-1};\sigma]$ for an automorphism $\sigma$ of $K$. 
        \item \cite[\S3.4]{BTCB} If $w$ is a \emph{symmetric band} then $R_{w}=K[x;\rho]/\langle r(x)\rangle*_{K}K[y;\tau]/\langle p(y)\rangle$ for automorphisms $\rho,\tau$ of $K$ and monic, normally-bound, non-singular quadratics $r(x),p(y)$  of semisimple type.
    \end{enumerate}
\end{defi}

\begin{remark}
\label{remark-nonsemilinear-symmetric-bands}
Recall that for the semilinear clannish blocks $1,\ldots,7$, we take $K=L$, and that the automorphisms $\sigma_{x}$ assigned to each arrow $x$ come from the Galois group $\Gal(L/F)=\{\myid_{L},\theta\}$. 
Suppose now $w$ is a symmetric  band, and so $R_{w}=L[x;\rho]/\langle r(x)\rangle*_{L}L[y;\tau]/\langle p(y)\rangle $ by (d) above. 
Since $L$ is a  finite-dimensional field extension of the field fixed by each of the automorphisms $\sigma_{x}$, namely $F$. 
This means $R_{w}$ is a \emph{classical hereditary order}. 
So, in principle, finite-length $R_{w}$-modules  are well understood; see \cite[Theorem~3.8,~Remark~3.9]{BTCB} and, for example, the survey of modules over classical hereditary orders by Levy \cite{Levy}. Let us consider two examples:
\begin{enumerate}[(i)]
    \item Let $\rho=\tau=\myid_{L}$,  $r(x)=x^{2}-u$ and $p(y)=y^{2}-u$ for $u$ as in the degree-$2$ situation. 
    It follows by results of Cohn  \cite[Theorem 3.5]{Cohn-free-prod-associative-II}, \cite[Lemma 2]{Cohn-quadratic} that  $L[x]/\langle x^{2}-u\rangle *_{L} L[y]/\langle y^{2}-u\rangle$, which is a free product (over $L$) of two isomorphic copies of the field extension $E$, is a left and right principal ideal domain.
    \item Let $\rho=\theta$, $\tau=\myid_{L}$,  $r(x)=x^{2}-1$ and $p(y)=y^{2}-u$ for $u$ as above and $\theta$ as in the degree-$2$ situation. Rings of the form $L[x;\rho]/\langle x^{2}-1\rangle *_{L} L[y]/\langle y^{2}-u\rangle$, and finite-dimensional modules over them, have been considered and studied before by Smits \cite{Smits}.
\end{enumerate}
\end{remark}

For the definition of the $A$-$R_{w}$-bimodule $M(C_{w})$, we refer the reader to \cite{BTCB}.
\begin{ex}
\label{ex-semilinear-representation-running-example}
Recall the $6^{\operatorname{th}}$ semilinear clannish block from Table \ref{table-semilinear-clannish-blocks-6-to-10}.
Throughout this example we simplify notation, letting $\xi\coloneqq\xi_{\alpha}$.  
Here we take $K=L$, take $\compactQhat$ to be the quiver depicted below on the left, and take $\bldsigma\colon \compactQhat_{1}\to \Aut(L)$ to be the function whose image is depicted below on the right
    \[
    \begin{array}{ccc}
\xymatrix@C=3em@R=1.5em{
 & 1  \ar[dr]^{\gamma} &  \\
 2 \ar@(u,l)_{s_2}\ar[ur]^{\alpha} & & 3 \ar@(u,r)^{s_3}\ar[ll]^{\beta}
} & \hspace{10mm}
&
\xymatrix@C=3em@R=1.5em{
 & L  \ar[dr]^{\theta^{-\xi}} &  \\
 L \ar@(u,l)_{\myid_{L}}\ar[ur]^{\theta^{\xi}} & & L \ar@(u,r)^{\theta}\ar[ll]^{\myid_{L}}
}    \end{array}
    \]
For block 6 we are additionally taking $Z=\{\alpha\beta,\beta\gamma,\gamma\alpha\}$, $\bbS=\{s_{2},s_{3}\}$, $q_{s_{2}}(x)=x^{2}-u\in L[x]$ and $q_{s_{3}}(x)=x^{2}-1\in L[x;\theta]$. 
Recall, from Example \ref{ex-main-example-for-semilinear-clannish-blocks}, that the quotient $A=L_{\bldsigma}\compactQhat/\langle Z\cup S\rangle$ is a semilinear clannish algebra which is normally bound, non-singular and of semisimple type.

We aim to give an example of a symmetric string module $M(C_{w})\otimes_{R_{w}} E$, following the notation from  \cite[\S 2.6,~\S3.2]{BTCB}. 
Let $
w=\gamma^{-1} s_{3}^{*}\beta^{-1} s_{2}^{*} \beta s_{3}^{*}\gamma
$, which is a  symmetric string with  $R_{w}=L[x]/\langle x^{2}-u\rangle$. 
Let $V=E$, an $R_{w}$-module where $x$ acts by multiplication with $v$, with $L$-basis $\{1,v\}$. 
One then defines the symmetric string module above by constructing the $A$-$R_{w}$-bimodule $M(C_{w})$. 
This construction is done in such a way that the following properties are  satisfied.
\begin{itemize}
    \item Considered as a left $L$-vector space $M(C_{w})$ has basis $\{b_{0},\dots,b_{7}\}$, and for any $\ell\in L$ we have 
\[
\begin{array}{ccc}
    b_{i}\ell=\ell b_{i} \quad (i=0,7), &  
    b_{j}\ell=\theta^{-\xi}(\ell) b_{j} \quad (j=1,6), &
b_{k}\ell=\theta^{1-\xi}(\ell) b_{k} \quad (k=2,3,4,5).
\end{array}
\]
\item The  $L$-ring generators of $A=L_{\bldsigma}\compactQhat/I$ and $R_{w}=L[x]/\langle x^{2}-u\rangle $ act according to  the diagram
\[
\xymatrix@C=4em@R=2em{
b_{0}\ar@/_1.75pc/[rrrrrrr]_{x}\ar[r]^{\gamma} & 
b_{1}\ar@/_1.5pc/[rrrrr]\ar[r]^{s_{3}} & 
b_{2}\ar@/_1.25pc/[rrr]\ar[r]^{\beta} &
b_{3}\ar@/_1pc/[r]\ar[r]^{s_{2}} & 
b_{4} &
b_{5}\ar[l]_{\beta} &
b_{6}\ar[l]_{s_{3}} & 
b_{7}\ar[l]_{\gamma}
}
\]
\end{itemize}
Next we describe the semilinear representation $N$ corresponding to $M(C_{w})\otimes_{R_{w}}V$, the symmetric string module  described above. 
By identifying such semilinear representations with representations of the species $(\compactQhat, \bldsigma)$ annihilated by the relations $Z\cup S$, one can consider $N$ as the image of the equivalence $\Omega$ recalled in \S\ref{subsubsec-representations-and-modules}.

Consider the right $L$-action on $M(C_{w})$ discussed in the first item above. 
As in \cite{BTCB}, for each $i=0,\dots,7$ we identify the left $L$-vector space $b_{i}\otimes V$ with a twisted copy of $L$. 
The $\bldsigma$-semilinear representation $N$ can hence  be described as follows. For any $\sigma\in\Aut(L)$ we write ${}^{\sigma}E$ to denote the $L$-vector space ${{}^{\sigma}L}\oplus {{}^{\sigma}L}$. 
\begin{enumerate}[(i)]
    \item $N_{1}= E$, $N_{2}=  {}^{\theta^{\xi-1}}E$ and $N_{3}= {}^{\theta^{\xi}}E\oplus{}^{\theta^{\xi-1}}E$ are considered as $L$-vector spaces.  
    \item $N_{\alpha}\colon N_{2}\to N_{1}$ is the zero map, which is $\theta^{\xi}$-semilinear.
    \item $N_{\beta}\colon N_{3}\to N_{2}$ is the projection onto the right-hand component, which is $\myid_{L}$-semilinear.
    \item $N_{\gamma}\colon N_{1}\to N_{3}$ is the canonical embedding into the left-hand component, which is  $\theta^{-\xi}$-semilinear.
    \item $N_{s_{2}}\colon N_{2}\to N_{2}$ is given by multiplication by $v$, which is $\myid_{L}$-semilinear, and $N_{s_{2}}^{2}=u\myid_{N_{2}}$. 
    \begin{itemize}
        \item Writing $N_{2}={}^{\theta^{\xi-1}}L^{2}=\{(\lambda,\mu)\mid\lambda,\mu\in L\}$, any $\ell\in L$ acts by $\ell(\lambda,\mu)=(\theta^{\xi-1}(\ell)\lambda,\theta^{\xi-1}(\ell)\mu)$. 
        \item The map $N_{s_{2}}$ is defined by the assignment  $(\lambda,\mu)\mapsto (\theta^{\xi-1}(u)\mu,\lambda)$, and so $N^{2}_{s_{2}}(\lambda,\mu)=u(\lambda,\mu)$. 
        \item Corresponding to  $N_{s_{2}}$ is an  anti-diagonal matrix in $M_{2}(L)$ with non-zero entries $1,\theta^{\xi-1}(u)$.
        \end{itemize}
\item $N_{s_{3}}\colon N_{3}\to N_{3}$ is given by swapping the entries, which is $\theta$-semilinear since $\theta^{2}=\myid_{L}$, and $N_{s_{3}}^{2}=\myid_{N_{3}}$.
\end{enumerate}
The $L$-vector spaces $N_{i}$ and $\sigma_{b}$-semilinear maps $N_{b}$ thus define a $\bldsigma$-semilinear representation $N$ of $\compactQhat$. 
\end{ex}

We now recall the main result in \cite{BTCB}.

\begin{thm}
\label{main-theorem-from-BTCB}
The modules $M(C_{w})\otimes_{R_{w}} V$ run through a complete set of pairwise non-isomorphic indecomposable $A$-modules which are finite-dimensional over $K$, where:
\begin{enumerate}
    \item $w$ runs through representatives of distinct equivalence classes of strings and bands; and 
    \item for each fixed string or band $w$, $V$ runs through representatives of distinct isomorphism classes of indecomposable $R_{w}$-modules which are finite-dimensional over $K$. 
\end{enumerate}
\end{thm}

Hence the classification of  finite-dimensional indecomposable modules over a semilinear clannish algebra are parameterised by  the finite-dimensional indecomposable $R_{w}$-modules $V$, as  $w$ runs through the equivalence classes of strings and bands.

By the verification of condition (iii) in the proof of Proposition \ref{prop-algebras-are-semilinear-clannish}, if $w$ is a string then the ring $R_{w}$ is simple artinian. 
Otherwise, by Proposition \ref{strings-and-bands-for-building-blocks}, $w$ is a symmetric band. 
In this case, by definition and Remark \ref{remark-nonsemilinear-symmetric-bands}, the finite-dimensional $R_{w}$-modules are well understood.

\subsection{Passing representations through the established Morita equivalence}
\label{subsec-passing-reps-through-the-equivalence}

\,

We aim at illustrating how the equivalence given in Theorem \ref{main-theorem} works for modules over the Jacobian and semilinear clannish blocks 6 from Tables \ref{table-Jacobian-blocks-6-to-10} and  \ref{table-semilinear-clannish-blocks-6-to-10}. 
In Example \ref{ex-semilinear-representation-running-example} we described a representation for the semilinear clannish block, chosen using Theorem \ref{main-theorem-from-BTCB}. 
We now pass the representation through the Morita equivalence exhibited in Proposition \ref{prop-morita-equivalences-for-blocks}. 

\begin{ex}
Recall the $6^{\operatorname{th}}$ Jacobian block is defined in Table \ref{table-Jacobian-blocks-6-to-10} as follows. 
As we did in Example \ref{ex-semilinear-representation-running-example}, in this example we let $\xi\coloneqq\xi_{\alpha}$. 
Let $(Q,\bldD)$ be the weighted quiver with $d\coloneqq \lcm(2,4,1)=4$, given by
    \[
\begin{array}{ccc}
\xymatrix@C=3em@R=.75em{
& 1 \ar[dr]^{\gamma} & \\
2 \ar[ur]^{\alpha} & & 3 \ar[ll]^{\beta} 
}
&
\hspace{10mm}
&
\xymatrix@C=3em@R=0.75em{ 
& d_1=2 & \\
d_2=4  & & d_3=1
}
\end{array}
\]
For block 6 we also have $E/L/F$, a degree-$4$ datum for $(Q,\bldD)$, recalled as follows.
\begin{itemize}
    \item $E/L$ and  $L/F$ are degree-$2$ cyclic Galois extensions with $E=L(v)$ and $L=F(u)$.
    \item $\zeta\in F$ is a primitive $4^{\operatorname{th}}$ root of unity and $v^{2}=u\in L$. 
    \item $\theta\in \Gal(L/F)$ and $\rho\in \Gal(E/L)$ are  generators   of the respective Galois groups.
\end{itemize}
For the Jacobian algebra we fix  $\xi\in\ZZ/2\ZZ$ and define a modulating function for $(Q,\bldD)$ and $E/L/F$ by $\alpha\mapsto \theta^{\xi}$ and $\beta,\gamma\mapsto\myid_{L}$. Thus $A(\inputtripleforblocks)$ is the $F$-modulation of $(Q,\bldD)$ given by 
\[
\begin{array}{cc}
(F_{1},F_{2},F_{3})=(L,E,F), & (A_{\alpha},A_{\beta},A_{\gamma})=(L^{\theta^{\xi}}\otimes_{L}E,E\otimes_{F}F ,F\otimes_{F}L).
\end{array}
\]
Moreover, $W(\inputtripleforblocks)=\alpha\beta\gamma$ is a potential and $\rho=\{\frac{1}{2} (\beta\gamma + 
\theta^{-\xi}(u^{-1})\beta\gamma u), \gamma\alpha, \alpha\beta\}$, the set of derivatives, defines the relations for the quotient  $A'=\jacobalg{\inputtripleforblocks)}\cong\usualRA{\inputtripleforblocks}/\langle\rho\rangle$, the  Jacobian algebra. 
Now consider the image  $M=\Phi(N)$ of the module $N$ for the semilinear clannish block $A$ exhibited in Example \ref{ex-semilinear-representation-running-example},  under the equivalence from Proposition \ref{prop-morita-equivalences-for-blocks}. We use notation from Proposition \ref{prop-morita-equivalences-for-blocks} and Example \ref{ex-semilinear-representation-running-example}. 
\begin{enumerate}[(i)]
    \item $M_{1}=E$, which, just as for  $N_{1}=E$, is considered as an $L$-vector space.
        \item $M_{2}$ is the $E=L(v)$-vector space ${}^{\theta^{\xi-1}}L^{2}$  where the $E$-action $E\times {}^{\theta^{\xi-1}}L^{2}\to {}^{\theta^{\xi-1}}L^{2}$ is defined by 
        \[
\begin{array}{c}
       (\ell +\ell' v ,(\lambda,\mu))\mapsto (\theta^{\xi-1}(\ell)\lambda+\theta^{\xi-1}(\ell' u)\mu,\theta^{\xi-1}(\ell)\mu+\theta^{\xi-1}(\ell')\lambda)
    \end{array}
        \]
        where $\ell+\ell'v\in E$ for unique $\ell,\ell'\in L$, and where $\lambda,\mu\in L$.
    \item $M_{3}=E^{\Delta}=\{(e,e)\mid e\in E\}$ which is an $F$-subspace of $N_{3}={}^{\theta^{\xi}}E\oplus{}^{\theta^{\xi-1}}E$, but  not an $L$-subspace. 
    \item $M_{\alpha}\colon M_{2}\to M_{1}$ is the zero map, which is again considered $\theta^{\xi}$-semilinear.
    \item $M_{\beta}\colon M_{3}\to M_{2}$ sends $(\lambda + \mu v,\lambda + \mu v)\in {}_{F}E^{\Delta}$ to $(\lambda,\mu)\in {}_{E}({}^{\theta^{\xi-1}}L^{2})$ which is $F$-linear.
    \item $M_{\gamma}\colon M_{1}\to M_{3}$ sends $(\ell,\ell')\in {}_{L}E=L^{2}$ to $\frac{1}{2}(\ell+\ell'v,\ell+\ell'v)\in {}_{F}E^{\Delta}$ which is $F$-linear.
\end{enumerate}
We have defined $F_{i}$-vector spaces $M_{i}$ ($i\in Q_{0}$) and $g(a)$-semilinear maps $M_{a}\colon M_{t(a)}\to M_{h(a)}$ ($a\in Q_{1}$), which combine to define a representation $M$ of $A(\inputtripleforblocks)$ given by the image of $N$ under the functor $\Phi$.

We now compute the module $Y$ over the Jacobian algebra $\jacobalg{\inputtripleforblocks}$ that corresponds to the representation $M$ under the funtor $\Gamma$ in \S\ref{subsubsec-representations-and-modules}. 
Begin by considering  
\[
Y={}_{L}E\oplus {}_{E}({}^{\theta^{\xi-1}}L^{2})\oplus {}_{F}E^{\Delta}
\]
as a module over the direct product $L\times E\times F$, via the diagonal action. 
Consider now the $F$-linear map $Y_{\gamma}\colon A_{\gamma}\otimes_{R_{1}} M_{1}\to M_{3}$, and the $E$-linear map $Y_{\beta}\colon A_{\beta}\otimes_{R_{3}} M_{3}\to M_{2}$, respectively defined by
\[
\begin{array}{cc}
Y_{\gamma}\colon F\otimes_{F}L\otimes_{L}E\to {}_{F}E^{\Delta},
   & 
   Y_{\beta}\colon E\otimes_{F}F\otimes_{F}E^{\Delta}\to {}_{E}({}^{\theta^{\xi-1}}L^{2}),
   \\
  1\otimes \lambda\otimes (\ell,\ell')\mapsto \frac{1}{2}(\lambda\ell+\lambda\ell'v,\lambda\ell+\lambda\ell'v),
     & 1\otimes f\otimes (\lambda+\mu v,\lambda+\mu v)\mapsto (f\lambda,f\mu).
\end{array}
\]
The left action $\jacobalg{\inputtripleforblocks}\times Y\to Y$ is given by canonically extending the maps $Y_{\gamma}$ and $Y_{\beta}$. For example, 
the action of the path 
$v^{3}\beta\gamma u\in\compRA{\curlS(Q,\bldD,\xi)}$ on any element $(\ell,\ell')\in  {}{}_{L}E$ is given by 
\[
\begin{array}{c}
{}_{E}({}^{\theta^{\xi-1}}L^{2})\ni v^{3}\beta\gamma u.(\ell,\ell')=
(0+uv)(M_{\beta}\circ M_{\gamma}(u\ell,u\ell'))=
\frac{1}{2}(u^{3}\ell',\theta^{\xi-1}(u)u\ell).
\end{array}
\]

\end{ex}

\section{On the need of cocycles to define the algebras}\label{sec:on-need-of-cocycles}

The constructions in Subsection \ref{subsec:algebras-for-arb-weights} have a colored triangulation $(\tau,\xi)$, that is, a triangulation $\tau$ and a $1$-cocycle $\xi$, as an input. As mentioned in Remark \ref{rem:cocycle-condition}, choosing a $1$-cocycle $\xi = \sum_{{{\alpha}}} \xi_{{\alpha}} {{\alpha}}^\vee \in \CZonetauwFtwo\subseteq C^1(\tau)\coloneqq \Hom_{\F_2}(C_1(\tau),\F_2)$ amounts to fixing, for each arrow \smash{${{\alpha}} \in \overline{Q}(\tau)_1$}, an element $\xi_{{\alpha}} \in \{0,1\} = \F_2$, in such a way that for every interior triangle $\triangle$ of $\tau$, the three arrows ${{\alpha}}, {{\beta}}, {{\gamma}}$ of $\overline{Q}(\tau)$ induced by $\triangle$ satisfy
 \[
  \xi_{{\alpha}} + \xi_{{\beta}} + \xi_{{\gamma}} = 0 \:\in\: \F_2 \,,
 \]
 which translates into the equation
 \begin{equation}\label{eq:cocycle-cond-inside-Gal-group}
  \theta^{\xi_{{\alpha}}} \theta^{\xi_{{\beta}}} \theta^{\xi_{{\gamma}}} = \myid \:\in\: \Gal(L/F) \,.
 \end{equation}
In \cite[Examples 3.12, 3.13, 3.25]{GLF1} and \cite[Example 4.1 and Section 11]{GLF2}, it is shown that if \eqref{eq:cocycle-cond-inside-Gal-group} failed to hold for some interior triangle $\triangle$ such that
\begin{equation}\label{eq:min-geq-2}
\min\{d(\tau,\omega)_{h(\alpha)},d(\tau,\omega)_{h(\beta)},d(\tau,\omega)_{h(\gamma)}\}\geq 2
\end{equation}
(notice that $h(\alpha),h(\beta)$ and $h(\gamma)$ are the three arcs of $\tau$ contained in $\triangle$),
then any oriented $3$-cycle involving the arrows $\alpha$, $\beta$ and $\gamma$ would be cyclically equivalent to zero in the corresponding (complete) path algebra. Since our cyclic derivatives are defined so that cyclically equivalent potentials have the same cyclic derivatives  (see \cite[Definition 3.11]{GLF1}), this ultimately implies that such (complete) path algebra fails to admit a non-degenerate potential (in the sense of Derksen-Weyman-Zelevinsky \cite[Definition 7.2]{DWZ1}) under the mutation of species with potential from \cite[Definitions 3.19 and 3.22]{GLF1}.

Thus, the imposition of a cocycle condition as part of the input $(\tau,\xi)$ comes from the desire to give the (complete) path algebra of the associated species $A(\tau,\xi)$ a chance to admit a non-degenerate potential. 

Now, why do we need to allow arbitrary $1$-cocycles?, why do we not simply work with the zero cocycle? These questions are fully answered in \cite[Proposition 2.12 and the paragraph that follows it]{GLF1} and \cite[Case~1 in the proof of Theorem 7.1]{GLF2}. Roughly speaking, the main point is that when one decomposes the tensor product of bimodules as a direct sum of indecomposable bimodules, one is forced to consider twisted bimodules such as $\mathbb{C}^\theta\otimes_{\mathbb{C}}\mathbb{C}$ even if the tensor factors are not twisted by field automorphisms  (notice that $\mathbb{C}$ does not act centrally on this bimodule: one needs to apply complex conjugation in order to move complex numbers through the tensor symbol).

Thus, if one wants the first step of the purely combinatorial weighted quiver mutation (i.e., the introduction of ``composite arrows'', see, e.g., \cite{LF-Oberwolfach-2013} or \cite[\S2]{LZ}) to be categorified as taking the tensor product of bimodules, one is forced to allow non-trivial $1$-cocycles.

Summarizing, the need to work with $1$-cocycles stems from the phenomena that arise in the categorification of mutations of weighted quivers via mutations of species with potential.

\begin{remark} 
What happens in Subsection \ref{subsec:constant-weights-algs-def-over-C/R} is that for a degree-$2$ datum $L/F$, the set of bimodules
$$
\{F\otimes_FF,F\otimes_FL,L\otimes_FF,L\otimes_FL\}
$$
is closed under tensor products, and the bimodule $L\otimes_FL$ (on which $L$ does not act centrally, see also Remark \ref{rem:decomposing-Lotimes_FL}) always connects pending arcs. So, under the setting in \eqref{eq:omega=4=>half-the-weights} and the paragraphs that follow it, one can coincidentally avoid bimodules with non-trivial twists, even when one is interested in mutations of species with potential. 
\end{remark}

On the other hand, the definitions and results from \cite{BTCB} on semilinear clannish algebras do not require the field automorphisms $\sigma_a$ attached to the arrows $a$ that are not special loops, to satisfy any particular identity. This means that we can modify the constructions in Subsections \ref{subsec:semilinear-clannish-blocks}, \ref{subsubsec:arb-weights-semilin-clan-alg} and \ref{subsubsec:constant-weights-semilin-clan-alg} by attaching to each non-loop $a:j\rightarrow k$ any field automorphism $\sigma_a:K\rightarrow K$, but keeping the automorphisms attached to the loops, as well as the definition of the relations intact, and still obtain a semilinear clannish algebra, most likely not isomorphic or even Morita equivalent to the ones we defined. (Recall that $K\coloneqq L$ in \ref{subsubsec:arb-weights-semilin-clan-alg} and in the first seven columns of Tables \ref{table-semilinear-clannish-blocks-1-to-5} and \ref{table-semilinear-clannish-blocks-6-to-10}, whereas $K\coloneqq F$ in \ref{subsubsec:constant-weights-semilin-clan-alg} and in the last three tables of the referred tables).

\section*{Acknowledgements}
    We are deeply grateful to Bill Crawley-Boevey,
    this paper originated in discussions with him. We thank the anonymous referee for many useful suggestions.

    The first author is grateful to have been supported during part of this work by: the Alexander von Humboldt Foundation in the framework of an Alexander von Humboldt Professorship endowed by the German Federal Ministry of Education; the Danish National Research Foundation (grant DNRF156); the Independent Research Fund Denmark (grant 1026-00050B); and the Aarhus University Research Foundation (grant AUFF-F-2020-7-16).

    The second author received support from UNAM's \emph{Direcci\'on General de Asuntos del Personal Acad\'emico} through its \emph{Programa de Apoyos para la Superaci\'on del Personal Acad\'emico}. Most of the paper was written during a visit of his to Sibylle Schroll at the Mathematisches Institut der Universit\"at zu K\"oln. The great working conditions and the hospitality are gratefully acknowledged.

\end{document}